\documentclass[reqno, 11pt, letterpaper]{amsart}

\usepackage{amsmath}
\usepackage{amsfonts}
\usepackage{amssymb}
\usepackage{graphicx}
\usepackage{amsthm,graphicx,color,yfonts}
\usepackage{pgfplots}
\pgfplotsset{compat=1.10}
\usepgfplotslibrary{fillbetween}
\usepackage{hyperref}

\newtheorem{theorem}{Theorem}
\newtheorem{definition}[theorem]{Definition}
\newtheorem{proposition}[theorem]{Proposition}
\newtheorem{lemma}[theorem]{Lemma}
\newtheorem{corollary}[theorem]{Corollary}

\theoremstyle{remark}
\newtheorem{remark}[theorem]{Remark}

\newcommand{\ls}{\lesssim}
\newcommand{\gs}{\gtrsim}
\newcommand{\la}{\langle}
\newcommand{\ra}{\rangle}
\newcommand{\R}{\mathbb{R}}
\newcommand{\T}{\mathbb{T}}

\newcommand{\Z}{\mathbb{Z}}
\newcommand{\pa}{\partial}

\def\norm#1{\left\|#1\right\|}

\usepackage{wrapfig}
\usepackage{tikz, pgfplots}
\usetikzlibrary{arrows,calc,decorations.pathreplacing}
\definecolor{light-gray1}{gray}{0.90}
\definecolor{light-gray2}{gray}{0.80}
\definecolor{light-gray3}{gray}{0.60}

\numberwithin{equation}{section}

\numberwithin{theorem}{section}

\numberwithin{table}{section}

\numberwithin{figure}{section}

\ifx\pdfoutput\undefined
  \DeclareGraphicsExtensions{.pstex, .eps}
\else
  \ifx\pdfoutput\relax
    \DeclareGraphicsExtensions{.pstex, .eps}
  \else
    \ifnum\pdfoutput>0
      \DeclareGraphicsExtensions{.pdf}
    \else
      \DeclareGraphicsExtensions{.pstex, .eps}
    \fi
  \fi
\fi

\title[KdV limit for FPU system]{Korteweg--de Vries limit for the Fermi--Pasta--Ulam system}

\author[Y. Hong]{Younghun Hong}
\address{Department of Mathematics, Chung-Ang University, Seoul 06974, Korea}
\email{yhhong@cau.ac.kr}

\author[C. Kwak]{Chulkwang Kwak}
\address{Facultad de Matem\'aticas, Pontificia Universidad Cat\'olica de Chile and Institute of Pure and Applied Mathematics, Jeonbuk National University}
\email{chkwak@mat.uc.cl}

\author[C. Yang]{Changhun Yang}
\address{Korea Institute for Advanced Study, Seoul 20455 and Institute of Pure and Applied Mathematics, Jeonbuk National University, Jeonju 54896, Korea}
\email{maticionych@kias.re.kr}

\begin{document}

\maketitle


\begin{abstract}
In this paper, we develop dispersive PDE techniques for the Fermi--Pasta--Ulam (FPU) system with infinitely many oscillators, and we show that general solutions to the infinite FPU system can be approximated by counter-propagating waves governed by the Korteweg--de Vries (KdV) equation as the lattice spacing approaches zero. Our result not only simplifies the hypotheses but also reduces the regularity requirement in the previous study \cite{SW-2000}.
\end{abstract}

\section{Introduction}\label{sec:intro}

The Fermi--Pasta--Ulam (FPU) system is a simple nonlinear dynamical lattice model describing a long one-dimensional chain of vibrating strings with nearest neighbor interactions. This model was first introduced by Fermi, Pasta, and Ulam in the original Los Alamos report \cite{FPUT-1955} in 1955 with regard to their numerical studies on nonlinear dynamics. At that time, it was anticipated that the energy initially given to only the lowest frequency mode would be shared by chaotic nonlinear interactions and it would be eventually thermalized to equilibrium. However, numerical simulations showed the opposite behavior. The energy is shared among only a few low-frequency modes and it exhibits quasi-periodic behavior. This phenomenon is known as the FPU paradox. Since then, the FPU paradox has emerged as one of the central topics in various fields, and it has stimulated extensive studies on nonlinear chaos. 

Among the various important studies in this regard, the most remarkable one is the fundamental work of Zabusky and Kruskal \cite{ZK-1965}, in which the problem was solved for the first time by discovering a connection to the Korteweg--de Vries (KdV) equation. The authors showed that the FPU system is \textit{formally} approximated by the KdV equation and its quasi-periodic dynamics is thus explained in  connection with solitary waves for the KdV equation. From an analysis perspective, Friesecke and Wattis proved that the FPU system has solitary waves \cite{FW-1994}, confirming the numerical observation \cite{E-1991}, whereas Friesecke and Pego established their convergence to the soliton solutions to the KdV equation \cite{FP-1999}. Moreover, various qualitative properties have been proved for the FPU solitary waves \cite{FP-2002, FP-2004, FP-2005, M-2013}.

The KdV approximation problem has also been investigated for general states without restriction to solitary waves. For an infinite chain, Schneider and Wayne showed that the FPU flow can be approximated by counter-propagating KdV flows (see \eqref{two wave approx} below) via the multi-scale method \cite{SW-2000}. This approach has been applied to a periodic setting \cite{PB-2005} as well as to generalized discrete models \cite{M-2006,CCPS-2012,GMWZ-2014}. Furthermore, with a different scaling, the cubic nonlinear Schr\"odinger equation is derived from the FPU system \cite{S-2010} (see also \cite{GM-2004,GM-2006}). 

In contrast, the FPU paradox can be explained in a completely different manner, i.e., by the approach of Izrailev and Chirikov \cite{IC-1965}, which involves the Kolmogorov--Arnold--Moser (KAM) theory:  quasi-periodicity occurs because the FPU system can be approximated by a finite-dimensional integrable system (see \cite{N-1971,R-2001,R-2006} for this direction). We also note that the quasi-periodic dynamics vanishes after a sufficiently long time-scale as predicted originally \cite{FMMPPRV-1982}. This phenomenon is called \textit{metastability}, and it has been investigated rigorously (e.g., \cite{BGG-2004,BP-2006}). Overall, the dynamics problem for the FPU system has garnered considerable research attention and it has been explored from various perspectives. We refer to the surveys  -\cite{Z-2005,G-2007} and the references therein for a more detailed history and an overview of the problem. 


In this article, we follow the approach of Zabusky and Kruskal \cite{ZK-1965}. Our objective is to provide a  rigorous justification of the KdV approximation for general solutions, including solitary waves, to infinite FPU chains. Let us begin with introducing the setup of the problem. Consider the FPU Hamiltonian 
\begin{equation}\label{original Hamiltonian}
H(q,p):=\sum_{x\in\mathbb{Z}}\frac{p(x)^2}{2}+V\big(q(x+1)-q(x)\big)
\end{equation}
for a function $(q,p)=\left(q(x),p(x)\right):\mathbb{Z}\to\mathbb{R}\times\mathbb{R}$. Here, $(q(x),p(x))$ denotes for the position and momentum of the $x$-th string, and the potential function $V:\mathbb{R}\to\mathbb{R}$ determines the potential energy from nearest-neighbor interactions. We assume that
\begin{equation}\label{V assumption}
V\in C^5,\quad V(0)=V'(0)=0,\quad V''(0)=:a>0\quad\textup{and}\quad V'''(0)=:b\neq 0.
\end{equation}
Such potentials include the cubic FPU potential $\frac{1}{2}ar^2+\frac{1}{6}br^3$, a more general polynomial potential $\sum_{k=2}^N \frac{c_k}{k!} r^k$, the Lennard-Jones potential $e[(1+\frac{r}{d})^{-12} - 2(1+\frac{r}{d})^{-6}+1]$ and the Toda potential $\alpha(e^{\beta r} - \beta r-1)$. 

The above-mentioned Hamiltonian generates the FPU system\begin{equation}\label{FPU before rescaling}
\left\{
\begin{aligned}
\partial_t q(t,x)&=p(t,x),\\
\partial_tp(t,x)&=V'\big(q(t,x+1)-q(t,x)\big)-V'\big(q(t,x)-q(t,x-1)\big),
\end{aligned}\right.
\end{equation}
where $(q,p)=\left(q(t,x), p(t,x)\right):\mathbb{R}\times\mathbb{Z}\to\mathbb{R}\times\mathbb{R}$. By combining the two equations in the system and then rewriting them for the relative displacement between two adjacent points, $r(t,x)=q(t,x+1)-q(t,x)$, we can simplify the system as 
\begin{equation}\label{FPU before scaling}
\boxed{\quad\partial_t^2 r=\Delta_1\big(V'(r)\big)\quad}
\end{equation}
where $\Delta_1 u=u(\cdot+1)+u(\cdot-1)-2u$. Next, by rescaling with
\begin{equation}\label{scaling}
\tilde{r}_h(t,x):=\frac{1}{h^2}r\Big(\frac{t}{h^3},\frac{x}{h}\Big): \mathbb{R}\times h\mathbb{Z}\to\mathbb{R}
\end{equation}
for small $h>0$, we obtain
\begin{equation}\label{scaled FPU}
h^6\partial_t^2 \tilde{r}_h=\Delta_h\big( V'(h^2\tilde{r}_h)\big),
\end{equation}
where $\Delta_h$ is a discrete Laplacian on $h\mathbb{Z}$, i.e.,
$$\Delta_hu=\frac{u(\cdot+h)+u(\cdot-h)-2u}{h^2}.$$
Finally, by extracting the linear term from the right-hand side of \eqref{scaled FPU}, we derive a discrete nonlinear wave equation, which we refer to hereafter as the FPU system 
\begin{equation}\label{FPU}
\boxed{\quad\textup{(FPU)}\quad\left\{\begin{aligned}
\partial_t^2 \tilde{r}_h-\frac{a}{h^4}\Delta_h\tilde{r}_h&=\frac{1}{h^6}\Delta_h\Big\{V'(h^2\tilde{r}_h)-ah^2\tilde{r}_h\Big\},\\
\tilde{r}_h(0)&=\tilde{r}_{h,0},\\
\partial_t\tilde{r}_h(0)&=\tilde{r}_{h,1},
\end{aligned}\right.\quad}
\end{equation}
where $\tilde{r}_h=\tilde{r}_h(t,x):\mathbb{R}\times h\mathbb{Z}\to\mathbb{R}$. This reformulated equation is still a Hamiltonian equation with the Hamiltonian\footnote{It is derived from \eqref{original Hamiltonian}.}
\begin{equation}\label{Hamiltonian}
H_h(\tilde{r}_h)=h\sum_{x\in h\mathbb{Z}}\Bigg\{\frac{1}{2}\bigg(\frac{h^2}{\sqrt{-\Delta_h}}\partial_t\tilde{r}_h\bigg)^2+\frac{1}{h^4}V(h^2\tilde{r}_h)\Bigg\}.
\end{equation}
Through the formal analysis described in Section \ref{sec: outline}, one would expect that the solutions to FPU \eqref{FPU} are approximated by counter-propagating waves
\begin{equation}\label{two wave approx}
\tilde{r}_h(t,x)\underset{h\to0}\approx w_h^+(t, x-\tfrac{t}{h^2})+ w_h^-(t, x+\tfrac{t}{h^2}),
\end{equation}
where each $w_h^\pm=w_h^\pm(t,x):\mathbb{R}\times\mathbb{R}\to\mathbb{R}$ is a solution to the KdV equation
\begin{equation}\label{KdV0}
\boxed{\quad\textup{(KdV)}\quad\left\{\begin{aligned}
\pa_t w_\pm \pm \frac{\sqrt{a}}{24} \pa_x^3 w_\pm \mp \frac{b}{4\sqrt{a}}\pa_x(w_\pm^2)&=0,\\
w_\pm(0)&=w_{\pm,0}.
\end{aligned}\right.\quad}
\end{equation}
This method of deriving the two KdV flows can be regarded as an infinite-lattice version of the method of Zabusky and Kruskal \cite{ZK-1965}.


In this study, we revisit the KdV limit problem for general solutions, albeit through a rather different approach. Indeed, in a broad sense, a dynamical system approach was adopted in all the aforementioned studies \cite{SW-2000, GM-2004, PB-2005, M-2006, GM-2006, S-2010, CCPS-2012,GMWZ-2014}. By regarding the FPU system \eqref{FPU} as a nonlinear dispersive equation, we exploit its dispersive and smoothing properties, and we then employ them to justify the KdV approximation. This approach enables us to not only simplify the assumptions on the initial data in the previous study but also reduce the regularity requirement.

For the statement of the main theorem, we introduce the basic definitions of function spaces, the Fourier transform and differentials on a lattice domain, and the linear interpolation operator. For $1\leq p\leq \infty$, the Lebesgue space $L^p(h\mathbb{Z})$ is defined by the collection of real-valued functions on a lattice domain $h\mathbb{Z}$ equipped with the $L^p$-norm
$$\|f_h\|_{L^p(h\Z)}:=
\left\{\begin{aligned}
&\left\{  h \sum_{x\in h\mathbb{Z}} |f_h(x)|^p \right\}^\frac1p&&\textup{if }1\leq p<\infty,\\
&\sup_{x\in h\mathbb{Z}} |f_h(x)|&&\textup{if }p=\infty.
\end{aligned}\right.$$
For $f_h\in L^1(h\mathbb{Z})$, we define its (discrete) Fourier transform by
$$(\mathcal{F}_h f_h)(\xi):=h\sum_{x\in h\mathbb{Z}}  f_h(x)e^{-ix\xi},\quad\forall\xi\in \R/(\tfrac{2\pi}{h}\Z) = [-\tfrac{\pi}{h},\tfrac{\pi}{h}).$$
Meanwhile, for a periodic function $f\in L^1([-\tfrac{\pi}{h},\tfrac{\pi}{h}))$, its inverse Fourier transform is given by
$$(\mathcal{F}_h^{-1} f)(x):=\frac{1}{2\pi}\int_{-\frac{\pi}{h}}^{\frac{\pi}{h}} f(\xi) e^{ix\xi}d\xi,\quad\forall x\in h\mathbb{Z}.$$
Then, Parseval's identity,
\begin{equation}\label{eq:parseval}
h\sum_{x\in h\mathbb{Z}} f(x)\overline{g(x)}=\frac{1}{2\pi}\int_{-\frac{\pi}{h}}^{\frac{\pi}{h}}(\mathcal{F}_hf)(\xi)\overline{(\mathcal{F}_hg)(\xi)}d\xi,
\end{equation}
extends the discrete Fourier transform (resp., its inversion) to $L^2(h\mathbb{Z})$ (resp., $L^2([-\frac{\pi}{h},\frac{\pi}{h}))$). 

There are several ways to define differentials on a lattice domain $h\mathbb{Z}$. Throughout the paper, we use the following different types of differentials, all of which are consistent with differentiation on the real line as the Fourier multiplier of the symbol $i\xi$ as $h\to0$.
\begin{definition}[Differentials on $h\mathbb{Z}$]\label{def: discrete differentials}
$(i)$ $\nabla_h$ (resp., $|\nabla_h|$, $\langle \nabla_h\rangle$) denotes the discrete Fourier multiplier of the symbol $\frac{2i}{h}\sin(\frac{h\xi}{2})$ (resp., $|\frac{2}{h}\sin(\frac{h\xi}{2})|$, $\langle\frac{2}{h}\sin(\frac{h\xi}{2})\rangle$), where $\langle \cdot \rangle=(1 + |\cdot|^2)^{\frac12}$.\footnote{These definitions are consistent with the discrete Laplacian $\Delta_h$, because $(-\Delta_h)$ is the Fourier multiplier of the symbol $\frac{4}{h^2}\sin^2(\frac{h\xi}{2})$; thus, $|\nabla_h|=\sqrt{-\Delta_h}$ and $\langle\nabla_h\rangle=\sqrt{1-\Delta_h}$.}\\
$(ii)$ $\partial_h$ (resp., $|\partial_h|$, $\langle \partial_h\rangle$) denotes the discrete Fourier multiplier of the symbol $i\xi$ (resp., $|\xi|$, $\langle\xi\rangle$).\\
$(iii)$ $\partial_h^+$ denotes the discrete right-hand side derivative naturally defined by
\[(\partial_h^+ f_h)(x) := \frac{f_h(x+h)-f_h(x)}{h}, \quad \forall x \in h\Z.\]
\end{definition}

For $s\in\mathbb{R}$, we define the Sobolev space $W^{s,p}(h\mathbb{Z})$ (resp., $\dot{W}^{s,p}(h\mathbb{Z})$) by the Banach space equipped with the norm
\begin{equation}\label{def:W}
\|f_h\|_{W^{s,p}(h\mathbb{Z})} := \|\langle\pa_h\rangle^sf_h\|_{L^p(h\mathbb{Z})}\quad\Big(\textup{resp., }\|f_h\|_{\dot{W}^{s,p}(h\mathbb{Z})} := \||\pa_h|^sf_h\|_{L^p(h\mathbb{Z})}\Big).
\end{equation}
In particular, when $p=2$, we denote
$$H^s(h\mathbb{Z}):=W^{s,2}(h\mathbb{Z})\quad\Big(\textup{resp., }\dot{H}^s(h\mathbb{Z}):=\dot{W}^{s,2}(h\mathbb{Z})\Big).$$

To compare functions on different domains, we introduce the linear interpolation
\begin{equation}\label{def:  linear interpolation}
\begin{aligned}
(l_h f_h)(x):&= f_h(hm)+(\partial_h^+ f_h)(hm)\cdot(x-hm)\\
&= f_h(hm)+\frac{f_h(hm+h)-f_h(hm)}{h}(x-hm)
\end{aligned}
\end{equation}
for all $x\in [hm,hm+h)$ with $m\in\mathbb{Z}$. Note that the linear interpolation converts a function $f_h:h\mathbb{Z}\rightarrow \R$ on a lattice into a continuous function on the real line.

Now, we are ready to state our main result.

\begin{theorem}[KdV limit for FPU]\label{main theorem}
If $V$ satisfies \eqref{V assumption}, then for any $R>0$, there exists $T(R)>0$ such that the following holds. Suppose that for some $s\in(\frac{3}{4},1]$,
\begin{equation}\label{our initial data assumption}
\sup_{h\in (0,1]}\big\|\big(\tilde{r}_{h,0}, h^2\nabla_h^{-1}\tilde{r}_{h,1}\big)\big\|_{H^{s}(h\mathbb{Z})\times H^{s}(h\mathbb{Z})}\leq R.
\end{equation}
Let $\tilde{r}_h(t)\in C_t([-T,T]; H_x^s(h\mathbb{Z}))$ be the solution to FPU \eqref{FPU} with initial data $(\tilde{r}_{h,0},\tilde{r}_{h,1})$, and let $w_h^\pm(t)\in C_t([-T,T]; H_x^s(\mathbb{R}))$ be the solution to KdV \eqref{KdV0} with interpolated initial data $\frac12 l_h(\tilde{r}_{h,0}\mp h^2\nabla_h^{-1} \tilde{r}_{h,1})$.\\
$(i)$ (Continuum limit)
\begin{equation}\label{continuum limit}
\sup_{t \in [-T,T]} \big\|(l_h\tilde{r}_h)(t,x)- w_h^+(t, x-\tfrac{t}{h^2})- w_h^-(t, x+\tfrac{t}{h^2})\big\|_{L_x^2(\mathbb{R})}\lesssim h^{\frac{2s}{5}}.
\end{equation}
\\
$(ii)$ (Small amplitude limit) Scaling back, 
$$r(t,x)=h^2\tilde{r}_h(h^3t,hx):\mathbb{R}\times\mathbb{Z}\to\mathbb{R}$$
is a solution to FPU \eqref{FPU before scaling}. Moreover, it satisfies
\begin{equation}\label{small amplitude limit}
\sup_{t \in [-\frac{T}{h^3},\frac{T}{h^3}]}\big\|(l_1r)(t,x)- h^2w_h^+(h^3t, h(x-t))- h^2w_h^-(h^3t, h(x+t))\big\|_{L_x^2(\mathbb{R})}\lesssim h^{\frac{3}{2}+\frac{2s}{5}}.
\end{equation}
\end{theorem}

We remark that the assumption on the initial data is simplified compared to the previous work \cite{SW-2000}. We assume only a uniform bound on the size of the initial data (see \eqref{our initial data assumption}) in a natural Sobolev norm (without any weight), and the mean-zero momentum condition $\sum_{x\in h\mathbb{Z}}\tilde{r}_{h,1}(x)=0$ is not imposed. Furthermore, the regularity requirement is reduced to $s>\frac{3}{4}$.

As for the regularity issue, we emphasize that reducing the regularity is not only a matter of mathematical curiosity but it may also lead to a significant improvement in the continuum limit \eqref{continuum limit}. As stated in our main theorems, the KdV approximation is stated in the form of either a continuum limit or a small amplitude limit. Mathematically, they are equivalent; however, the continuum limit \eqref{continuum limit} seems rather weaker because it holds only in a short time interval $[-T,T]$, whereas the small amplitude limit \eqref{small amplitude limit} is valid almost globally in time $[-\frac{T}{h^3},\frac{T}{h^3}]\to (-\infty, \infty)$ as $h\to 0$. Thus, it would be desirable to extend the time interval $[-T,T]$ arbitrarily for the continuum limit. For comparison, we state that for discrete nonlinear Schr\"odinger equations (DNLS), the continuum limit is established in a compact time interval of any size \cite{HY-2019SIAM}, and an exponential-in-time bound is obtained. In the proof, conservation laws obviously play a crucial role. However, unlike DNLS, the FPU system does not have a conservation law controlling a higher regularity norm, say the $H^1$ norm. Only an $L^2$-type quantity is controlled by its Hamiltonian \eqref{Hamiltonian}. Therefore, it would be desirable to establish the continuum limit for $L^2$-data. If such a low regularity convergence is achieved, then one may try to employ the conservation law to extend the size of the interval to be arbitrarily large. Although the regularity is significantly reduced in this study, our assumption that $s>\frac{3}{4}$ is still far from the desired case of $s=0$. At the end of this section, we mention the technical obstacle that prevents us from going below  $s=\frac{3}{4}$. Instead of sharpening the estimates, a new idea seems necessary to reduce the regularity.

The main contribution of this article is to present a new approach to the KdV limit problem from the perspective of the theory of nonlinear dispersive PDEs. In spite of the dispersive nature of the FPU system, which is clear from its connection to the KdV equation, to the best of authors' knowledge, there has been no attempt to tackle the problem using dispersive PDE techniques thus far.

Our approach is achieved on the basis of the following two observations. First, as outlined in Section \ref{sec: outline}, we reformulate the FPU  system \eqref{FPU} by separating its Duhamel formula into two coupled equations \eqref{coupled FPU}, which we refer to as the coupled FPU. Indeed, this is a standard method to deal with inhomogeneous wave equations; however after implementing it, we realized that it is much easier to understand the limit procedure by analyzing the symbols of the linear propagators and their asymptotics (see Remark \ref{coupled remark}). By this refomulation, we introduce a different convergence scheme to the KdV equation via the decoupled FPU \eqref{decoupled FPU}. It makes the problem more suitable and clearer for analysis by dispersive PDE techniques.

Second, we discover that the linear propagators $S_h^\pm(t)=e^{\mp\frac{t}{h^2}(\nabla_h-\partial_h)}$ for the coupled and decoupled FPUs exhibit  properties similar to those of the Airy flows $S^\pm(t)=e^{\mp\frac{t}{24}\partial_x^3}$ in many aspects. A technical but crucial feature of our analysis is that the phase functions of the linear FPU propagators are comparable with those of the Airy propagators at different derivative levels. Indeed, direct calculations show that
$$\begin{aligned}
&\tfrac{1}{h^2}(\xi-\tfrac{2}{h}\sin(\tfrac{h\xi}{2}))\sim \xi^3,&&\big\{\tfrac{1}{h^2}(\xi-\tfrac{2}{h}\sin(\tfrac{h\xi}{2}))\big\}'=\tfrac{1}{h^2}(1-\cos(\tfrac{h\xi}{2}))\sim \xi^2,\\
&\big\{\tfrac{1}{h^2}(\xi-\tfrac{2}{h}\sin(\tfrac{h\xi}{2}))\big\}''=\tfrac{1}{2h}\sin(\tfrac{h\xi}{2})\sim \xi,
\end{aligned}$$
on the frequency domain $[-\frac{\pi}{h},\frac{\pi}{h})$ for the discrete Fourier transform (see Figure \ref{Fig: FPU symbol}). This allows us to recover the Strichartz estimates, the local smoothing and maximal function estimates (Proposition \ref{prop:Strichartz}), and the bilinear estimates (Lemma \ref{lem:bi1}) for the linear FPU flows owing to the ``magical'' property of Zabusky and Kruskal's transformation of the FPU system in their original study \cite{ZK-1965}. Indeed, dispersive equations on a lattice domain do not enjoy smoothing in general. For instance, the phase function for the linear Schr\"odinger flow $e^{it\Delta_h}$ is comparable with that for the linear Schr\"odinger flow on $\mathbb{R}$, i.e., $-\frac{2t}{h^2}(1-\cos(h\xi))\sim-t\xi^2$ on $[-\frac{\pi}{h},\frac{\pi}{h})$; however, its derivative $-\frac{2t}{h}\sin(h\xi)$ is far from $-t\xi$ near the high frequency edge $\xi=\pm\frac{\pi}{h}$  (see Figure \ref{Fig: Schrodinger symbol}). Therefore, the discrete linear Schr\"odinger flow does not enjoy local smoothing at all (see \cite{IZ-2009}). With various dispersive and smoothing estimates for the linear FPU flows, we follow a general strategy (see \cite{HY-2019SIAM} for instance) to prove the convergence from the coupled to the decoupled FPU and the convergence from the decoupled FPU to the KdV equation. First, we employ the linear and bilinear estimates to obtain $h$-uniform bounds for solutions to the coupled and decoupled FPUs. Then, using the uniform bounds, we directly measure the differences to prove the convergences.

\begin{figure}[h!]
\begin{center}
\begin{tikzpicture}[scale=1]
\draw[line width=1.5] (0,0) -- (3,0);
\draw[->] (3,0) -- (5,0) node[below] {$\xi$};
\draw[->] (0,0) -- (0,4) node[right] {};
\draw[gray] (0,0) parabola (2,4);
\draw[gray] (0,0) parabola (3.4641,4);
\node at (2,3.6){\color{gray}$\xi^2$};
\node at (3.4,3.5){\color{gray}$\frac{4\xi^2}{\pi^2}$};
\draw[thick] (0,0) cos (1.5,1.5) sin (3,3) cos (4.5,1.5);
\node at (4.5,1.2){$\frac{2}{h^2}(1-\cos(h\xi))$};
\draw[dashed] (0,3)--(3,3);
\draw[dashed] (3,0)--(3,3);
\node at (-0.1,-0.1){$0$};
\node at (3,-0.3){$\frac{\pi}{h}$};
\node at (-0.3,3){$\frac{4}{h^2}$};
\node at (1.5,-0.3){\tiny (frequency domain)};
\node at (1.5,-1){\boxed{\text{symbols}}};
\end{tikzpicture}
\qquad\qquad
\begin{tikzpicture}[scale=1]
\draw[line width=1.5] (0,0) -- (3,0);
\draw[->] (3,0) -- (5,0) node[below] {$\xi$};
\draw[->] (0,0) -- (0,4) node[right] {};
\draw[gray] (0,0) -- (2.4317,4);
\node at (1.4,2){\color{gray}$2\xi$};
\draw[thick] (0,0) sin (1.5,1.5708) cos (3,0);
\node at (3.8,0.5){$\frac{2}{h}\sin(h\xi)$};
\draw[dashed] (3,0)--(3,4);
\node at (-0.1,-0.1){$0$};
\node at (3,-0.3){$\frac{\pi}{h}$};
\node at (1.5,-0.3){\tiny (frequency domain)};
\node at (2.5,-1){\boxed{\text{derivatives of symbols}}};
\end{tikzpicture}
\end{center}
\caption{Comparison between the discrete and the continuous Schr\"odinger flows: The symbol $\frac{2}{h^2}(1-\cos(h\xi))$ is comparable with the symbol $\xi^2$ on the frequency domain $[-\frac{\pi}{h},\frac{\pi}{h})$.  However, their derivatives are not comparable particularly near the endpoints $\xi=\pm\frac{\pi}{h}$.}\label{Fig: Schrodinger symbol}
\end{figure}
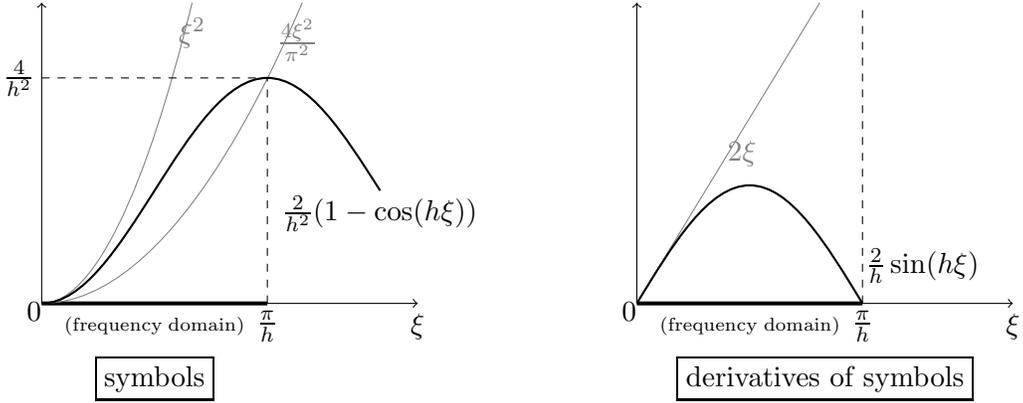

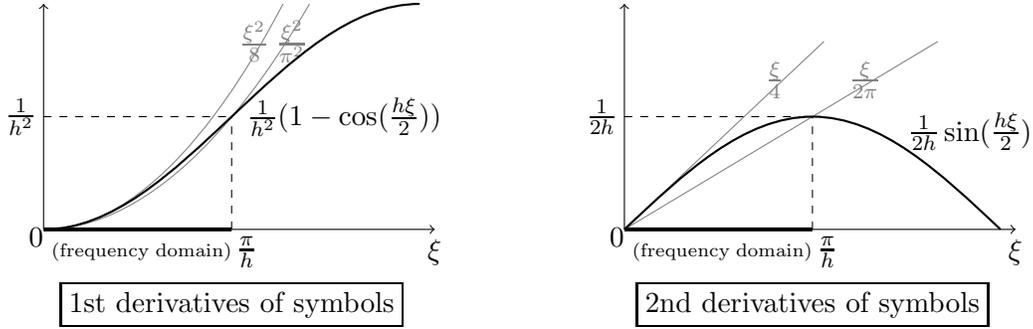
\begin{figure}[h!]
\begin{center}
\begin{tikzpicture}[scale=1]
\draw[line width=1.5] (0,0) -- (2.5,0);
\draw[->] (2.5,0) -- (5.2,0) node[below] {$\xi$};
\draw[->] (0,0) -- (0,3) node[right] {};
\draw[gray] (0,0) parabola (3.1831,3);
\draw[gray] (0,0) parabola (3.5355,3);
\node at (2.8,2.5){\color{gray}$\frac{\xi^2}{8}$};
\node at (3.3,2.5){\color{gray}$\frac{\xi^2}{\pi^2}$};
\draw[thick] (0,0) cos (2.5,1.5) sin (5,3);
\node at (4,1.5){$\frac{1}{h^2}(1-\cos(\frac{h\xi}{2}))$};
\draw[dashed] (2.5,0)--(2.5,1.5);
\draw[dashed] (0,1.5)--(2.5,1.5);
\node at (-0.3,1.5){$\frac{1}{h^2}$};
\node at (-0.1,-0.1){$0$};
\node at (2.7,-0.3){$\frac{\pi}{h}$};
\node at (1.3,-0.3){\tiny (frequency domain)};
\node at (2.5,-1){\boxed{\text{1st derivatives of symbols}}};
\end{tikzpicture}
\qquad\qquad
\begin{tikzpicture}[scale=1]
\draw[line width=1.5] (0,0) -- (2.5,0);
\draw[->] (2.5,0) -- (5.2,0) node[below] {$\xi$};
\draw[->] (0,0) -- (0,3) node[right] {};
\draw[gray] (0,0) -- (2.6526,2.5);
\draw[gray] (0,0) -- (4.1667,2.5);
\node at (2,2){\color{gray}$\frac{\xi}{4}$};
\node at (3.2,2){\color{gray}$\frac{\xi}{2\pi}$};
\draw[thick] (0,0) sin (2.5,1.5) cos (5,0);
\node at (4.6,1.3){$\frac{1}{2h}\sin(\frac{h\xi}{2})$};
\draw[dashed] (2.5,0)--(2.5,1.5);
\draw[dashed] (0,1.5)--(2.5,1.5);
\node at (-0.3,1.5){$\frac{1}{2h}$};
\node at (-0.1,-0.1){$0$};
\node at (2.7,-0.3){$\frac{\pi}{h}$};
\node at (1.3,-0.3){\tiny (frequency domain)};
\node at (2.5,-1){\boxed{\text{2nd derivatives of symbols}}};
\end{tikzpicture}
\end{center}
\caption{Comparison between the linear FPU flow and the Airy flow: The derivatives of the symbol $\frac{1}{h^2}(\xi-\tfrac{2}{h}\sin(\tfrac{h\xi}{2}))$ are comparable with those of the symbol $\xi^3$ on the frequency domain $[-\frac{\pi}{h},\frac{\pi}{h})$.}\label{Fig: FPU symbol}
\end{figure}

We conclude the introduction with a comment on the obstacle to reducing the regularity below $s=\frac{3}{4}$. As mentioned above, our strategy heavily employs uniform bounds for the coupled and decoupled FPUs, and their proofs resemble those of the local well-posedness of the KdV equation. The $X^{s,b}$-norm (see Section \ref{sec: Xsb space}) is well known as a powerful tool for the low regularity theory. Indeed, Kenig, Ponce, and Vega established the local well-posedness of KdV in $H^s(\mathbb{R})$ for $s>-\frac{3}{4}$ using the bilinear estimates in $X^{s,b}$ \cite{KPV-1996}. Thus, one may attempt to employ the $X^{s,b}$- norm for the KdV limit problem. However, this norm is too sensitive to the linear propagator, because a certain weight is imposed away from the characteristic curve on the space-time Fourier side. Hence, it is not suitable to measure two different linear flows having different characteristic curves at the same time. Indeed, Proposition \ref{Xsb bad} shows that linear FPU flows are not uniformly bounded in the $X^{s,b}$-norm associated with the Airy flows (the other direction can be proved similarly). Therefore, we do not use the $X^{s,b}$-norm. We  employ the Strichartz estimates, the local smoothing and maximal function estimates, and their corresponding norms (see Proposition \ref{prop:Strichartz} and \ref{Thm:linear estimates for KdV flows}), because they are not sensitive to the propagators. These norms have been employed in the previous work of Kenig, Ponce, and Vega \cite{KPV-1991} for the local well-posedness of KdV in $H^s$ for $s>\frac{3}{4}$. However, it is known that the maximal function estimate holds only when $s>\frac{3}{4}$. Therefore, we are currently unable to go below $s=\frac{3}{4}$.

 \subsection{Organization of the paper} The remainder of this paper is organized as follows. In Section \ref{sec: outline}, the outline of the proof of Theorem \ref{main theorem} is presented. In particular, FPU systems are reformulated and Theorem \ref{main theorem} is reduced to two propositions. In Section \ref{sec:pre}, some definitions and estimates, in particular, well-known estimates and the Littlewood--Paley theory on a lattice and $X^{s,b}$ space are introduced. In Section \ref{sec: Well-posedness}, the local well-posedness of FPU is established. In Section \ref{sec: linear estimate}, Strichartz, local smoothing, and maximal function estimates of linear FPU flows are discussed in comparison with linear KdV flows. In Section \ref{sec: bilinear estimates}, $X^{s,b}$ bilinear estimates are proven. In Section \ref{proof of uniform bound}. In Section \ref{sec:main proof}, the main theorem is proven by combining the proofs of two propositions. Finally, in Appendices \ref{app:FLE} and \ref{app:GP}, justification of the non-triviality of the approximation via $X^{s,b}$ analysis and the estimate of the higher-order term are discussed, respectively.

\subsection{Notations and basic definitions}

In this article, we deal with two different types of functions, i.e.,  functions on the real line $\mathbb{R}$ and functions on the lattice domain $h\mathbb{Z}$. To avoid possible confusion, we use the subscript $h$ for functions on $h\mathbb{Z}$ with no exception. For instance, $u_h$, $v_h$, and $w_h$ are defined on $h\mathbb{Z}$, while $u$, $v$, and $w$ are defined on $\mathbb{R}$.

If there is no confusion, we assign lower-case letters $x, y, z, ...$ to spatial variables regardless of whether they are on the lattice or on the real line; for instance, $u_h(x):h\mathbb{Z}\to \mathbb{R}$ and $u(x):\mathbb{R}\to\mathbb{R}$. Note that the subscript $h$ determines the space of the spatial variable.

For notational convenience, we may abbreviate the domain and codomain of a function in the norm. For example, for $f_h=f_h(x): h\mathbb{Z}\to\mathbb{R}$ (resp., $f=f(x):\mathbb{R}\to\mathbb{R}$), 
$$\begin{aligned}
\|f_h\|_{L^p}&=\|f_h\|_{L^p(h\mathbb{Z})}\quad\big(\textup{resp., }\|f\|_{L^p}=\|f\|_{L^p(\mathbb{R})}\big),\\
\|f_h\|_{W^{s,p}}&=\|f_h\|_{W^{s,p}(h\mathbb{Z})}\quad\big(\textup{resp., }\|f\|_{W^{s,p}}=\|f\|_{W^{s,p}(\mathbb{R})}\big),
\end{aligned}$$
and for $F_h=F_h(t,x): h\mathbb{Z}\to\mathbb{R}$ (resp., $F=F(t,x):\mathbb{R}\to\mathbb{R}$), 
$$\begin{aligned}
\|F_h\|_{L_t^qL_x^r}&=\|F_h\|_{L_t^q([-T,T];L_x^r(h\mathbb{Z}))}\quad\big(\textup{resp., }\|F\|_{L_t^qL_x^r}=\|F\|_{L_t^q([-T,T];L_x^r(\mathbb{R}))}\big),\\
\|F_h\|_{L_t^qW_x^{s,r}}&=\|F_h\|_{L_t^q([-T,T];W_x^{s,r}(h\mathbb{Z}))}\quad\big(\textup{resp., }\|F\|_{L_t^qW_x^{s,r}}=\|F\|_{L_t^q([-T,T];W_x^{s,r}(\mathbb{R}))}\big).
\end{aligned}$$
Similarly, for a vector-valued function $(f_h^+, f_h^-): h\mathbb{Z}\to \mathbb{R}\times\mathbb{R}$, we have
$$\|(f_h^+, f_h^-)\|_{L^p}=\big\||(f_h^+, f_h^-)|\big\|_{L^p(h\mathbb{Z})}=\Big\|\big\{(f_h^+)^2+(f_h^-)^2\big\}^{1/2}\Big\|_{L^p(h\mathbb{Z})}.$$

\subsection{Acknowledgement}
This research of the first author was supported by Basic Science Research Program through the National Research Foundation of Korea(NRF) funded by the Ministry of Education (NRF-2017R1C1B1008215). The second author was supported by FONDECYT Postdoctorado 2017 Proyecto No. 3170067 and project France-Chile ECOS-Sud C18E06. The third author was supported by Samsung Science and Technology Foundation under Project Number SSTF-BA1702-02. Part of this work was complete while the second author was visiting Chung-Ang University (Seoul, Republic of Korea). The second author acknowledges the warm hospitality of the institution.

\section{Outline of the proof}\label{sec: outline}

The proof of the main theorem (Theorem \ref{main theorem}) is outlined in this section. Although our proof is strongly inspired by the original idea of Zabusky and Kruskal \cite{ZK-1965}, it is reorganized to fit into the theory of dispersive PDEs.

First, we note that the constants $a$ and $b$ in the assumption \eqref{V assumption} can be normalized by constant multiplication and scaling $\tilde{r}_h(t,x) \mapsto \frac{a}{b}\tilde{r}_h(\sqrt{a}t,x)$. Thus, we assume that $a=b=1$. By Taylor's theorem, the nonlinear term in FPU \eqref{FPU} can be expressed as 
$$\frac{1}{h^6}\Delta_h\Big\{V'(h^2\tilde{r}_h)-h^2\tilde{r}_h\Big\}=\frac{1}{2h^2}\Delta_h\Big\{\big(\tilde{r}_h\big)^2+ h^2\mathcal{R}\Big\},$$
with the higher-order remainder
\begin{equation}\label{eq:R}
\mathcal{R}:=\frac{V^{(4)}(h^2\tilde{r}_h^*)}{3} \big(\tilde{r}_h\big)^3,
\end{equation}
where $V^{(4)}$ denotes the fourth-order derivative of $V$ and $\tilde{r}_h^*$ is some number between $0$ and $\tilde{r}_h$. To avoid non-essential complexity, we suggest that readers consider the FPU with simple quadratic nonlinearity, i.e., $\mathcal{R}=0$, which corresponds to the normalized standard FPU potential $V(r)=\frac{r^2}{2}+\frac{r^3}{6}$.

\subsection{Reformulation of FPU as a coupled system}\label{sec: outline 1}
By Duhamel's formula, the initial data problem for FPU \eqref{FPU} is written as 
\begin{equation}\label{Duhamel FPU'}
\begin{aligned}
\tilde{r}_h(t)&=\cos\left(\frac{t\sqrt{-\Delta_h}}{h^2}\right)\tilde{r}_{h,0}+\sin\left(\frac{t\sqrt{-\Delta_h}}{h^2}\right)\frac{h^2}{\sqrt{-\Delta_h}}\tilde{r}_{h,1}\\
&\quad-\frac{1}{2}\int_0^t \sin\left(\frac{(t-t_1)\sqrt{-\Delta_h}}{h^2}\right)\sqrt{-\Delta_h}\Big\{\tilde{r}_h(t_1)^2 + h^2\mathcal{R}(t_1)\Big\} dt_1.
\end{aligned}
\end{equation}
Observe that
\begin{align*}
\cos\left(\frac{t\sqrt{-\Delta_h}}{h^2}\right)=\frac{1}{2}\big(e^{\frac{t}{h^2}\nabla_h}+e^{-\frac{t}{h^2}\nabla_h}\big),\quad\sin\left(\frac{t\sqrt{-\Delta_h}}{h^2}\right)= \frac{1}{2i}\big(e^{\tfrac{t}{h^2}\nabla_h}  -  e^{-\frac{t}{h^2}\nabla_h}\big)\mathcal{H}.
\end{align*}
where $\nabla_h$ is the discrete Fourier multiplier of the  symbol $\frac{2i}{h}\sin(\frac{h\xi}{2})$ (see Definition \ref{def: discrete differentials}) and $\mathcal{H}$ is the Hilbert transform, i.e., the Fourier multiplier of the symbol $-i\textup{sign}(\xi)$. Indeed, $\cos(\frac{t}{h^2}|\frac{2}{h}\sin(\frac{h\xi}{2})|)=\cos(\frac{2t}{h^3}\sin(\tfrac{h\xi}{2}))=\frac{1}{2}(e^{i\frac{2t}{h^3}\sin(\frac{h\xi}{2})}+e^{-i\frac{2t}{h^3}\sin(\frac{h\xi}{2})})$, and the other identity can be shown similarly. Thus, by inserting these into the Duhamel formula \eqref{Duhamel FPU'} and separating the operators $e^{\mp\frac{t}{4h^2}\nabla_h}$, we deduce that if
$$\big(\tilde{r}_h^+, \tilde{r}_h^-\big):\mathbb{R}\times h\mathbb{Z}\to\mathbb{R}\times\mathbb{R}$$
solves the system of coupled equations
\begin{equation}\label{coupled FPU}
\tilde{r}_h^\pm(t) =e^{\mp\frac{t}{h^2}\nabla_h}\tilde{r}_{h,0}^\pm \mp\frac14\int_0^t e^{\mp\frac{(t-t_1)}{h^2}\nabla_h} \nabla_h\Big\{\tilde{r}_h(t_1)^2 + h^2\mathcal{R}(t_1)\Big\} dt_1
\end{equation}
with initial data
\begin{equation}\label{q0+-}
\tilde{r}_{h,0}^\pm= \frac12\Big\{ \tilde{r}_{h,0}\mp h^2 \nabla_h^{-1} \tilde{r}_{h,1}\Big\},
\end{equation}
then
\begin{equation}\label{q+-}
\tilde{r}_h(t,x)=\tilde{r}_h^+(t,x) +\tilde{r}_h^-(t,x)
\end{equation}
is the solution to FPU \eqref{Duhamel FPU'}. 

Next, we introduce
\begin{equation}\label{trans}
u_h^\pm(t):=e^{\pm\frac{t}{h^2}\partial_h}\tilde{r}_h^\pm(t),
\end{equation}
where $\partial_h$ is given in Definition \ref{def: discrete differentials}. Then, they solve the coupled integral equation
$$u_h^\pm(t) =  e^{\mp\frac{t}{h^2}(\nabla_h-\partial_h)}\tilde{r}_{h,0}^\pm \mp \frac14\int_0^t e^{\mp\frac{(t-t_1)}{h^2}(\nabla_h-\partial_h)} \nabla_h e^{\pm\frac{t_1}{h^2}\partial_h}\Big\{\tilde{r}_h(t_1)^2 + h^2\mathcal{R}(t_1)\Big\}dt_1.$$
Note that the main nonlinear term in the integral can be written as
\begin{equation}\label{eq:translation}
e^{\pm\frac{t_1}{h^2}\partial_h}\Big\{\tilde{r}_h(t_1)^2\Big\}=e^{\pm\frac{t_1}{h^2}\partial_h}\Big\{\tilde{r}_h^+(t_1)+\tilde{r}_h^-(t_1)\Big\}^2=\Big\{u_h^\pm(t_1)+e^{\pm\frac{2t_1}{h^2}\partial_h}u_h^\mp(t_1)\Big\}^2,
\end{equation}
because $e^{a\partial_h}(u_h^2)=(e^{a\partial_h}u_h)^2$ holds.\footnote{By the discrete Fourier transform, $$\begin{aligned}
\mathcal{F}_h\left(e^{a\partial_h}(u_h^2)\right)(\xi)&=e^{ia\xi}\frac{1}{2\pi}\int_{-\pi/h}^{\pi/h}(\mathcal{F}_hu_h)(\xi-\eta) (\mathcal{F}_hu_h)(\eta)d\eta\\
&=\frac{1}{2\pi}\int_{-\pi/h}^{\pi/h}e^{ia(\xi-\eta)}(\mathcal{F}_hu_h)(\xi-\eta) e^{ia\eta}(\mathcal{F}_hu_h)(\eta)d\eta=\mathcal{F}_h\left((e^{a\partial_h}u_h)^2\right)(\xi).
\end{aligned}$$
This computation can be extended to any polynomial of finite degree.}
Therefore, the equation \eqref{coupled FPU} is reformulated as a coupled system of integral equations, which we refer to as the \textit{coupled FPU}, 
\begin{equation}\label{coupled FPU'}
\boxed{\quad\begin{aligned}
u_h^\pm(t)&= S_h^\pm(t)u_{h,0}^\pm\mp\frac14 \int_0^t S_h^\pm(t-t_1) \nabla_h \bigg[\Big\{u_h^\pm(t_1)+e^{\pm\frac{2t_1}{h^2}\partial_h}u_h^\mp(t_1)\Big\}^2\\
&\hspace{19em}+ h^2e^{\pm\frac{t_1}{h^2}\partial_h}\mathcal{R}(t_1)\bigg]dt_1,
\end{aligned}\quad}
\end{equation}
with initial data
\begin{equation}\label{coupled FPU' initial}
u_{h,0}^\pm= \frac12 \Big\{ \tilde{r}_{h,0}\mp h^2\nabla_h^{-1} \tilde{r}_{h,1} \Big\},
\end{equation}
where the linear FPU propagator is denoted by
\begin{equation}\label{eq: linear FPU flow}
\boxed{\quad S_h^\pm(t)=e^{\mp\frac{t}{h^2}(\nabla_h-\partial_h)}\quad}.
\end{equation}
\begin{remark}\label{coupled remark}
$(i)$ By construction, FPU \eqref{FPU} can be recovered from the equation \eqref{coupled FPU'} via
\begin{equation}\label{coupled FPU' to FPU}
\tilde{r}_h(t,x)=e^{-\frac{t}{h^2}\partial_h}u_h^+(t,x) +e^{\frac{t}{h^2}\partial_h}u_h^-(t,x).
\end{equation}
$(ii)$ $e^{\pm\frac{t}{h^2}\partial_h}$ is an \textit{almost translation} in that at each discrete time $t=h^3k$ with $k\in\mathbb{N}$, 
$$(e^{\mp \frac{t}{h^2}\partial_h}u_h^\pm)(t,x)=(e^{\mp hk\partial_h}u_h^\pm)(t,x)=u_h^\pm(t, x\mp hk)=u_h^\pm(t, x\mp \tfrac{t}{h^2}).$$
Thus, at least formally, if $u_h^\pm(t,x) \approx w^\pm(t,x)$, then by \eqref{coupled FPU' to FPU}, the solution $\tilde{r}_h(t,x)$ to the FPU becomes asymptotically decoupled into the counter-propagating flows $w^+(t, x- \tfrac{t}{h^2})$ (moving to the right) and $w^-(t, x+\tfrac{t}{h^2})$ (moving to the left).\\
$(iii)$ If the nonlinear solution $u_h^\pm(t)$ behaves almost linearly in a short time interval, the coupled term $e^{\pm\frac{2t_1}{h^2}\partial_h}u_h^\mp(t_1,x)$ in \eqref{coupled FPU'} is approximated by
$$e^{\pm\frac{t_1}{h^2}(\partial_h+\nabla_h)}u_{h,0}^\mp(x)=\frac{1}{2\pi}\int_{-\frac{\pi}{h}}^{\frac{\pi}{h}}e^{\pm\frac{it_1}{h^2}(\xi+\frac{2}{h}\sin(\frac{h\xi}{2}))+ix\xi}(\mathcal{F}_hu_{h,0}^\mp)(\xi)d\xi.$$
This term is expected to vanish as $h\to 0$ owing to fast dispersion. Note that its group velocity $\mp\frac{1}{h^2}(1+\cos(\frac{h\xi}{2}))=\mp \frac{1}{h^2}(2-\frac{h^2\xi^2}{8}+\cdots)$ diverges to $\mp\infty$ for $|\xi|\leq\frac{\pi}{h}$. The higher-order remainder  $e^{\pm\frac{t_1}{h^2}\partial_h}\mathcal{R}(t_1)$ is also expected to vanish owing to the spare $h$ of order $2$. \\
$(iv)$ The linear propagator $S_h^\pm(t)$ formally converges to the Airy flow, because by Taylor's theorem, $\mp\frac{1}{h^2}(\frac{2}{h}\sin(\frac{h\xi}{2})-\xi)=\pm(\frac{\xi^3}{24}-\frac{h^2\xi^5}{1920}+\cdots)\to\pm\frac{\xi^3}{24}$ as $h\to 0$.
\end{remark}

\subsection{From coupled to decoupled FPU}\label{sec: outline 2}

As mentioned in Remark \ref{coupled remark} $(iii)$ and $(v)$, one would expect that the coupled terms $e^{\pm\frac{2t_1}{h^2}\partial_h}u_h^\mp(t_1,x)$ in \eqref{coupled FPU'}, as well as the $O(h^2)$-order remainder term, vanish as $h\to 0$. Thus, dropping them in \eqref{coupled FPU'}, we derive the following decoupled system, which we refer to as the \textit{decoupled FPU}:
\begin{equation}\label{decoupled FPU}
\boxed{\quad\begin{aligned}
v_h^\pm(t) = S_h^\pm(t)u_{h,0}^\pm\mp\frac14 \int_0^t S_h^\pm(t-t_1) \nabla_h\Big\{v_h^{\pm}(t_1)\Big\}^2 dt_1,
\end{aligned}\quad}
\end{equation}
where
$$v_h^\pm=v_h^\pm(t,x): \mathbb{R}\times h\mathbb{Z}\to\mathbb{R}.$$
It is easy to show that both the coupled and the decoupled FPUs are well-posed (see Proposition \ref{LWP of CS}). However, their well-posedness is not sufficient for rigorous reduction to the decoupled equation. Indeed, the time interval of existence, given by the well-posedness, may shrink to zero as $h\to 0$; however, some regularity is also required to measure the difference between two solutions. Thus, we exploit the dispersive and smoothing properties of the linear FPU flows in Section \ref{sec: linear estimate} and \ref{sec: bilinear estimates}, and we obtain finer uniform-in-$h$ bounds for nonlinear solutions (Proposition \ref{Prop:LWP}). In Section \ref{sec: coupled to decoupled}, by using these uniform bounds, we justify the convergence from the coupled to the decoupled FPU.

\begin{proposition}[Convergence to the decoupled FPU]\label{convergence to decoupled FPU}
If $V$ satisfies \eqref{V assumption}, then for any $R>0$, there exists $T(R)>0$ such that the following holds. Let $s\in (0,1]$. Suppose that
$$\sup_{h\in(0,1]}\big\|\big(u_{h,0}^+, u_{h,0}^-\big)\big\|_{H^s(h\mathbb{Z})\times H^s(h\mathbb{Z})}\leq R.$$
Let $(u_h^+(t), u_h^-(t))$ (resp., $(v_h^+(t),v_h^-(t))$) be the solution to the coupled FPU \eqref{coupled FPU'} (resp., the decoupled FPU \eqref{decoupled FPU}) with initial data $(u_{h,0}^+, u_{h,0}^-)$. Then, 
$$\big\|u_h^\pm(t)-v_h^\pm(t)\big\|_{C_t([-T,T]; L_x^2(h\mathbb{Z}))}\lesssim h^s\|u_0^\pm\|_{H^s(\mathbb{R})}.$$
\end{proposition}

\subsection{From decoupled FPU to KdV}\label{sec: outline 3}
As mentioned in Remark \ref{coupled remark} $(iv)$, by convergence of symbols, each equation in the decoupled FPU \eqref{decoupled FPU} is expected to converge to the Korteweg--de Vries equation (KdV)
\begin{equation}\label{KdV}
\boxed{\quad
w^\pm(t) = S^\pm(t) u_0^\pm \mp \frac14 \int_0^t S^\pm(t-t_1) \partial_x\Big\{w^\pm(t_1)\Big\}^2 dt_1,
\quad}
\end{equation}
where
$$w^\pm=w^\pm(t,x): \mathbb{R}\times\mathbb{R}\to \mathbb{R}$$
and
\begin{equation}\label{eq: Airy flow}
\boxed{\quad S^\pm(t)=e^{\mp\frac{t}{24}\partial_x^3}\quad}
\end{equation}
denotes the Airy flow.

In Section \ref{sec: decoupled to KdV}, we establish the convergence from the decoupled FPU \eqref{decoupled FPU} to KdVs \eqref{KdV}. Its proof  uses uniform bounds for nonlinear solutions (see Section \ref{proof of uniform bound}).
\begin{proposition}[Convergence to KdVs]\label{convergence to KdV}
If $V$ satisfies \eqref{V assumption}, then for any $R>0$, there exists $T(R)>0$ such that the following holds. Let $s\in (\frac{3}{4},1]$. Suppose that
$$\sup_{h\in(0,1]}\big\|\big(u_{h,0}^+, u_{h,0}^-\big)\big\|_{H^s(h\mathbb{Z})\times H^s(h\mathbb{Z})}\leq R.$$
Let $\big(v_h^+(t), v_h^-(t)\big)$ (resp., $(w_h^+(t),w_h^-(t))$)
be the solution to the decoupled FPU \eqref{decoupled FPU} (resp., the KdVs \eqref{KdV}) with initial data $(u_{h,0}^+, u_{h,0}^-)$ (resp., with initial data $(l_hu_{h,0}^+, l_hu_{h,0}^-)$), where $l_h$ is the linear interpolation operator to be defined in \eqref{def:  linear interpolation}. Then,
$$\big\|l_hv_h^\pm(t)-w_h^\pm(t)\big\|_{C_t([-T,T]; L_x^2(\mathbb{R}))}\lesssim h^{\frac{2s}{5}}.$$
\end{proposition}

\begin{remark}\label{rem:w_h}
To avoid confusion, we explain here that the solutions $w_h^{\pm}$ to KdVs \eqref{KdV} with initial data $l_hu_{h,0}^{\pm}$ are real-valued functions posed  not on $h\Z$ but on $\R$ because, as seen, their initial data depend on $h$. Therefore, $w_h^{\pm}$ involve subscript $h$.
\end{remark}
Finally, by combining Proposition \ref{convergence to decoupled FPU} (with Lemma \ref{Lem:discretization linearization inequality}) and Proposition \ref{convergence to KdV}, we complete the proof of our main theorem. Figure \ref{Fig:Main} shows the convergence scheme outlined in this section.

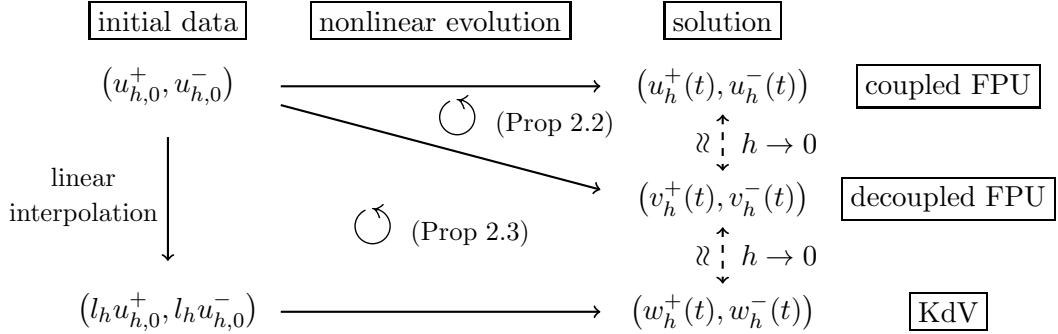
\begin{figure}[h!]
\begin{center}
\begin{tikzpicture}[scale=0.5]
\node at (-7.25,-4){$\big(l_h u_{h,0}^+,l_h u_{h,0}^-\big)$};
\draw[thick,->] (-4.25,-4) -- (4.25,-4);
\node at (-7.25,3.75){\boxed{\text{initial data}}};
\node at (-7.25,2){$\big(u_{h,0}^+,u_{h,0}^-\big)$};
\draw[thick,->] (-4.25,2) -- (4.25,2);
\node at (0,3.75){\boxed{\text{nonlinear evolution}}};
\draw[thick,->] (-4.25,1.5) -- (4.25,-0.75);
\node at (2.25,1.15){\huge$\circlearrowleft$\ \small (Prop \ref{convergence to decoupled FPU})};
\draw[thick,->] (-7.25,0.65) -- (-7.25,-2.65);
\node at (-0,-1.75){\huge$\circlearrowleft$\ \small (Prop \ref{convergence to KdV})};
\node at (-9.5,-0.4){{\small linear}};
\node at (-9.5,-1.4){{\small interpolation}};
\node at (7.5,-4){$\big(w_h^+(t),w_h^-(t)\big)$};
\node at (13.5,-4){\boxed{\text{KdV}}};
\node at (7.5,-1){$\big(v_h^+(t),v_h^-(t)\big)$};
\node at (13.5,-1){\boxed{\text{decoupled FPU}}};
\draw[thick,dashed,<->] (7.5,-1.75) -- (7.5,-3.25);
\node at (7,-2.5){\rotatebox{90}{$\approx$}};
\node at (9,-2.5){$h \to 0$};
\node at (7.5,3.75){\boxed{\text{solution}}};
\node at (7.5,2){$\big(u_h^+(t),u_h^-(t)\big)$};
\node at (13.5,2){\boxed{\text{coupled FPU}}};
\node at (7,0.5){\rotatebox{90}{$\approx$}};
\draw[thick,dashed, <->] (7.5,1.25) -- (7.5,-0.25);
\node at (9,0.5){$h \to 0$};
%
\end{tikzpicture}
\end{center}
\caption{Convergence scheme from FPU to KdV.}\label{Fig:Main}
\end{figure}

\section{Preliminaries}\label{sec:pre}

In this section, we summarize the basic analysis tools for functions on a lattice.

\subsection{Basic inequalities and Littlewood-Paley theory on a lattice}

By definition, $L^p(h\mathbb{Z})=\ell^p(\mathbb{Z})$ but $\|f_h\|_{L^p(h\mathbb{Z})}=h^{1/p}\|f_h\|_{\ell^p(\mathbb{Z})}$. Thus, we have H\"older's inequality
$$\|f_hg_h\|_{L^p(h\mathbb{Z})} \le \|f_h\|_{L^{p_1}(h\mathbb{Z})}\|g_h\|_{L^{p_2}(h\mathbb{Z})},\quad \tfrac{1}{p}=\tfrac{1}{p_1}+\tfrac{1}{p_2},$$
the standard duality relation
$$\|f_h\|_{L^p(h\mathbb{Z})} =\sup_{\|g\|_{L^{p'}(h\mathbb{Z})}\le 1} h\bigg|\sum_{x\in h\mathbb{Z}^d}f_h(x)\overline{g_h(x)}\bigg|,\quad \tfrac{1}{p}+\tfrac{1}{p'}=1,$$
and Young's inequality
\begin{equation}\label{eq:young}
\|f_h * g_h\|_{L^r(h\mathbb{Z})} \le \|f_h\|_{L^{p}(h\mathbb{Z})}\|g_h\|_{L^{q}(h\mathbb{Z})},\quad 1+\tfrac{1}{r}=\tfrac{1}{p}+\tfrac{1}{p},
\end{equation}
where the convolution $*=*_h$ is defined by
\[(f_h * g_h)(x) = h \sum_{y \in h\Z} f_h(y) g_h(x-y).\]

For the basic properties of Sobolev spaces on a lattice, we have the following lemmas.

\begin{lemma}[Norm equivalence, Proposition 1.2 in \cite{HY-2019DCDS}]\label{Prop:norm equivalence}
	For any $1<p<\infty$, we have
\begin{equation}\label{norm equivalence}
\|f_h\|_{\dot{W}^{s,p}(h\mathbb{Z})}\sim \| |\nabla_h|^s f_h\|_{L^p(h\mathbb{Z})}\quad\forall s\in\mathbb{R}
\end{equation}
	and
	$$\|f_h\|_{\dot{W}^{1,p}(h\mathbb{Z})}\sim \| \nabla_h f_h \|_{L^p(h\mathbb{Z})}\sim \|\pa_{h}^{\pm}f_h\|_{L^p(h\mathbb{Z})}.$$
\end{lemma}

\begin{lemma}[Sobolev embedding, Proposition 2.5 in \cite{HY-2019DCDS}]
	Let $h\in(0,1]$. If $1\leq p<q\leq\infty$ and $\frac{1}{q}>\frac{1}{p}-s$, then
\begin{equation}\label{Sobolev}
\|f_h\|_{L^q(h\mathbb{Z})}\lesssim \|f_h\|_{W^{s,p}(h\mathbb{Z})}.
\end{equation}
\end{lemma}

In contrast to the continuous domain case, differential operators are bounded on a lattice; however, the bound blows up as $h\to0$.
\begin{lemma}[Boundedness of differential operators, Lemma 2.2 in \cite{HY-2019SIAM}]\label{Lem:Sobolev bound depending on h}
For $h > 0$ and $0\le s_1\le s_2$, we have
\begin{align*}
\|f_h\|_{\dot H^{s_2}(h\mathbb{Z})} &\lesssim \frac{1}{h^{s_2-s_1}} \|f_h\|_{\dot H^{s_1}(h\mathbb{Z})}.
\end{align*}
\end{lemma}

\begin{lemma}[Leibniz rule for discrete differentials]\label{Lem:Leibnitz rule}
Differential operators $\partial_h^+$ and $\nabla_h$ allow the following types of Leibniz rule:
\begin{align}
\partial_h ^+( f_h g_h ) &= \partial_h^+ f_h\cdot g_h+f_h(\cdot+h)\cdot\partial_h^+ g_h, \label{Leibnitz rule1}\\
\nabla_h ( f_h g_h ) &=\nabla_h f_h\cdot\cos(\tfrac{-ih\pa_h}{2})g_h 
+\cos(\tfrac{-ih\pa_h}{2})f_h\cdot\nabla_h g_h,\label{Leibnitz rule2}
\end{align}
where $\cos(\frac{-ih\pa_h}{2})$ denotes the Fourier multiplier of the symbol $\cos(\frac{h\xi}{2})$.
\end{lemma}

\begin{proof}
Here, \eqref{Leibnitz rule1} follows from the definition. For \eqref{Leibnitz rule2}, we take the Fourier transform of the left-hand side: 
$$\mathcal F_h \big( \nabla_h(f_hg_h)\big)(\xi)=\frac{2i}{h}\sin\Big(\frac{h\xi}{2}\Big)\frac{1}{2\pi}\int_{-\frac{\pi}{h}}^{\frac{\pi}{h}} (\mathcal{F}_hf_h)(\eta)(\mathcal{F}_hg_h)(\xi-\eta)d\eta.$$
We apply the identity $\frac{2i}{h}\sin(\frac{h\xi}{2})=\frac{2i}{h}\sin(\frac{h\eta}{2})\cos(\frac{h(\xi-\eta)}{2})+\cos(\frac{h\eta}{2})\frac{2i}{h}\sin(\frac{h(\xi-\eta)}{2})$ to the integral and then take the inversion.
\end{proof}

Let $\phi:\mathbb{R}\to[0,1]$ be an even smooth bump function such that $\phi(\xi)=1$ for $|\xi|\leq 1$ and $\phi(\xi)= 0$ for $|\xi|\geq 2$. For a dyadic number $N\in 2^{\mathbb{Z}}$ with $N\leq 1$, set $\psi_N$ by
\[\psi_N(\xi) = \phi\left(\frac{h\xi}{\pi N}\right) - \phi\left(\frac{2h\xi}{\pi N}\right).\]
Note that $\mbox{supp} \psi_N \subset \{\xi : \frac{\pi N}{2h} \le |\xi| \le \frac{2\pi N}{h}\}$, and $\{\psi_N\}$ is a partition of unity on $\T_h$, i.e., $\sum_{N\le1} \psi_N \equiv 1$. Now we define the Littlewood-Paley projection operator $P_N=P_{N;h}$ as the Fourier multiplier operator given by
\begin{equation}\label{LP}
\mathcal{F}_h(P_{N}f_h)(\xi)=\psi_N(\xi)(\mathcal{F}_hf_h)(\xi).
\end{equation}
Moreover, we define $ P_{\le N}$ by $\mathcal F_h (P_{\le N}f_h)=\sum_{M\le N} \psi_M \mathcal F f_h$. 

\begin{remark}\label{rem:low}
For each $h \in (0,1]$, there exists $N_0 = N_{0,h} \in 2^{\mathbb{Z}}$ such that $\pi N \le h$ holds for all $N \le N_0$. The projection $P_{\le N_0}$ on the lattice corresponds to $P_{\le 1}$ on $\R$ (referred to as a \emph{low} frequency piece). Indeed, 
\begin{equation}\label{eq:decomposition}
\sum_{N\leq 1}\psi_N(\xi) \sim \phi(\xi)+\sum_{N_0 < N\leq 1}\psi_N(\xi) \equiv 1,\quad\forall\xi\in \T_h.
\end{equation}
\end{remark}

\begin{proposition}[Littlewood-Paley inequality, Theorem 4.2 in \cite{HY-2019DCDS}]\label{LP inequalities}
For $1<p<\infty$, we have
$$\|f_h\|_{L^p(h\mathbb{Z})}\lesssim\Big\|\Big(\sum_{N\leq 1}|P_N f_h|^2\Big)^{1/2}\Big\|_{L^p(h\mathbb{Z})}\lesssim \|f_h\|_{L^p(h\mathbb{Z})}.$$
\end{proposition}

\begin{lemma}[Bernstein’s inequality, Lemma 2.3 in \cite{HY-2019DCDS}]
Let $h \in (0,1]$. If $1 \le p \le q \le \infty$, then we have
\begin{equation}\label{ineq:Bernstein}
\|P_N f_h\|_{L^q(h\mathbb{Z})} \lesssim \left(\frac{N}{h}\right)^{\frac1p - \frac1q}\|P_N f_h\|_{L^p(h\mathbb{Z})}.
\end{equation}
\end{lemma}

\subsection{Properties of
 Linear interpolation}\label{sec:DisLin}
A function $f_h:h\mathbb{Z}\rightarrow \R$ on a lattice domain becomes continuous by linear interpolation
\begin{equation}\label{def:  linear interpolation}
(l_h f_h)(x):= f_h(x_m)+(\partial_h^+ f_h)(x_m)\cdot(x-x_m),\quad\forall x\in x_m+[0,h),
\end{equation}

This operator is bounded in Sobolev spaces.
\begin{lemma}[Boundedness of 
linear interpolation, Lemma
5.2 in \cite{HY-2019SIAM}]\label{Lem:discretization linearization inequality}
Let $0\leq s \leq 1$. Then, for $f_h\in H^s(h\mathbb{Z})$, we have 
\begin{align}\label{ineq:discretization linearization inequality}
\| l_h f_h \|_{\dot{H}^s(\mathbb{R})}\lesssim \| f_h \|_{\dot{H}^s(h\mathbb{Z})}.
\end{align}
\end{lemma}

\begin{proof}
See Lemma
5.2 in \cite{HY-2019SIAM} for the proofs.
\end{proof}

The linear interpolation operator and the differential (in some sense) are exchangeable at the cost of one additional derivative.
\begin{proposition}
\label{Pro:reverse order}
If $f_h \in \dot{H}^2(h\mathbb{Z})$, then
\begin{align*}
\| l_h \nabla_h f_h -\pa_x l_h f_h\|_{L^2(\mathbb{R})}
\ls h \| f_h\|_{\dot H^2(h\mathbb{Z})}.
\end{align*}
\end{proposition}
\begin{proof}
By definition, we have
$$l_h \nabla_h f_h(x) - \pa_x l_h f_h(x)
= \nabla_h f_h(x_m) +\pa_h^+(\nabla_h f_h)(x_m) \cdot (x-x_m)
-\pa_h^+ f_h(x_m)$$
for $x\in[x_m,x_m+h)$; thus, 
\begin{align*}
\| l_h \nabla_h f_h  - \pa_x l_h f_h \|_{L^2(\R)}
\le \|\nabla_h f_h- \pa_h^+ f_h \|_{L^2(h\Z)}
+ h \| \pa_h^+\nabla_h f_h\|_{L^2(\R)}.
\end{align*}
Plancherel's theorem and the norm equivalence \eqref{norm equivalence} yield
\[\begin{aligned}
\|\nabla_h f_h- \pa_h^+ f_h \|_{L^2(h\mathbb{Z})}
&=\big\|(\tfrac{2i}{h}\sin(\tfrac{h\xi}{2}) - \tfrac{e^{ih\xi}-1}{h})(\mathcal{F}_hf_h)(\xi)\big\|_{L_\xi^2([-\frac{\pi}{h},\frac{\pi}{h}))} \\
&=\big\|\big\{\tfrac{\cos(h\xi)-1}{h}+i(\tfrac{\sin(h\xi)}{h} -\tfrac{2}{h}\sin(\tfrac{h\xi}{2}))\big\} (\mathcal{F}_hf_h)(\xi)\big\|_{L_\xi^2([-\frac{\pi}{h},\frac{\pi}{h}))}\\
&\le h \| f\|_{\dot H^2(h\mathbb{Z})} +h^2 \| f_h\|_{ \dot H^3(h\mathbb{Z})} \ls h \| f\|_{\dot H^2(h\mathbb{Z})},
\end{aligned}\]
where, in the last step, we use Lemma \ref{Lem:Sobolev bound depending on h}.
\end{proof}

\begin{proposition}[Almost distribution]\label{Pro:almost distribution}
\begin{align*}
\| \pa_x l_h (f_h^2)- \pa_x ( l_h f_h )^2 \|_{L^2(\R)}
\ls h \| (\pa_h^+ f_h)^2\|_{L^2(h\mathbb{Z})}.
\end{align*}
\end{proposition}
\begin{proof}
From the definition of $l_h$ and \eqref{Leibnitz rule1}, we write for $x\in[x_m,x_m+h)$ that
\begin{align*}
\pa_x l_h (f_h^2)(x)- \pa_x ( l_h f_h )^2(x) 
=&~{}\pa_h^+(f_h^2)(x_m) - 2l_hf_h(x)\pa_h^+f_h(x_m) \\
=&~{}\pa_h^+f_h(x_m)f_h(x_m+h)+\pa_h^+f_h(x_m)f_h(x_m)\\
&-2( f_h(x_m) + \pa_h^+f_h(x_m)(x-x_m))\pa_h^+f_h(x_m) \\
=&~{}h\big(\pa_h^+f_h(x_m)\big)^2-2\big(\pa_h^+f_h(x_m)\big)^2(x-x_m).
\end{align*}
Taking the $L^2(\mathbb{R})$ norm, we complete the proof.
\end{proof}
\subsection{$X^{s,b}$ spaces}\label{sec: Xsb space}
In this subsection, we introduce the $X^{s,b}$ spaces\footnote{They are sometimes called the Bourgain spaces or dispersive Sobolev spaces.} introduced by Bourgain \cite{B-1993Sch} and further developed by Kenig, Ponce, and Vega \cite{KPV-1996} and Tao \cite{T-2001}. 

First, we define the function space in a general setting. In subsequent applications, the spatial domain $\Lambda$ will be either the real line $\mathbb{R}$ or the lattice $h\mathbb{Z}$, and the associated symbol $P$ is chosen according to the model considered. Since the following are stated in a general setting, they can be applied in a unified way. We refer to \cite{T-2006} for the details and proofs.

\begin{definition}[$X^{s,b}$ spaces]\label{def:Xsb}
Let $\Lambda$ be either $\mathbb{R}$ or $h\mathbb{Z}$. Let $P$ be a real-valued continuous function. For $s,b\in\R$, we define the $X_P^{s,b}(\R \times \Lambda)$ spaces ($X^{s,b}$ in short) as the completion of $\mathcal S(\R\times \Lambda)$ with respect to the norm 
\[\| u \|_{X^{s,b}} := \bigg\{\iint_{\R \times \widehat{\Lambda}}\la \xi \ra^{2s} \la \tau -P(\xi) \ra^{2b} |\tilde{u}(\tau,\xi)|^2 d\xi d\tau\bigg\}^{1/2},\]
where $\tilde{u}$ denotes the space-time Fourier transform of $u$ defined by\footnote{In particular, when $\lambda = h\Z$, $\tilde u$ (as in Definition \ref{def Xsb}) is defined by 
\[\tilde u_h(\tau, \xi) = h \sum_{x \in h\Z} \int_{\R} e^{-it\tau} e^{-ix \xi} u_h(t,x) \; d\tau.\]
}
\[\tilde u (\tau, \xi) = \int_{\R \times \Lambda} e^{-it\tau} e^{-ix \xi} u(t,x) \; d\tau dx\]
and $\widehat{\Lambda}$ is the Pontryagin dual space of $\Lambda$, i.e, $\widehat{\R} = \R$ and $\widehat{h\Z} = \T_h$.
\end{definition}
The following are well-known properties of $X^{s,b}$ spaces (see, for instance, \cite{T-2006} for the proofs).
\begin{lemma}\label{lem:properties}
Let $s,b \in \R$ and $X^{s,b}$ spaces be defined as in Definition \ref{def:Xsb}. Let $\theta \in \mathcal S(\R)$ be a (compactly supported) cut-off function. Then, the following properties hold:
\begin{enumerate}
\item (Nesting) $X^{s',b'} \subset X^{s,b}$ whenever $s \le s',~b\le b'$.
\item (Well-defined for linear solutions) For any $f \in H^s$, we have 
\[\left\|\theta(t) e^{itP(-i\nabla)} f \right\|_{X^{s,b}} \ls_{\theta,b} \| f\|_{H^s}.\]
\item (Transference principle) Let $Y$ be a Banach space such that the inequality 
\[\|e^{it\tau_0}e^{itP(-i\nabla)}f\|_{Y} \lesssim_b \|f\|_{H^s}\]
holds for all $f \in H^s$ and $\tau_0 \in \R$. If,  additionally, $b > \frac12$, then we have the embedding
\[\|u\|_{Y} \lesssim \|u\|_{X^{s,b}}.\]
In particular, we have
\begin{equation}\label{eq:embedding}
\|u\|_{C_tH_x^s} \lesssim \|u\|_{X^{s,b}}.
\end{equation}
\item (Stability with respect to time localization) Let $0 < T < 1$, $b > \frac12$ and $f \in H^s$. We have
\[\big\|\theta(\tfrac{t}{T}) e^{itP(-i\nabla)} f \big\|_{X^{s,b}} \ls_{\theta, b} T^{\frac12 - b} \| f\|_{H^s}.\]
If $-\frac12<b' \le b<\frac12$, then we have
\[\big\|\theta(\tfrac{t}{T}) u\big\|_{X^{s,b'}} \ls_{\theta, b, b'} T^{b-b'} \| u\|_{X^{s,b}}.\]
\item (Inhomogeneous estimate) Let $b>\frac12$. Then, we have 
\begin{equation}\label{inhomogeneous estimate}
\left\| \theta(t) \int_0^t e^{i(t-t_1)P(-i\nabla)} F(t_1)dt_1 \right\|_{X^{s,b}}  \ls \| F\|_{X^{s,b-1}}.
\end{equation}
\end{enumerate} 
\end{lemma}
\begin{remark}
The proof of the above-mentioned lemma under the discrete setting is analogous to the one under the continuous setting, since the proof is based on the temporal Fourier analysis. 
\end{remark}

Now, we fix the symbols associated with the discrete linear FPU flows, and we focus on the corresponding $X^{s,b}$ spaces, because they are our main function spaces.

\begin{definition}[$X_{h,\pm}^{s,b}$ spaces]\label{def Xsb}
For $s,b\in\R$, we define the discrete Bourgain spaces $X_{h,\pm}^{s,b}=X_{h,\pm}^{s,b}(\R\times h\Z)$ as the completion of $\mathcal S(\R\times h\Z)$ with respect to the norm 
\[\| u_h \|_{X_{h,\pm}^{s,b}} 
:= \bigg\{\int_{-\infty}^\infty\int_{-\frac{\pi}{h}}^{\frac{\pi}{h}}\la \xi \ra^{2s} \la \tau \mp s_h(\xi) \ra^{2b} |\tilde{u}_h(\tau,\xi)|^2 d\xi d\tau\bigg\}^{1/2},\]
where $\tilde{u}_h$ denotes the (discrete) space-time Fourier transform of $u_h$, and 
\begin{equation}\label{eq:s_h}
s_h(\xi):=\frac{1}{h^2}\Big(\xi-\frac{2}{h}\sin\Big(\frac{h\xi}{2}\Big)\Big).
\end{equation}
\end{definition}


\begin{remark}
The Littlewood-Paley theory ensures
\[\|f_h\|_{X_{h,\pm}^{s,b}}^2 \sim \sum_{N_0 \le N \le 1} \left(\frac{N}{h}\right)^{2s}\|P_N f_h\|_{X_{h,\pm}^{0,b}}^2.\]
This facilitates a type of fractional Leibniz rule; see Lemma \ref{lem:prod}.
\end{remark}

We end this section with the following temporal Sobolev embedding property.
\begin{lemma}[Temporal Sobolev embedding]\label{lem:tsobolev}
Let $2 \le p < \infty$ and $u_h^{\pm}$ be a smooth function on $\R \times h\Z$. Then for $b \ge \frac12 - \frac1p$, we have
\[\norm{f}_{L_t^p(\R;H_x^s(h\Z))} \lesssim \norm{f}_{X_{h,\pm}^{s,b}}.\]
When $p = \infty$, the usual Sobolev embedding ($b > \frac12$) holds.
\end{lemma}

\begin{proof}
The proof directly follows from the Sobolev embedding with respect to the temporal variable $t$. For $S_h^{\pm}(t)u_{h}^{\pm}(t,x) = \mathcal F_h^{-1}[e^{\mp its_h(\xi)}\mathcal F(u_h^\pm)(t,\xi)]$, we know that $\|S_h^\pm(-t)u_h^\pm\|_{H_x^s} = \|u_h^\pm\|_{H_x^s}$. Thus,
\[\|u_h^\pm\|_{L_t^p(H_x^s)} = \|\|S_h^\pm(-t)u_h^\pm\|_{H_x^s}\|_{L_t^p} \lesssim \|\|S_h^\pm(-t)u_h^\pm\|_{H_x^s}\|_{H_t^b}= \|u_h^\pm\|_{X_{h,\pm}^{s,b}},\]
which completes the proof.
\end{proof}


\section{Well-posedness of coupled and decoupled FPUs}\label{sec: Well-posedness}

The well-posedness of a nonlinear difference (or discrete differential) equation is obvious in most cases  owing to the boundedness of discrete differential operators. Nevertheless, the proof of  the local well-posedness of the coupled and decoupled FPUs is included for the readers' convenience.

\begin{proposition}[Local well-posedness of coupled and decoupled FPUs]\label{LWP of CS}
Let $h \in (0,1]$ be fixed. For any $R>0$, there exists $T(R,h)>0$ such that the following holds. Suppose that
$$\big\|\big(u_{h,0}^+, u_{h,0}^-\big)\big\|_{L^2(h\mathbb{Z})}\leq R.$$
Then, there exists a unique solution $(u_h^+, u_h^-)\in C_t([-T,T]; L_x^2(h\mathbb{Z}))$ $(\textup{resp., }(v_h^+, v_h^-)\in C_t([-T,T]; L_x^2(h\Z)))$ to the coupled FPU \eqref{coupled FPU'} (resp., the decoupled FPU \eqref{decoupled FPU}) with initial data $(u_{h,0}^+,u_{h,0}^-)$. Moreover, $(u_h^+, u_h^-)$ preserves the Hamiltonian $H_h(\tilde{r}_h)$ (see \eqref{Hamiltonian}), where $\tilde{r}_h$ is given by \eqref{coupled FPU' to FPU}.
\end{proposition}
\begin{remark}\label{rem:simple}
In the proof below, we do not estimate the higher-order remainder term in \eqref{coupled FPU'}, since the higher-order term contains a spare $h$ of order $2$ and it is thus small and nonessential in our analysis. For readers' convenience, we refer to Lemma \ref{lem:AppB} for the proof of the estimate of the higher-order remainder term.
\end{remark}
\begin{proof}[Proof of Proposition \ref{LWP of CS}]
We drop the time interval $[-T,T]$ in the notation $C_t([-T,T])$. We consider only the coupled FPU, because the decoupled FPU can be dealt with in the same way. 

We define a nonlinear map $\Phi=(\Phi^+, \Phi^-)$ by 
$$\Phi^\pm(u_h^+, u_h^-):= S_h^\pm(t)u_{h,0}^\pm\mp\frac{1}{4}  \int_0^t S_h^\pm(t-t_1) \nabla_h \Big\{u_h^\pm(t_1)+e^{\pm\frac{2t_1}{h^2}\partial_h}u_h^\mp(t_1)\Big\}^2dt_1.$$
Let $T>0$ be a small number to be chosen later. Then, by unitarity, it follows that
$$\begin{aligned}
\|\Phi^{\pm}(u_h^+, u_h^-)\|_{C_tL_x^2}\leq \|u_{h,0}^\pm\|_{L_x^2}+\frac{T}{2}\Big\|\nabla_h \Big\{u_h^\pm+e^{\pm\frac{2t}{h^2}\partial_h}u_h^\mp\Big\}^2\Big\|_{C_tL_x^2}.
\end{aligned}$$
For the nonlinear term, we observe that by the boundedness of the discrete differential operator $\nabla_h$ (see Definition \ref{def: discrete differentials}) and the trivial inequality $\|f_h\|_{L_x^4} \le h^{-\frac14}\|f_h\|_{L_x^2}$, we have
$$\begin{aligned}
\Big\|\nabla_h \Big\{u_h^\pm+e^{\pm\frac{2t}{h^2}\partial_h}u_h^\mp\Big\}^2\Big\|_{L_x^2}&\leq\frac{2}{h}  \Big\| \Big\{u_h^\pm+e^{\pm\frac{2t}{h^2}\partial_h}u_h^\mp\Big\}^2\Big\|_{L_x^2}=\frac{2}{h}  \big\|u_h^\pm+e^{\pm\frac{2t}{h^2}\partial_h}u_h^\mp\big\|_{L_x^4}^2\\
&\leq \frac{2}{h^{3/2}}\Big\{\|u_h^\pm\|_{L_x^2}+\big\|e^{\pm\frac{2t}{h^2}\partial_h}u_h^\mp\big\|_{L_x^2}\Big\}^2\leq\frac{4}{h^{3/2}}\|(u_h^+,u_h^-)\|_{L_x^2}^2.
\end{aligned}$$
Thus, we obtain 
\[\|\Phi(u_h^+, u_h^-)\|_{C_tL_x^2}\leq R+\frac{2T}{h^{3/2}}\|(u_h^+,u_h^-)\|_{C_tL_x^2}^2.\]
Similarly, one can show that 
$$\begin{aligned}
&\|\Phi(u_h^+, u_h^-)-\Phi(\tilde{u}_h^+, \tilde{u}_h^-)\|_{C_tL_x^2}\\
&\leq \frac{2T}{h^{3/2}}\Big\{\|(u_h^+,u_h^-)\|_{C_tL_x^2}+\|(\tilde{u}_h^+,\tilde{u}_h^-)\|_{C_tL_x^2}\Big\}\|(u_h^+-\tilde{u}_h^+,u_h^--\tilde{u}_h^-)\|_{C_tL_x^2}.
\end{aligned}$$
Taking $T=\frac{h^{3/2}}{16R}$, we prove that $\Phi$ is contractive on the ball in $C_tL_x^2$ of radius 
$2R$ centered at zero. Therefore, local well-posedness follows from the contraction mapping principle. 

By a straightforward computation, we prove the conservation law,
\[\begin{aligned}
\frac{d}{dt}H_h(\tilde{r}_h) &=h\sum_{x \in h\Z} h^4\frac{1}{\sqrt{-\Delta_h}}\partial_t \tilde{r}_h\cdot\frac{1}{\sqrt{-\Delta_h}}\partial_t^2 \tilde{r}_h+\frac{1}{h^2}V'(h^2\tilde{r}_h)\partial_t\tilde{r}_h\\
&=h^5\sum_{x \in h\Z} (-\Delta_h)^{-1}\partial_t \tilde{r}_h\cdot\Big\{\partial_t^2 \tilde{r}_h-\frac{1}{h^6}\Delta_h V'(h^2\tilde{r}_h)\Big\}=0,
\end{aligned}\]
where in the last step, we use $\tilde{r}_h$ to solve \eqref{FPU}.
\end{proof}


\section{Linear FPU flows}\label{sec: linear estimate}

We investigate various dispersive and smoothing properties for the linear FPU flows (Proposition \ref{prop:Strichartz}), and we then show how these discrete flows can be approximated by the Airy flows as $h\to0$ (Proposition \ref{linear approx}). Later, in Section \ref{sec: decoupled to KdV}, the main results of this section will be employed to prove the convergence from the decoupled FPU to KdVs.

\subsection{Estimates for the linear FPU flows}

We establish dispersive and smoothing inequalities for the linear FPU propagator $S_h^\pm(t)$, i.e., the discrete Fourier multipliers of the symbol $e^{\mp\frac{it}{h^2}(\frac{2}{h}\sin(\frac{h\xi}{2})-\xi)}$.

\begin{proposition}[Estimates for the linear FPU flows]\label{prop:Strichartz} Let $s>\frac{3}{4}$. Suppose that $2\leq q,r\leq\infty$ and $\frac2q+\frac1r=\frac12$ with $(q,r)\neq (4,\infty)$.\\
$(i)$ (Strichartz estimates)
$$\big\| |\pa_h|^{\frac1q} S_h^\pm(t)u_{h,0}\big\|_{L_{t}^q(\mathbb{R}; L_x^r(h\mathbb{Z}))}\ls \|u_{h,0}\|_{L^2(h\mathbb{Z})}.$$
$(ii)$ (Local smoothing estimate)
$$\| \partial_h S_h^\pm(t)u_{h,0}\|_{L_x^\infty(h\mathbb{Z}; L_t^2(\mathbb{R}))} \ls \|u_{h,0}\|_{L^2(h\mathbb{Z})}.$$
$(iii)$ (Maximal function estimate)
$$\|S_h^\pm(t)u_{h,0} \|_{L_x^2(h\mathbb{Z}; L_{t}^\infty([-1,1]))} \ls \|u_{h,0}\|_{H^s(h\mathbb{Z})}.$$
In all the three above-mentioned inequalities, the implicit constants are independent of $h\in(0,1]$. Moreover, the differential operator $\partial_h$ in $(i)$ and $(ii)$ can be replaced by $\pa_h^+$ or $\nabla_h$ (see Definition \ref{def: discrete differentials}).
\end{proposition}

\begin{remark}\label{rem:LS}
Proposition \ref{prop:Strichartz} $(i)$ may hold at $(q,r)=(4,\infty)$; however, it is excluded here to simplify the proof. Indeed, this endpoint case is not necessary in this article.
\end{remark}

As a direct consequence of Proposition \ref{prop:Strichartz} and the transference principle (Lemma \ref{lem:properties} (3)), we obtain the bounds in the associated Bourgain spaces.
\begin{corollary}\label{cor:Strichartz}
Let $s>\frac{3}{4}$ and $b>\frac{1}{2}$. Suppose that $2\leq q,r\leq\infty$ and $\frac2q+\frac1r=\frac12$ with $(q,r)\neq (4,\infty)$. Then, we have
$$\big\| |\pa_h|^{\frac1q} u_h\big\|_{L_{t}^q(\mathbb{R}; L_x^r(h\mathbb{Z}))}+\| \partial_h u_h\|_{L^\infty(h\mathbb{Z}; L_t^2(\mathbb{R}))}\ls \|u_h\|_{X_{h,\pm}^{0,b}}$$
and
$$\|u_h\|_{L^2(h\mathbb{Z}; L_{t}^\infty([-1,1]))} \ls \|u_h\|_{X_{h,\pm}^{s,b}},$$
where the implicit constants are independent of $h\in(0,1]$ and $\partial_h$ can be replaced by $\pa_h^+$ or $\nabla_h$.
\end{corollary}

Before presenting the proof of Proposition \ref{prop:Strichartz}, let us recall and compare with the linear estimates for the Airy propagator $S^\pm(t)=e^{\mp\frac{t}{24}\partial_x^3}$ from Kenig, Ponce and Vega \cite{KPV-1991}.

\begin{proposition}[Linear estimates for the Airy flows]\label{Thm:linear estimates for KdV flows} Let $s>\frac{3}{4}$. Suppose that $2\leq q,r\leq\infty$ and $\frac2q+\frac1r=\frac12$.\\
$(i)$ (Strichartz estimates)
$$\big\| |\pa_x|^\frac1q S^\pm(t) u_0 \big\|_{L_{t}^q (\mathbb{R};L_x^r(\mathbb{R}))} \ls \|u_0\|_{L^2(\mathbb{R})}.$$
$(ii)$ (Local smoothing estimate)
$$\| \pa_x S^\pm(t) u_0\|_{L_x^\infty(\mathbb{R}; L_t^2(\mathbb{R}))} \ls \|u_0\|_{L^2(\mathbb{R})}.$$
$(iii)$ (Maximal function estimate)
$$\| S^\pm(t) u_0\|_{L_x^2(\mathbb{R}; L_{t}^\infty([-1,1]))} \ls \|u_0\|_{H^s(\mathbb{R})}.$$
\end{proposition}

Let $X_\pm^{s,b}$ denote the Bourgain space with the norm 
\begin{equation}\label{eq:XsbKdV}
\|u \|_{X_\pm^{s,b}} := \|\la \xi \ra^{s} \la \tau \mp\tfrac{\xi^3}{24} \ra^{b}\tilde{u}(\tau,\xi)\|_{L_\tau^2L_\xi^2(\mathbb{R}\times\mathbb{R})}.
\end{equation}
By the transference principle, we have the following corollary. 
\begin{corollary}\label{cor:linear estimates for KdV flows} Let $s>\frac{3}{4}$ and $b>\frac{1}{2}$. Suppose that $2\leq q,r\leq\infty$ and $\frac2q+\frac1r=\frac12$. Then, we have 
$$\big\| |\pa_x|^\frac1q w \big\|_{L_{t}^q (\mathbb{R};L_x^r(\mathbb{R}))}+\| \pa_x w\|_{L_x^\infty(\mathbb{R}; L_t^2(\mathbb{R}))} \ls \|w\|_{X_\pm^{0,b}}$$
and
$$\|w\|_{L_x^2(\mathbb{R}; L_{t}^\infty([-1,1]))} \ls \|w\|_{X_\pm^{s,b}}.$$
\end{corollary}

Proposition \ref{prop:Strichartz} will be proved below by adapting the argument in \cite{KPV-1991} and the references therein. For the proof, we decompose the linear propagator into dyadic pieces,
$$S_h^\pm(t)=\sum_{N\leq1}S_h^\pm(t)P_N=S_h^\pm(t)P_{\le N_0}+\sum_{N_0< N\leq1}S_h^\pm(t)P_N,$$
where $N_0$ is the dyadic number introduced in Remark \ref{rem:low}. We denote the kernel of a dyadic piece of the linear propagator $S_h^\pm(t)P_N$ (resp., $S_h^\pm(t)P_{\le N_0}$) by $K_N^\pm(t,x)$ (resp., $K_{\le N_0}^\pm(t,x)$). By the discrete Fourier transform in addition to \eqref{eq:decomposition}, their kernels can be expressed as oscillatory integrals
$$K_N^\pm(t,x)=\frac{1}{2\pi}\int_{-\frac{\pi}{h}}^{\frac{\pi}{h}} e^{\pm \frac{it}{h^2}(\xi-\frac{2}{h}\sin(\frac{h\xi}{2}))+ix\xi}\psi_N(\xi)d\xi$$
and
$$K_{\le N_0}^\pm(t,x)=\frac{1}{2\pi}\int_{-\frac{\pi}{h}}^{\frac{\pi}{h}} e^{\pm \frac{it}{h^2}(\xi-\frac{2}{h}\sin(\frac{h\xi}{2}))+ix\xi}\phi(\xi)d\xi.$$
These kernels obey the following decay properties.

\begin{lemma}\label{Lem:Oscillartory integral}
$(i)$ For every $N=2^k\leq 1$, we have
$$|K_N^\pm(t,x)|\lesssim\left(\frac{h}{N|t|}\right)^{1/2}.$$
$(ii)$ Suppose that $|t|\leq 1$. Then, we have
\begin{equation}\label{ineq:I0}
|K_{\le N_0}^\pm(t,x)|\lesssim\frac{1}{1+x^2}.
\end{equation}
For every $N_0 < N \le 1$, we have
\begin{equation}\label{ineq:I}
|K_N^\pm(t,x)|\ls \left\{\begin{aligned}
&\frac{N}{h} &&\text{if } |x|\le \frac{h}{N},\\
&\left(\frac{N}{h|x|}\right)^{1/2} &&\text{if }  \frac{h}{N}\le |x| \le \frac{N^2}{h^2},\\
&\frac{h}{Nx^2} &&\text{if }|x|\ge \frac{N^2}{h^2}.
\end{aligned}\right.
\end{equation}
\end{lemma}
\begin{remark}
It is easy to see that $s_h(\xi)=\frac{1}{h^2}(\xi-\frac{2}{h}\sin(\frac{h\xi}{2}))$ is comparable with $\xi^3$ on the frequency domain $[-\frac{\pi}{h},\frac{\pi}{h})$ in the sense that
\begin{equation}\label{comparable symbol}
\begin{aligned}
s_h'(\xi)=\frac{1}{h^2}\Big(1-\cos\Big(\frac{h\xi}{2}\Big)\Big)\sim |\xi|^2,\quad |s_h''(\xi)|\sim |\xi|,\quad|s_h'''(\xi)|\sim 1.
\end{aligned}
\end{equation}
\end{remark}
\begin{proof}
$(i).$ The proof follows from the van der Corput lemma with $|(\pm ts_h(\xi)+x\xi)''|=|ts_h''(\xi)|\sim |t||\xi|\sim \frac{N}{h}|t|$ on the support of $\psi_N$.

$(ii).$ For \eqref{ineq:I0}, by integration by parts twice, we  write
$$K_{\le N_0}^\pm(t,x)=\int_{-\frac{\pi}{h}}^{\frac{\pi}{h}} \frac{1}{(ix)^2}(e^{ix\xi})''e^{\pm its_h(\xi)}\phi(\xi)d\xi=- \frac{1}{x^2}\int_{-\frac{\pi}{h}}^{\frac{\pi}{h}}e^{ix\xi}\big(e^{\pm its_h(\xi)}\phi(\xi)\big)''d\xi.$$
A straightforward computation shows that 
$$\begin{aligned}
|(e^{\pm its_h(\xi)}\phi(\xi))''|&=\big|\big(t^2 s_h'(\xi))^2\mp it s_h''(\xi)\big)\phi(\xi)\mp 2it s_h'(\xi)\phi'(\xi)-\phi''(\xi)\big|\\
&\lesssim (\xi^4+|\xi|)\phi(\xi)+\xi^2|\phi'(\xi)|+|\phi''(\xi)|<\infty.
\end{aligned}$$
Thus, it follows that $|K_{\le N_0}^\pm(t,x)|\lesssim\frac{1}{x^2}$. Together with the trivial bound $|K_{\le N_0}^\pm(t,x)|\lesssim1$, we obtain \eqref{ineq:I0}.

For \eqref{ineq:I}, it is obvious that $|K_N^\pm(t,x)|\lesssim\frac{N}{h}$. Suppose that $|x|\ge\frac{N^2}{h^2}\geq 1$. By integration by parts twice, we have
$$|K_N^\pm(t,x)|\leq\int_{-\frac{\pi}{h}}^{\frac{\pi}{h}} \bigg|\bigg(\frac{1}{x\pm ts_h'(\xi)}\Big(\frac{\psi_N(\xi)}{x\pm ts_h'(\xi)}\Big)'\bigg)'\bigg|d\xi.$$
A direct computation gives
$$\begin{aligned}
\bigg(\frac{1}{x\pm ts_h'(\xi)}\Big(\frac{\psi_N(\xi)}{x\pm ts_h'(\xi)}\Big)'\bigg)'&=\frac{\psi_N''(\xi)}{(x\pm t s_h'(\xi))^2} + \frac{3t^2\psi _N(\xi)(s_h''(\xi))^2}{(x\pm t s_h'(\xi))^4 }\\
&\quad\mp\frac{ts_h'''(\xi)\psi_N(\xi)+3t s_h''(\xi)\psi_N'(\xi)}{(x\pm t s_h'(\xi))^3}.
\end{aligned}$$
Note that in this case, $|x\pm t s_h'(\xi)|=|x|-\tfrac{2|t|}{h^2}\sin^2(\tfrac{h\xi}{4})\gtrsim|x|$ on $\textup{supp}\psi_N$. Thus, by \eqref{comparable symbol}, we obtain that 
$$\begin{aligned}
|K_N^\pm(t,x)| &\lesssim \int_{-\frac{\pi}{h}}^{\frac{\pi}{h}}\frac{|\psi_N''(\xi)|}{|x|^2} + \frac{\xi^2\psi_N(\xi)}{|x|^4}+ \frac{\psi_N(\xi)+|\xi||\psi_N'(\xi)|}{|x|^3}d\xi\\
&\lesssim \frac{h}{N|x|^2} + \left(\frac{N}{h}\right)^{3} \frac{1}{|x|^4} +  \frac{N}{h|x|^3}\lesssim\frac{h}{N |x|^2}.
\end{aligned}$$

It remains to consider the case $\frac{h}{N}\le |x| \le \frac{N^2}{h^2}$ for \eqref{ineq:I}. When $t=0$, by integration by parts, we obtain
\begin{equation}\label{eq:K_N0}
|K_N^\pm(0,x)|=\Big|\frac{1}{2\pi}\int_{-\frac{\pi}{h}}^{\frac{\pi}{h}} e^{ix\xi}\psi_N(\xi)d\xi\Big|\ls \frac{1}{|x|}\le\left(\frac{N}{h|x|}\right)^{1/2}.
\end{equation}
When $t\neq 0$, by simple change of variables, we may assume that $t,x>0$, since $s_h(\xi)=-s_h(-\xi)$ and $\psi_N$ is an even function. For $K_N^+(t,x)$, by integration by parts with $(x\xi+ts_h(\xi))'=x+\frac{2t}{h^2}\sin^2(\frac{h\xi}{4})\ge x$, we obtain
\begin{equation}\label{K_N^+ bound}
\begin{aligned}
|K_N^+ (t,x) | &\le  \int_{-\frac{\pi}{h}}^{\frac{\pi}{h}} \Big|\Big(\frac{ \psi_N(\xi)}{x+t s_h'(\xi)}\Big)'\Big| d\xi=\int_{-\frac{\pi}{h}}^{\frac{\pi}{h}} \Big|\frac{\psi_N'(\xi)}{x+t s_h'(\xi)}-\frac{ts_h''(\xi)\psi_N(\xi)}{(x+t s_h'(\xi))^2}\Big|d\xi  \\
&\lesssim\frac{1}{|x|}+\frac{N^2}{h^2|x|^2}\lesssim \left(\frac{N}{h|x|}\right)^{1/2}.
\end{aligned}
\end{equation}
For $K_N^-(t,x)$, we note that its phase function may have stationary points. Thus, by splitting the frequency domain $[-\tfrac{\pi}{h},\tfrac{\pi}{h})=\{\xi: |s_h'(\xi)|\le \frac{x}{2t} \} \cup \{\xi: |s_h'(\xi)|> \frac{x}{2t} \}=:\Omega_1\cup \Omega_2$, we decompose
$$K_N^-(t,x) =  \frac{1}{2\pi}\left(\int_{\Omega_1} + \int_{\Omega_2}\right)e^{i(x\xi-ts_h(\xi))} \psi_N(\xi)d\xi=:K_{N,1}^-(t,x)+K_{N,2}^-(t,x).$$
Note that the phase function does not have a stationary point in $\Omega_1$. More precisely, we have a lower bound 
$(x\xi-ts_h(\xi))'= x-ts_h'(\xi) \geq\frac{x}{2}$. Hence, by repeating \eqref{K_N^+ bound}, we obtain $|K_{N,1}^- (t,x) |\lesssim (\frac{N}{h|x|})^{1/2}$. For $K_{N,2}^-(t,x)$, we observe from \eqref{comparable symbol} that $|\xi|^2\sim s_h'(\xi)\geq \frac{x}{2t}$ on $\Omega_2$; consequently, $|(x\xi-ts_h(\xi))''|=t|s_h''(\xi)| \sim t|\xi|\geq (\frac{N}{h|x|})^{1/2}$. Thus, by the van der Corput lemma, we obtain $|K_{N,2}^- (t,x) |\lesssim (\frac{N}{h|x|})^{1/2}$. By combining all the above-mentioned results, we complete the proof.
\end{proof}

\begin{proof}[Proof of Proposition \ref{prop:Strichartz}]

$(i).$ We use Lemma \ref{Lem:Oscillartory integral} $(i)$ to obtain 
$$\|S_h^\pm(t)P_N u_{h,0}\|_{L_x^\infty}=\bigg\|h\sum_{y\in h\mathbb{Z}} K_N(t,x-y) u_{h,0}(y)\bigg\|_{L_x^\infty}\lesssim\left(\frac{h}{N|t|}\right)^{1/2} \|u_{h,0}\|_{L^1}.$$
By interpolating with the trivial inequality $\|S_h^\pm(t)P_N u_{h,0}\|_{L_x^2}\leq \|u_{h,0}\|_{L^2}$, it follows that  
$$\|S_h^\pm(t)P_N u_{h,0}\|_{L_x^r}\lesssim \left(\frac{h}{N|t|}\right)^{\frac{1}{2}-\frac{1}{r}}\|u_{h,0}\|_{L^{r'}}$$
for $2\leq r\leq\infty$, where $r'$ is the H\"older conjugate of $r$. Hence, a standard $TT^*$ argument \cite{KT-1998} yields 
$$\|S_h^\pm(t)P_Nu_{h,0}\|_{L_{t}^qL_x^r}\ls \left(\frac{N}{h}\right)^{-\frac{1}{q}} \|u_{h,0}\|_{L^2}.$$
The Littlewood-Paley theory (Lemma \ref{LP inequalities}) and the norm equivalence (Lemma \ref{Prop:norm equivalence}) enable us to show $(i)$ for all $\partial_h$, $\partial_h^+$ and $\nabla_h$.

$(ii).$ We follow the argument in \cite{CS-1988}. First, by changing the variable $\tau=\pm \frac{1}{h^2}(\xi-\frac{2}{h}\sin(\frac{h\xi}{2}))$ such that $d\tau=\pm \frac{1}{h^2}(1-\cos(\frac{h\xi}{2}))d\xi$, we write 
$$\begin{aligned}
\partial_h S_h^\pm(t) u_{h,0}(x)
&=\frac{1}{2\pi}\int_{-\frac{\pi}{h}}^{\frac{\pi}{h}}i\xi e^{\pm \frac{it}{h^2}(\xi-\frac{2}{h}\sin(\frac{h\xi}{2}))}e^{ix\xi} (\mathcal{F}_hu_{h,0})(\xi)d\xi\\
&=\frac{i}{2\pi}\int_{-\infty}^\infty e^{it\tau}\mathbf{1}_{[-\frac{\pi}{h},\frac{\pi}{h})}(\xi)e^{ix\xi} (\mathcal{F}_hu_{h,0})(\xi)\frac{\pm h^2\xi}{1-\cos(\frac{h\xi}{2})}d\tau,
\end{aligned}$$
where, in the last integral, $\xi=\xi(\tau)$ is a function of $\tau$. Thus, Plancherel's theorem yields
$$\| \partial_h S_h^\pm(t)u_{h,0}\|_{L_t^2}\sim \bigg\|\mathbf{1}_{[-\frac{\pi}{h},\frac{\pi}{h})}(\xi)e^{ix\xi} (\mathcal{F}_hu_{h,0})(\xi)\frac{\pm h^2\xi}{1-\cos(\frac{h\xi}{2})}\bigg\|_{L_\tau^2}.$$
By changing the variable back to $\tau=\pm \frac{1}{h^2}(\xi-\frac{2}{h}\sin(\frac{h\xi}{2}))\mapsto \xi$ and using $1-\cos(\frac{h\xi}{2})=2\sin^2(\frac{h\xi}{4})\sim h^2\xi^2$ on $[-\frac{\pi}{h},\frac{\pi}{h})$, we prove that 
$$\begin{aligned}
\| \partial_h S_h^\pm(t)u_{h,0}\|_{L_t^2}^2\sim \int_{-\frac{\pi}{h}}^{\frac{\pi}{h}} |(\mathcal{F}_hu_{h,0})(\xi)|^2\frac{h^2\xi^2}{1-\cos(\frac{h\xi}{2})} d\xi\sim \|u_{h,0}\|_{L^2}^2.
\end{aligned}$$
Similarly, one can show the inequality with $\pa_h^+$ (resp., $\nabla_h$) by replacing the symbol $i\xi$ for $\partial_h$ by $\frac{e^{ih\xi}-1}{h}$ (resp., $\frac{2i}{h}\sin(\frac{h\xi}{2})$).

$(iii).$ We claim that if $N_0< N\leq 1$, then
$$\|S_h^\pm(t)P_N u_{h,0}\|_{L_x^2L_{t}^\infty} \ls \left(\frac{N}{h}\right)^\frac34\| u_{h,0}\|_{L^2},$$
where the time interval $[-1,1]$ is omitted in the norm. Indeed, by a $TT^*$ argument, the claim is equivalent to 
$$\bigg\| \int_{-1}^1  S_h^\pm(t-t_1)P_N g(t_1) dt_1 \bigg\|_{L_x^2L_{t}^\infty}\le \left(\frac{N}{h}\right)^\frac32 \|g\|_{L_x^2L_{t}^1}.$$
We observe that 
$$\begin{aligned}
\bigg|\int_{-1}^1S_h^\pm(t-t_1)P_N g(t_1) dt_1\bigg|&=\bigg|h\sum_{x_1\in h\mathbb{Z}}\int_{-1}^1 K_N(t-t_1,x_1) g(t_1,\cdot-x_1)dt_1\bigg|\\
&\leq h\sum_{x_1\in h\mathbb{Z}}\|K_N(t,x_1)\|_{L_t^\infty} \|g(t,\cdot-x_1)\|_{L_t^1}.
\end{aligned}$$
Using \eqref{eq:young} and \eqref{ineq:I} (in particular $\|K_N(t,x)\|_{L_x^1L_t^\infty}\lesssim(\frac{N}{h})^{3/2}$), we conclude that
$$\bigg\| \int_{-1}^1S_h^\pm(t-t_1)P_N g(t_1) dt_1 \bigg\|_{L_x^2L_{t}^\infty}\leq\|K_N\|_{L_h^1L_t^\infty} \|g\|_{L_x^2L_t^1}\le \left(\frac{N}{h}\right)^\frac32 \|g\|_{L_x^2L_{t}^1}.$$
Similarly but by using \eqref{ineq:I0}, one has 
$$\big\|S_h^\pm(t)P_{\le N_0} u_{h,0} \big\|_{L_x^2 L_{t}^\infty} \ls \| u_{h,0}\|_{L^2}.$$
Summing them up for $s > \frac34$, we obtain
\begin{align*}
\|S_h^\pm(t)u_{h,0}\|_{L_x^2L_{t}^\infty}&\leq \|S_h^\pm(t)P_{\le N_0}u_{h,0}\|_{L_x^2L_{t}^\infty}+\sum_{N_0 < N\leq 1}\|S_h^\pm(t)P_N \tilde{P}_Nu_{h,0}\|_{L_x^2L_{t}^\infty}\\
&\lesssim \|u_{h,0}\|_{L^2}+\sum_{N_0 < N \leq 1}\left(\frac{N}{h}\right)^\frac34\|\tilde{P}_Nu_{h,0}\|_{L^2}\\
&\lesssim \Bigg\{1+\sum_{N_0 < N \leq 1}\left(\frac{N}{h}\right)^{-(s-\frac34)}\Bigg\}\|u_{h,0}\|_{H^s}\lesssim \|u_{h,0}\|_{H^s}.
\end{align*}
Therefore, we complete the proof.
\end{proof}

As mentioned in Remark \ref{rem:LS}, the time-averaged $L^\infty(h\Z)$ bound corresponding to the endpoint $(q,r)=(4,\infty)$ is excluded here. Nevertheless, together with the Sobolev inequality, we still have the following bound.
\begin{corollary}\label{Corendpoint}
If $\frac34<s<\frac32$ and $4<q<\frac{6}{3-2s}$, then 
\begin{equation}\label{Stri:endpoint}
\| \partial_h S_h^\pm(t)u_{h,0}\|_{L_{t}^{q}(\mathbb{R}; L_x^\infty(h\mathbb{Z}))}\ls \|u_{h,0}\|_{H^s(h\mathbb{Z})}.
\end{equation}
Here, $\partial_h$ can be replaced by $\pa_h^+$ and $\nabla_h$.
\end{corollary}

\begin{proof}
Let $r$ satisfy $\frac{2}{q}+\frac{1}{r}=\frac{1}{2}$. Then, it follows from the Sobolev inequality \eqref{Sobolev} and Strichartz estimates (Proposition \ref{prop:Strichartz} $(i)$) that 
$$\begin{aligned}
\|\partial_h S_h^\pm(t)u_{h,0}\|_{L_t^{q}L_x^{\infty}}
&\ls \| \la \pa_h\ra^{\frac1r+\delta} \partial_h S_h^\pm(t)u_{h,0}\|_{L_t^{q}L_x^{r}}\\
&\lesssim \|\la \pa_h\ra^{\frac1r+\delta} \partial_h |\partial_h|^{-\frac{1}{q}}u_{h,0}\|_{L^2}\lesssim \|\la \pa_h\ra^s u_{h,0}\|_{L^2}
\end{aligned}$$
with $\delta=s+\frac{3}{q}-\frac{3}{2}>0$. Thus, \eqref{Stri:endpoint} follows from the norm equivalence (Lemma \ref{Prop:norm equivalence}). Similarly, one can show the desired inequalities for different operators.
\end{proof}

\begin{remark}\label{rem:stri}
Let $\tilde S_h^{\pm}(t)$ denote another linear propagator with the kernel
$$\tilde{K}_N^\pm(t,x)=\frac{1}{2\pi}\int_{-\frac{\pi}{h}}^{\frac{\pi}{h}} e^{\pm \frac{it}{h^2}(\xi+\frac{2}{h}\sin(\frac{h\xi}{2}))+ix\xi}\psi_N(\xi)d\xi.$$
By repeating the proof of Lemma \ref{Lem:Oscillartory integral} $(i)$, one can show the same kernel estimate, 
$$|K_N^\pm(t,x)|\lesssim\left(\frac{h}{N|t|}\right)^{1/2}.$$
Consequently, the Strichartz estimates for the propagator $\tilde S_h^{\pm}(t)$ of the form in Proposition \ref{prop:Strichartz} $(i)$ follow (see the proof of Proposition \ref{prop:Strichartz} $(i)$ above). This will enable us to control the cubic term for the analysis on a general nonlinearity (see Appendix \ref{app:GP}).
\end{remark}

\subsection{Approximation of linear FPU flows by Airy flows}\label{sec:AFA}
We now compare the linear FPU flows with the Airy flows with respect to the space-time norm
\begin{equation}\label{mixed spacetime norm}
\| u\|_{S([-T,T])}:= \| u\|_{C_t([-T,T]; L_x^2(\mathbb{R}))} + \|\pa _x u\|_{L_x^\infty (\mathbb{R};L_t^2([-T,T]))}
\end{equation}
using the linear interpolation $l_h$ defined by \eqref{def:  linear interpolation}. 

\begin{proposition}[Comparison between linear FPU and Airy flows]\label{linear approx}
For $s\in(0,1]$, we have
$$\big\| l_h(S_h^\pm(t)u_{h,0})-S^\pm(t)(l_hu_{h,0})\big\|_{S([-1,1])}\lesssim h^{\frac{2s}{5}}\|u_{h,0}\|_{H^s(h\mathbb{Z})}.$$
\end{proposition}
\begin{remark}
We note that the linear interpolation can be regarded as a Fourier multiplier (see \cite[Lemma 5.5]{HY-2019SIAM}).
\end{remark}
\begin{lemma}[Symbol of linear interpolation operator]\label{p_h symbol}
The interpolation operator $l_h$ is a Fourier multiplier operator in the sense that
$$\mathcal{F}(l_h f_h)(\xi) = \mathcal{L}_h(\xi) \mathcal (\tilde{\mathcal{F}}_h f_h)(\xi),\quad\forall\xi\in\mathbb{R},$$
where
\begin{equation}\label{L_x symbol calculation}
\mathcal{L}_h(\xi)=\frac{1}{h}\int_0^h e^{-ix\xi} dx+\frac{e^{ih\xi}-1}{h^2}\int_0^h x e^{-ix\xi}dx=\frac{4\sin^2(\frac{h\xi}{2})}{h^2\xi^2}
\end{equation}
and $\tilde{\mathcal{F}}_h$ denotes the $[-\frac{\pi}{h},\frac{\pi}{h})$-periodic extension of the discrete Fourier transform $\mathcal{F}_h$, precisely, 
$$(\tilde{\mathcal{F}}_hf_h)(\xi)=(\mathcal{F}_hf_h)(\xi'),$$ where $\xi'=\xi-\frac{2\pi m}{h}\in[-\frac{\pi}{h},\frac{\pi}{h})$ for some $m\in\mathbb{Z}$.
\end{lemma}
A straightforward computation of the Fourier transform and Lemma \ref{p_h symbol} give 
\begin{equation}\label{linear interpolated linear FPU flow}
\begin{aligned}
l_h(S_h^\pm(t)u_{h,0})(x)&=\frac{1}{2\pi}\int_{-\infty}^\infty \mathcal L_h(\xi) (\tilde{\mathcal{F}}_hS_h^{\pm}(t)u_{h,0})(\xi) e^{ix\xi}d\xi\\
&=\frac{1}{2\pi}\int_{-\infty}^\infty \mathcal L_h(\xi) e^{\pm \frac{it}{h^2}(\xi'-\frac{2}{h}\sin(\frac{h\xi'}{2}))}(\mathcal{F}_hu_{h,0})(\xi') e^{ix\xi}d\xi\\
&=\tilde{S}_h^\pm(t)(l_hu_{h,0})(x),
\end{aligned}
\end{equation}
where the new linear propagator $\tilde{S}_h^\pm(t): L^2(\mathbb{R})\to L^2(\mathbb{R})$ is given by 
$$\mathcal{F}(\tilde{S}_h^\pm(t)u_0)(\xi)=e^{\pm \frac{it}{h^2}(\xi'-\frac{2}{h}\sin(\frac{h\xi'}{2}))}(\mathcal{F}u_0)(\xi)$$
with $\xi'=\xi-\frac{2\pi m}{h}\in[-\frac{\pi}{h},\frac{\pi}{h})$ for some $m\in\mathbb{Z}$.
 
\begin{proof}[Proof of Proposition \ref{linear approx}]

First, by \eqref{linear interpolated linear FPU flow}, we write
\begin{align*}
l_h (S_h^\pm(t)u_{h,0}) - S^\pm(t)(l_hu_{h,0})&=\tilde{S}_h^\pm(t)(l_h u_{h,0}) - S^\pm(t)(l_hu_{h,0})\\
&=(\tilde{S}_h^\pm(t)-S^\pm(t))P_{low}(l_h  u_{h,0})+\tilde{S}_h^\pm(t)P_{high}\big(l_h  u_{h,0}\big)\\
&\quad-S^\pm(t)P_{high}(l_h u_{h,0})\\
&=:I+II+III,
\end{align*}
where $P_{low}$ (resp., $P_{high}$) is the Fourier multiplier of the symbol $\mathbf{1}_{|\xi|\leq h^{-2/5}}$ (resp., $\mathbf{1}_{|\xi|\geq h^{-2/5}}$).

For $III$, we use Proposition \ref{Thm:linear estimates for KdV flows} and Lemma \ref{Lem:discretization linearization inequality} to obtain
$$\|III\|_{S([-1,1])}\lesssim\|P_{high}(l_h u_{h,0})\|_{L^2}\lesssim h^{\frac{2s}{5}}\|l_h u_{h,0}\|_{H^s}\lesssim h^{\frac{2s}{5}} \|u_{h,0}\|_{H^s}.$$

For $I$, we observe from \eqref{linear interpolated linear FPU flow} that
$$I=\frac{1}{2\pi}\int_{-\infty}^{\infty}\Big(e^{\pm \frac{it}{h^2}(\xi-\frac{2}{h}\sin(\frac{h\xi}{2}))}-e^{\pm \frac{it}{24}\xi^3}\Big) \mathbf{1}_{|\xi|\leq h^{-\frac{2}{5}}} \mathcal{F}_h(l_hu_{h,0})(\xi) e^{ix\xi} d\xi.$$
The fundamental theorem of calculus yields
$$e^{\pm \frac{it}{h^2}(\xi-\frac{2}{h}\sin(\frac{h\xi}{2}))}-e^{\pm \frac{it}{24}\xi^3}=\pm it\Big(\tfrac{\xi-\frac{2}{h}\sin(\frac{h\xi}{2})}{h^2}-\tfrac{\xi^3}{24}\Big)\int_0^1 e^{\pm it(\frac{\alpha}{h^2}(\xi-\frac{2}{h}\sin(\frac{h\xi}{2}))+\frac{1-\alpha}{24}\xi^3)} d\alpha.$$
Thus, introducing the linear propagator $S_{h,\alpha}^\pm(t)$ given by 
$$\mathcal{F}(S_{h,\alpha}^\pm(t) u_{0})(\xi)=e^{\pm it(\frac{\alpha}{h^2}(\xi-\frac{2}{h}\sin(\frac{h\xi}{2}))+\frac{1-\alpha}{24}\xi^3)}(\mathcal{F}u_{0})(\xi),$$ 
we write
$$\begin{aligned}
I &=\pm it\int_0^1S_{h,\alpha}^\pm(t)\mathcal{F}^{-1}\Big(\mathbf{1}_{|\xi|\leq h^{-\frac{2}{5}}}\Big(\tfrac{\xi-\frac{2}{h}\sin(\frac{h\xi}{2})}{h^2}-\tfrac{\xi^3}{24}\Big)\mathcal{F}_h(l_hu_{h,0})\Big)d\alpha.
\end{aligned}$$
Analogously, one can show for the new propagator $S_{h,\alpha}^\pm(t)$ that
$$\|S_{h,\alpha}^\pm(t)P_{low}u_0\|_{C_tL_x^2}\lesssim \|u_0\|_{L^2}$$
and
$$\|\partial_xS_{h,\alpha}^\pm(t)P_{low}u_0\|_{L_x^\infty L_t^2}\lesssim \|u_0\|_{L^2}.$$
These results, together with Plancherel's theorem, the Taylor series expansion $\tfrac{\xi-\frac{2}{h}\sin(\frac{h\xi}{2})}{h^2}-\tfrac{\xi^3}{24}=O(h^2\xi^5)$,  and Lemma \ref{Lem:discretization linearization inequality} yield
$$\begin{aligned}
\|I\|_{S([-1,1])}&\lesssim \Big\|\mathcal{F}^{-1}\Big(\mathbf{1}_{|\xi|\leq h^{-\frac{2}{5}}}\Big(\tfrac{\xi-\frac{2}{h}\sin(\frac{h\xi}{2})}{h^2}-\tfrac{\xi^3}{24}\Big)\mathcal{F}_h(l_hu_{h,0})\Big)\Big\|_{L^2}\\
&\lesssim h^2\Big\||\xi|^5\mathbf{1}_{|\xi|\leq h^{-\frac{2}{5}}}\mathcal{F}_h(l_hu_{h,0})\Big\|_{L_\xi^2}\leq h^2h^{-\frac{2}{5}(5-s)}\big\||\xi|^s\mathcal{F}(l_hu_{h,0})(\xi)\big\|_{L_\xi^2}\\
&\lesssim h^{\frac{2s}{5}}\|l_hu_{h,0}\|_{H^s}\lesssim h^{\frac{2s}{5}}\|u_{h,0}\|_{H^s}.
\end{aligned}$$
For $II$, by the unitarity of $\tilde{S}_h^\pm(t)$ and Lemma \ref{Lem:discretization linearization inequality}, we obtain
$$\|II\|_{C_tL_x^2}\lesssim \|P_{high}(l_h u_{h,0})\|_{L^2}\lesssim h^{\frac{2s}{5}}\|l_h u_{h,0}\|_{H^s}\lesssim h^{\frac{2s}{5}}\|u_{h,0}\|_{H^s}.$$

It remains to estimate $\|\partial_x II\|_{L_x^\infty L_t^2}$. By \eqref{linear interpolated linear FPU flow}, we write 
$$\begin{aligned}
II&=\frac{1}{2\pi}\int_{|\xi|\ge \frac{1}{h^{2/5}}}\mathcal{L}_h(\xi) e^{\pm \frac{it}{h^2}(\xi'-\frac{2}{h}\sin(\frac{h\xi'}{2}))}(\mathcal{F}_hu_{h,0})(\xi') e^{ix\xi}d\xi\\
&=\frac{1}{2\pi}\int_{-\infty}^\infty\mathbf 1_{\frac{1}{h^{2/5}}\leq|\xi|\leq\frac{\pi}{h}}\mathcal{L}_h(\xi) e^{\pm \frac{it}{h^2}(\xi-\frac{2}{h}\sin(\frac{h\xi}{2}))}(\mathcal{F}_hu_{h,0})(\xi) e^{ix\xi}d\xi\\
&\quad+\sum_{m \in \Z \setminus \{0\}} \frac{1}{2\pi}\int_{-\frac\pi h}^{\frac\pi h}\mathcal{L}_h(\xi+\tfrac{2m\pi}{h}) e^{\pm \frac{it}{h^2}(\xi-\frac{2}{h}\sin(\frac{h\xi}{2}))}(\mathcal{F}_hu_{h,0})(\xi) e^{ix(\xi+\frac{2m\pi}{h})}d\xi\\
&=: II_0 + II_{\neq 0}.
\end{aligned}$$
For $II_0$, we apply the local smoothing estimate (Proposition \ref{prop:Strichartz} $(ii)$) to obtain 
\begin{equation}\begin{aligned}\label{m=0}
\| \pa_x II_0\|_{L_x^\infty L_t^2}&=\Big\|\pa_x\tilde{S}_h^\pm(t)\mathcal{F}_x^{-1}\Big(\mathbf 1_{\frac{1}{h^{2/5}}\leq |\xi|\leq\frac{\pi}{h}}  \mathcal L_x(\xi) (\mathcal{F}_hu_{h,0})(\xi)\Big)\Big\|_{L_x^\infty L_t^2}\\
&\ls \Big\|\mathcal{F}_x^{-1}\Big(\mathbf 1_{\frac{1}{h^{2/5}}\leq |\xi|\leq\frac{\pi}{h}}  \mathcal L_x(\xi) (\mathcal{F}_hu_{h,0})(\xi)\Big)\Big\|_{L^2}\ls h^{\frac25 s} \|u_{h,0}\|_{H^s}.
\end{aligned}
\end{equation}
Now, we consider $II_{\neq0}$. A direct computation with \eqref{L_x symbol calculation} gives
\[\pa_x II_{\neq 0}(t,x) = \frac{1}{2\pi}\int_{-\frac\pi h}^{\frac\pi h} e^{it(\pm\frac{1}{h^2}(\xi-\frac{2}{h}\sin(\frac{h\xi}{2})))} e^{ix\xi} \sum_{m \neq 0}\frac{4\sin^2(\frac{h\xi}{2})}{h(2m\pi+h\xi)} e^{i\frac{2m\pi x}{h}} (\mathcal{F}_hu_{h,0})(\xi) d\xi.\]
Analogously to the proof of Proposition \ref{prop:Strichartz} $(ii)$, one can show that 
$$\begin{aligned}
\|\pa_x II_{\neq 0}\|_{L_t^2}&\ls\Bigg\| \frac{h}{\sqrt{1-\cos(\frac{h\xi}{2})}}\sum_{m\neq0}\frac{4\sin^2(\frac{h\xi}{2})}{h(2m \pi+h\xi)} e^{i\frac{2m\pi x}{h}} (\mathcal{F}_hu_{h,0})(\xi)\Bigg\|_{L_\xi^2([-\frac{\pi}{h},\frac{\pi}{h}))}\\
&\lesssim\Bigg\| \Bigg(\sum_{m \neq0}\frac{e^{i\frac{2m\pi x}{h}}}{2m\pi +h\xi}\Bigg) \sin\left(\frac{h\xi}{2}\right) (\mathcal{F}_hu_{h,0})(\xi)\Bigg\|_{L_\xi^2([-\frac{\pi}{h},\frac{\pi}{h}))}.
\end{aligned}$$
Note that 
\[\Bigg|\sum_{m\neq0}\frac{e^{imx}}{m}\Bigg| = 2\Bigg|\sum_{m=1}^{\infty}\frac{\sin(mx)}{m}\Bigg| \lesssim \Bigg|\int_0^\infty \frac{\sin (xt)}{t} \; dt\Bigg| < \infty, \quad \mbox{uniformly in } x,\]
which implies that
$$\left|\sum_{m\neq0} \frac{e^{i\frac{2m\pi x}{h}}}{2m\pi +h\xi}\right|=\left|\sum_{m\neq0}\frac{e^{i\frac{2m\pi x}{h}}}{2m\pi}+\sum_{m\neq0} \bigg(\frac{e^{i\frac{2m\pi x}{h}}}{2m\pi +h\xi}-\frac{e^{i\frac{2m\pi x}{h}}}{2m\pi}\bigg)\right|\lesssim 1+\sum_{m \neq0}\frac{1}{m^2} < \infty,
$$
uniformly in $x \in \R$. Thus, we prove that
$$\|\pa_x II_{\neq 0}\|_{L_t^2}\ls \| \sin(\tfrac{h\xi}{2}) (\mathcal{F}_hu_{h,0})(\xi)\|_{L_\xi^2([-\frac{\pi}{h},\frac{\pi}{h}))}\ls h^s\| u_{h,0}\|_{H^s},$$
which, in addition to \eqref{m=0}, implies that
$$\| \pa_x II\|_{L_x^\infty L_t^2}\ls h^{\frac{2s}{5}}\| u_{h,0}\|_{H^s}.$$
By combining all the results, we complete the proof of the proposition.
\end{proof}

\section{Bilinear estimates}\label{sec: bilinear estimates}
In this section, we prove a series of $X^{s,b}$ bilinear estimates, which are the key estimates in our analysis.

\begin{lemma}[Bilinear estimate I]\label{lem:bi1}
For $s \ge 0$, there exist $b = b(s) > \frac{1}{2}$ and $\delta = \delta(b)>0$ such that
\begin{equation}\label{eq:bi1}
\big\|\nabla_h (u_h^\pm\cdot v_h^\pm) \big\|_{X_{h,\pm}^{s,-(1-b-\delta)}} \ls \|u_h^\pm\|_{X_{h,\pm}^{s,b}}\|v_h^\pm\|_{X_{h,\pm}^{s,b}},
\end{equation}
for any $u_h^{\pm}, v_h^{\pm} \in X_{h,\pm}^{s,b}$.
\end{lemma}

The following elementary integral estimates will be employed.
\begin{lemma}[Lemma 2.3 in \cite{KPV-1996}]\label{lem:elementary}
Let $\alpha, \beta \in \R$. For $b > \frac12$, we have 
\begin{equation}\label{eq:integral1}
\int_{-\infty}^\infty \frac{dx}{\la x-\alpha \ra^{2b}\la x - \beta \ra^{2b}} \lesssim \frac{1}{\la \alpha-\beta \ra^{2b}}
\end{equation}
and
\begin{equation}\label{eq:integral2}
\int_{-\infty}^\infty \frac{1}{\la x\ra^{2b}\sqrt{|x-\beta|}}\ls \frac{1}{\la\beta\ra^\frac12}.
\end{equation}
\end{lemma}
%

\begin{proof}[Proof of Lemma \ref{lem:bi1}]
We prove Lemma \ref{lem:bi1} only for the $\|\nabla_h (u_h^+v_h^+) \|_{X_{h,+}^{s,-(1-b-\delta)}}$ case, otherwise, an analogous argument is applicable.

By Parseval's identity, we write 
$$\begin{aligned}
&\int_{-\infty}^\infty\sum_{x\in h\mathbb{Z}}\nabla_h (u_h^+ v_h^+)(t,x)\overline{w_h(t,x)}\\
&\sim\iiiint \tfrac{2i}{h}\sin(\tfrac{h\xi}{2}) \tilde{u}_h^+(\tau_1,\xi_1) \tilde{v}_h^+(\tau-\tau_1,\xi-\xi_1)\overline{\tilde{w}_h(\tau,\xi)} d\xi_1 d \xi  d\tau_1  d\tau,
\end{aligned}$$
where $\tilde{u}$ is the space-time Fourier transform, and the intervals of integration are omitted for notational convenience. By the symmetry and duality, it suffices to show that 
$$\begin{aligned}
&\bigg|\iiiint \tfrac{2i}{h}\sin(\tfrac{h\xi}{2}) \tilde{u}_h^+(\tau_1,\xi_1) \tilde{v}_h^+(\tau-\tau_1,\xi-\xi_1)\overline{\tilde{w}_h(\tau,\xi)} d\xi_1 d \xi  d\tau_1  d\tau\bigg|\\
&\lesssim  \|u_h^+\|_{X_{h,+}^{s,b}}\|v_h^+\|_{X_{h,+}^{0,b}}\|w_h\|_{X_{h,+}^{-s, 1-b-\delta}},
\end{aligned}$$
which is equivalent to 
$$\begin{aligned}
&\bigg|\iiiint\frac{\frac{2i}{h}\sin(\frac{h\xi}{2})\la \xi \ra^s F(\tau_1,\xi_1)G(\tau-\tau_1,\xi-\xi_1) \overline{W(\tau,\xi)}d\xi_1d \xi d\tau_1  d\tau}{\la \xi_1 \ra^{s}\la \xi-\xi_1 \ra^{s}\la \tau_1-s_h(\xi_1) \ra^{b} \la \tau-\tau_1-s_h(\xi-\xi_1) \ra^{b} \la \tau - s_h(\xi) \ra^{1-b - \delta}}\bigg|\\
&\lesssim  \|F\|_{L_{\tau,\xi}^2}\|G\|_{L_{\tau,\xi}^2}\|W\|_{L_{\tau,\xi}^2},
\end{aligned}$$
where
$$s_h(\xi):=\tfrac{1}{h^2}(\xi-\tfrac{2}{h}\sin(\tfrac{h\xi}{2}))$$
and
$$\left\{\begin{aligned}
&F(\tau,\xi) =\la \xi \ra^s\la \tau-s_h(\xi) \ra^{b} \tilde{u}_h^+(\tau,\xi),\\
&G(\tau,\xi) =\la \tau-s_h(\xi) \ra^{b} \tilde{v}_h^+(\tau,\xi),\\
&W(\tau,\xi) =\la \xi \ra^{-s}\la \tau-s_h(\xi) \ra^{1-b-\delta} \tilde{w}_h(\tau,\xi).
\end{aligned}\right.$$
Hence, by the trivial inequality 
\begin{equation}\label{simple xi bound}
\frac{\la \xi \ra^s}{\la \xi_1 \ra^{s}\la \xi-\xi_1 \ra^{s}} \lesssim 1
\end{equation}
the Cauchy-Schwarz inequality for the $\tau_1$- and $\xi_1$- variables and  \eqref{eq:integral1}, we have
$$\begin{aligned}
&\bigg|\iiiint \frac{\frac{2i}{h}\sin(\frac{h\xi}{2})\la \xi \ra^s F(\tau_1,\xi_1)G(\tau-\tau_1,\xi-\xi_1) \overline{W(\tau,\xi)}d\xi_1d \xi d\tau_1  d\tau}{\la \xi_1 \ra^{s}\la \tau_1-s_h(\xi_1) \ra^{b} \la \tau-\tau_1-s_h(\xi-\xi_1) \ra^{b} \la \tau - s_h(\xi) \ra^{1-b - \delta}}\bigg|\\
&\lesssim \iiiint\frac{|\frac{2}{h}\sin(\frac{h\xi}{2})| |F(\tau_1,\xi_1)||G(\tau-\tau_1,\xi-\xi_1)||W(\tau,\xi)|}{\la \tau_1-s_h(\xi_1) \ra^{b} \la \tau-\tau_1-s_h(\xi-\xi_1) \ra^{b} \la \tau - s_h(\xi) \ra^{1-b - \delta}} d\xi_1  d \xi d\tau_1 d\tau\\
&\lesssim \Bigg(\sup_{\tau \in \R}\sup_{|\xi|\leq\frac{\pi}{h}} \iint\frac{\frac{4}{h^2}\sin^2(\frac{h\xi}{2})}{\la \tau_1-s_h(\xi_1) \ra^{2b} \la \tau-\tau_1- s_h(\xi-\xi_1)\ra^{2b}\la \tau - s_h(\xi) \ra^{2(1-b-\delta)}}d\tau_1d\xi_1\Bigg)^{\frac{1}{2}}\\
&\quad\quad \cdot\|F\|_{L_{\tau,\xi}^2}\|G\|_{L_{\tau,\xi}^2}\|W\|_{L_{\tau,\xi}^2}\\
&\lesssim\Bigg(\sup_{\tau \in \R}\sup_{|\xi|\leq\frac{\pi}{h}}\int\frac{\frac{4}{h^2}\sin^2(\frac{h\xi}{2})}{\la \tau- s_h(\xi_1)-s_h(\xi-\xi_1)\ra^{2b}\la \tau - s_h(\xi) \ra^{2(1-b-\delta)}}d\xi_1\Bigg)^{\frac{1}{2}}\\
&\quad\quad \cdot\|F\|_{L_{\tau,\xi}^2}\|G\|_{L_{\tau,\xi}^2}\|W\|_{L_{\tau,\xi}^2}.
\end{aligned}$$
Therefore, it suffices to show that 
\begin{equation}\label{eq:bi1_1}
\sup_{\tau \in \R}\sup_{|\xi|\leq\frac{\pi}{h}}\frac{\frac{4}{h^2}\sin^2(\frac{h\xi}{2})}{\la \tau - s_h(\xi) \ra^{2(1-b-\delta)}}  \int_{-\frac{\pi}{h}}^{\frac{\pi}{h}} \frac{d\xi_1}{ \la \tau-s_h(\xi_1) - s_h(\xi-\xi_1) \ra^{2b}} \ls 1.
\end{equation}

Note that the left-hand side of \eqref{eq:bi1_1} vanishes when $\xi = 0$. In what follows, we assume that $\xi \neq 0$. By the symmetry $s_h(-\xi)=-s_h(\xi)$, we may assume that $\xi>0$. To show \eqref{eq:bi1_1}, by the sum-to-product rule for sine functions, we write
$$\begin{aligned}
I_{\tau,\xi}:&= \int_{-\frac{\pi}{h}}^{\frac{\pi}{h}} \frac{d\xi_1}{ \la \tau-s_h(\xi_1) - s_h(\xi-\xi_1) \ra^{2b}}=\int_{-\frac{\pi}{h}}^{\frac{\pi}{h}}\frac{d\xi_1}{\langle \tau -\frac{\xi}{h^2}+\frac{2}{h^3}(\sin(\frac{h\xi_1}{2})+\sin(\frac{h(\xi-\xi_1)}{2})) \rangle^{2b}}\\
&=\int_{-\frac{\pi}{h}}^{\frac{\pi}{h}}\frac{d\xi_1}{\langle \tau - \frac{\xi}{h^2} + \frac{4}{h^3} \sin(\frac{h\xi}{4})\cos(\frac{h(\xi-2\xi_1)}{4}) \rangle^{2b}}.
\end{aligned}$$
Then, by changing the variables $\frac{h(\xi-2\xi_1)}{4}\mapsto\xi_1$ and since $\cos\xi_1$ is an even function, it follows that
\begin{equation}\label{eq:CoV1}
\begin{aligned}
I_{\tau,\xi}&=\frac{2}{h}\int_{-\frac{\pi}{2}+\frac{h\xi}{4}}^{\frac{\pi}{2}+\frac{h\xi}{4}}\frac{d\xi_1}{\langle\tau-\frac{\xi}{h^2}+\frac{4}{h^3}\sin(\frac{h\xi}{4})\cos\xi_1\rangle^{2b}}\\
&=\frac{2}{h}\left(\int_0^{\frac{\pi}{2}+\frac{h\xi}{4}}+\int_{-\frac{\pi}{2}+\frac{h\xi}{4}}^0\right)\frac{d\xi_1}{\langle\tau-\frac{\xi}{h^2}+\frac{4}{h^3}\sin(\frac{h\xi}{4})\cos\xi_1\rangle^{2b}}\\
&\leq\frac{4}{h}\int_0^{\frac{\pi}{2}+\frac{h\xi}{4}}\frac{d\xi_1}{\langle\tau-\frac{\xi}{h^2}+\frac{4}{h^3}\sin(\frac{h\xi}{4})\cos\xi_1\rangle^{2b}}.
\end{aligned}
\end{equation}
Since $\cos \xi_1$ is invertible in the interval $[0,\pi]$, changing the variable $\mu=\frac{4}{h^3}\sin(\frac{h\xi}{4})\cos\xi_1$ with
$$\begin{aligned}
\frac{d\mu}{d\xi_1}&=-\frac{4}{h^3}\sin\left(\frac{h\xi}{4}\right)\sin\xi_1=-\frac{4}{h^3}\sin\left(\frac{h\xi}{4}\right)\sqrt{1-\cos^2\xi_1}\\
&=-\sqrt{\tfrac{4}{h^3}\sin(\tfrac{h\xi}{4})-\mu}\sqrt{\tfrac{4}{h^3}\sin(\tfrac{h\xi}{4})+\mu}
\end{aligned}$$
yields
$$\begin{aligned}
I_{\tau,\xi}&\leq\frac{4}{h}\int_{-\frac{4}{h^3}\sin^2\frac{h\xi}{4}}^{\frac{4}{h^3}\sin(\frac{h\xi}{4})}\frac{d\mu}{\langle\tau-\frac{\xi}{h^2}+\mu\rangle^{2b}\sqrt{\tfrac{4}{h^3}\sin(\tfrac{h\xi}{4})-\mu}\sqrt{\tfrac{4}{h^3}\sin(\tfrac{h\xi}{4})+\mu}}\\
&\leq \frac{4}{h\sqrt{\tfrac{4}{h^3}\sin(\frac{h\xi}{4})-\frac{4}{h^3}\sin^2\frac{h\xi}{4}}}\int_{-\infty}^{\frac{4}{h^3}\sin(\frac{h\xi}{4})}\frac{d\mu}{\langle\tau-\frac{\xi}{h^2}+\mu\rangle^{2b}\sqrt{\tfrac{4}{h^3}\sin(\tfrac{h\xi}{4})-\mu}}\\
&\leq \frac{4}{\sqrt{\tfrac{4}{h}\sin(\frac{h\xi}{4})(1-\sin(\frac{h\xi}{4}))}}\int_{-\infty}^\infty\frac{d\mu}{\langle \mu\rangle^{2b}\sqrt{|\mu-(\tau-\frac{\xi}{h^2}+\tfrac{4}{h^3}\sin(\tfrac{h\xi}{4}))|}}.
\end{aligned}$$
Next, we apply the inequality \eqref{eq:integral2} together with $|1-\sin(\frac{h\xi}{4})| \sim 1$ for all $\xi \in(0, \frac{\pi}{h})$ to obtain
$$I_{\tau,\xi}\lesssim\frac{1}{\sqrt{\frac{4}{h}\sin(\frac{h\xi}{4})}}\frac{1}{\langle\tau-\frac{\xi}{h^2}+\tfrac{4}{h^3}\sin(\tfrac{h\xi}{4})\rangle^{1/2}}.$$
Coming back to \eqref{eq:bi1_1}, we insert the bound
$$\begin{aligned}
&\frac{\frac{4}{h^2}\sin^2(\frac{h\xi}{2})}{\la \tau - s_h(\xi) \ra^{2(1-b-\delta)}}I_{\tau,\xi} \lesssim\frac{\frac{4}{h^2}\sin^2(\frac{h\xi}{2})}{\sqrt{\frac{4}{h}\sin(\frac{h\xi}{4})}\la \tau - \frac{\xi}{h^2}+\frac{2}{h^3}\sin(\frac{h\xi}{2}) \ra^{2(1-b-\delta)}\langle\tau-\frac{\xi}{h^2}+\tfrac{4}{h^3}\sin(\tfrac{h\xi}{4})\rangle^{1/2}}.
\end{aligned}$$
Note that
\begin{equation}\label{eq:abz}
\left\la \alpha - \beta \right\ra^{\frac12} \lesssim \la \zeta - \alpha \ra^{\frac12}\la \zeta - \beta \ra^{\frac12} \lesssim \la \zeta - \alpha \ra^{2(1-b-\delta)}\la \zeta - \beta \ra^{\frac12}, \quad \forall \zeta, \alpha, \beta \in \R, 
\end{equation}
whenever $0 < \delta < \frac14$ and $\frac12 < b < \frac34 - \delta$. Thus, taking $\zeta = \tau - \frac{\xi}{h^2}$, $\alpha =- \frac{2}{h^3}\sin(\frac{h\xi}{2})$ and $\beta=-\frac{4}{h^3}\sin(\frac{h\xi_1}{4})$ in \eqref{eq:abz}, and using trigonometric identities, 
we conclude that
\begin{equation}\label{eq:bi1_3}
\begin{aligned}
\frac{\frac{4}{h^2}\sin^2(\frac{h\xi}{2})}{\la \tau - s_h(\xi) \ra^{2(1-b-\delta)}}I_{\tau,\xi}&\lesssim\frac{\frac{4}{h^2}\sin^2(\frac{h\xi}{2})}{\sqrt{\frac{4}{h}\sin(\frac{h\xi}{4})}\sqrt{\tfrac{4}{h^3}\sin(\tfrac{h\xi}{4})- \frac{2}{h^3}\sin(\frac{h\xi}{2})}}\\
&=\frac{\frac{16}{h^2}\sin^2\frac{h\xi}{4}\cos^2(\frac{h\xi}{4})}{\sqrt{\frac{4}{h}\sin(\frac{h\xi}{4})}\sqrt{\tfrac{4}{h^3}\sin(\tfrac{h\xi}{4})- \frac{4}{h^3}\sin(\frac{h\xi_1}{4})\cos(\frac{h\xi}{4})}}\\
&=\frac{\frac{4}{h}\sin(\frac{h\xi}{4})\cos^2(\frac{h\xi}{4})}{\frac{1}{h}\sqrt{1- \cos(\frac{h\xi}{4})}}=\frac{\frac{8}{h}\sin(\frac{h\xi}{8})\cos\frac{h\xi}{8}\cos^2(\frac{h\xi}{4})}{\frac{1}{h}\sqrt{2}\sin(\frac{h\xi}{8})}\\
&=4\sqrt{2}\cos\left(\frac{h\xi}{8}\right)\cos^2\left(\frac{h\xi}{4}\right)\lesssim1,
\end{aligned}
\end{equation}
which proves the desired bound \eqref{eq:bi1_1}.
\end{proof}

\begin{lemma}[Bilinear estimate II]\label{lem:bi2}
For $s \ge 0$, there exist $b = b(s) > 1/2$ and $\delta = \delta(b)>0$ such that if $s\leq s'\leq s+1$, then 
\begin{equation}\label{eq:bi2.1}
\Big\|\nabla_h \big(e^{\pm\frac{2t}{h^2}\partial_h}u_h^\mp \cdot e^{\pm\frac{2t}{h^2}\partial_h}v_h^\mp\big)\Big\|_{X_{h,\pm}^{s,b-1+\delta}}\ls h^{s'-s}\|u_h^\mp\|_{X_{h,\mp}^{s',b}}\|v_h^\mp\|_{X_{h,\mp}^{s',b}}.
\end{equation}
\end{lemma}

\begin{proof}
We consider the case of $\|\nabla_h(e^{-\frac{2t}{h^2}\partial_h}u_h^+ \cdot e^{-\frac{2t}{h^2}\partial_h}v_h^+)\|_{X_{h,-}^{s,-(1-b-\delta)}}$ only. The proof closely follows from that of Lemma \ref{lem:bi1}. By Parseval's identity, we write 
$$\begin{aligned}
&\int_{-\infty}^\infty\sum_{x\in h\mathbb{Z}}\nabla_h \big(e^{-\frac{2t}{h}\partial_h}u_h^+ \cdot e^{-\frac{2t}{h}\partial_h}v_h^+\big)(t,x)\overline{w_h(t,x)}\\
&\sim\iiiint\frac{\frac{2i}{h}\sin(\frac{h\xi}{2})\la \xi \ra^s F(\tau_1,\xi_1)G(\tau-\tau_1,\xi-\xi_1) \overline{W(\tau,\xi)}d\xi_1  d \xi d\tau_1d\tau}{\la \xi_1 \ra^{s'}\la \xi-\xi_1 \ra^{s'}\la \tau_1+\tfrac{2\xi_1}{h^2}-s_h(\xi_1) \ra^{b} \la \tau-\tau_1+\tfrac{2(\xi-\xi_1)}{h^2}-s_h(\xi-\xi_1) \ra^{b} \la \tau + s_h(\xi) \ra^{1-b - \delta}}.
\end{aligned}$$
where 
$$\left\{\begin{aligned}
&F(\tau,\xi) =\la \xi \ra^s\la \tau+\tfrac{2\xi}{h^2}-s_h(\xi) \ra^{b} \tilde{u}_h^+(\tau+\tfrac{2\xi}{h^2},\xi),\\
&G(\tau,\xi) =\la \tau+\tfrac{2\xi}{h^2}-s_h(\xi) \ra^{b} \tilde{v}_h^+(\tau+\tfrac{2\xi}{h^2},\xi),\\
&W(\tau,\xi) =\la \xi \ra^{-s}\la \tau+s_h(\xi) \ra^{1-b-\delta} \tilde{w}_h(\tau,\xi)
\end{aligned}\right.$$
and the second integral has the same structure. Hence, by repeating the reduction to \eqref{eq:bi1_1} but using $\frac{\la \xi \ra^{s'}}{\la \xi_1 \ra^{s'}\la \xi-\xi_1 \ra^{s'}} \lesssim 1$ instead of \eqref{simple xi bound}, one can reduce the proof of Lemma \ref{lem:bi2} to get a uniform bound for
\begin{equation}\label{eq: bi2 proof}
\frac{\frac{4}{h^2}\sin^2(\frac{h\xi}{2})}{\langle \xi\rangle^{s'-s}\la \tau + s_h(\xi) \ra^{2(1-b-\delta)}}\int_{-\frac{\pi}{h}}^{\frac{\pi}{h}} \frac{d\xi_1}{\la \tau +\frac{\xi}{h^2}+ \frac{2}{h^3}(\sin(\frac{h\xi_1}{2})+\sin(\frac{h(\xi-\xi_1)}{2}))\ra^{2b}}\lesssim h^{s'-s}
\end{equation}
for all $|\xi|\leq\frac{\pi}{h}$ and $\tau \in \R$. We may assume that $\xi>0$. We denote the integral in \eqref{eq: bi2 proof} by $I_{\tau,\xi}$. Then, by following the argument in the proof of Lemma \ref{lem:bi1}, we write
$$\begin{aligned}
I_{\tau,\xi}:&=\int_{-\frac{\pi}{h}}^{\frac{\pi}{h}} \frac{d\xi_1}{\la \tau +\frac{\xi}{h^2}+ \frac{4}{h^3}\sin(\frac{h\xi}{4})\cos(\frac{h(\xi-2\xi_1)}{4}\ra^{2b})}\\
&=\frac{2}{h}\int_{-\frac{\pi}{2}+\frac{h\xi}{4}}^{\frac{\pi}{2}+\frac{h\xi}{4}} \frac{d\xi_1}{\la \tau +\frac{\xi}{h^2}+ \frac{4}{h^3}\sin(\frac{h\xi}{4})\cos\xi_1\ra^{2b}}\\
&\leq\frac{4}{h}\int_0^{\frac{\pi}{2}+\frac{h\xi}{4}} \frac{d\xi_1}{\la \tau +\frac{\xi}{h^2}+ \frac{4}{h^3}\sin(\frac{h\xi}{4})\cos\xi_1\ra^{2b}}.
\end{aligned}$$
Next, changing the variable $\mu= \frac{4}{h^3}\sin(\frac{h\xi}{4})\cos\xi_1$ yields
$$\begin{aligned}
I_{\tau,\xi}&\leq \frac{4}{h}\int_{-\frac{4}{h^3}\sin^2\frac{h\xi}{4}}^{\frac{4}{h^3}\sin(\frac{h\xi}{4})} \frac{d\mu}{\la \tau +\frac{\xi}{h^2}+\mu\ra^{2b}\sqrt{\frac{4}{h^3}\sin(\frac{h\xi}{4})-\mu}\sqrt{\frac{4}{h^3}\sin(\frac{h\xi}{4})+\mu}}\\
&\leq \frac{4}{h\sqrt{\frac{4}{h^3}\sin(\frac{h\xi}{4})(1-\sin(\frac{h\xi}{4}))}}\int_{-\infty}^{\frac{4}{h^3}\cos(\frac{h\xi}{4})} \frac{d\mu}{\la \tau +\frac{\xi}{h^2}+\mu\ra^{2b}\sqrt{\frac{4}{h^3}\sin(\frac{h\xi}{4})-\mu}}\\
&= \frac{4}{\sqrt{\frac{4}{h}\sin(\frac{h\xi}{4})(1-\sin(\frac{h\xi}{4}))}}\int_{-\infty}^{\infty} \frac{d\mu}{\la\mu\ra^{2b}\sqrt{|\mu-(\tau+\frac{\xi}{h^2}+\frac{4}{h^3}\sin(\frac{h\xi}{4}))|}}\\
&\lesssim \frac{1}{\sqrt{\frac{4}{h}\sin(\frac{h\xi}{4})}}\frac{1}{\langle\tau+\frac{\xi}{h^2}+\frac{4}{h^3}\sin(\frac{h\xi}{4})\rangle^{1/2}}\quad\textup{(by \eqref{eq:integral2})}.
\end{aligned}$$
Thus, by inserting this bound in \eqref{eq: bi2 proof}, we prove that 
\begin{equation}\label{eq:h-bound}
\begin{aligned}
&\frac{\frac{4}{h^2}\sin^2(\frac{h\xi}{2})}{\langle \xi\rangle^{s'-s}\la \tau + s_h(\xi) \ra^{2(1-b-\delta)}}I_{\tau,\xi}\\
&\lesssim \frac{\frac{4}{h^2}\sin^2(\frac{h\xi}{2})}{\langle \xi\rangle^{s'-s}\sqrt{\frac{4}{h}\sin(\frac{h\xi}{4})}\la \tau + \frac{\xi}{h^2}-\frac{2}{h^3}\sin(\frac{h\xi}{2})  \ra^{2(1-b-\delta)}\langle\tau+\frac{\xi}{h^2}+\frac{4}{h^3}\sin(\frac{h\xi}{4})\rangle^{1/2}}\\
&\lesssim \frac{\frac{4}{h^2}\sin^2(\frac{h\xi}{2})}{\langle \xi\rangle^{s'-s}\sqrt{\frac{4}{h}\sin(\frac{h\xi}{4})}\sqrt{\frac{2}{h^3}\sin(\frac{h\xi}{2})+\frac{4}{h^3}\sin(\frac{h\xi}{4})}}\quad\textup{(by \eqref{eq:abz})}\\
&=\frac{\frac{16}{h^2}\sin^2(\frac{h\xi}{4})\cos^2(\frac{h\xi}{4})}{\langle \xi\rangle^{s'-s}\sqrt{\frac{4}{h}\sin(\frac{h\xi}{4})}\sqrt{\frac{4}{h^3}\sin(\frac{h\xi}{4})(1+\cos(\frac{h\xi}{4}))}}\sim\frac{\sin(\frac{h\xi}{4})}{\langle \xi\rangle^{s'-s}}.
\end{aligned}
\end{equation}
Since $\sin(\frac{h\xi}{4})\leq \min(\frac{h\xi}{4},1)\leq (\frac{h\xi}{4})^{s'-s}$, we prove \eqref{eq: bi2 proof}.
\end{proof}


\begin{lemma}[Bilinear estimate III]\label{lem:bi3}
For $s \ge 0$, there exist $b = b(s) > 1/2$ and $\delta = \delta(b)>0$ such that if $s\leq s'\leq s+1$, then 
\begin{equation}\label{eq:bi3}
\Big\| \nabla_h \big(u_h^\pm\cdot e^{\pm\frac{2t}{h^2}\partial_h}v_h^\mp\big) \Big\|_{X_{h,\pm}^{s,b-1+\delta}}\ls h^{s'-s}\| u_h^\pm\|_{X_{h,\pm}^{s',b}}\|v_h^\mp\|_{X_{h,\mp}^{s',b}}
\end{equation}
\end{lemma}

\begin{proof}
Again, we consider the case of $\|\nabla_h (u_h^- e^{-\frac{2t}{h^2}\partial_h}v_h^+)\|_{X_{h,-}^{s,-(1-b-\delta)}}$ only, and we write 
$$\begin{aligned}
&\int_{-\infty}^\infty\sum_{x\in h\mathbb{Z}}\nabla_h \big(u_h^-\cdot e^{-\frac{2t}{h}\partial_h}v_h^+\big)(t,x)\overline{w_h(t,x)}\\
&\sim\iiiint \frac{\frac{2i}{h}\sin(\frac{h\xi}{2})\la \xi \ra^s F(\tau_1,\xi_1)G(\tau-\tau_1,\xi-\xi_1) \overline{W(\tau,\xi)}d\xi_1 d \xi d\tau_1 d\tau}{\la \xi_1 \ra^{s}\la \tau_1+s_h(\xi_1) \ra^{b} \la \tau-\tau_1+\tfrac{2(\xi-\xi_1)}{h^2}-s_h(\xi-\xi_1) \ra^{b} \la \tau + s_h(\xi) \ra^{1-b - \delta}},
\end{aligned}$$
where 
$$\left\{\begin{aligned}
&F(\tau,\xi) =\la \xi \ra^s\la \tau+s_h(\xi) \ra^{b} \tilde{u}_h^+(\tau,\xi),\\
&G(\tau,\xi) =\la \tau+\tfrac{2\xi}{h^2}-s_h(\xi) \ra^{b} \tilde{v}_h^+(\tau+\tfrac{2\xi}{h^2},\xi),\\
&W(\tau,\xi) =\la \xi \ra^{-s}\la \tau+s_h(\xi) \ra^{1-b-\delta} \tilde{w}_h(\tau,\xi).
\end{aligned}\right.$$
Thus, similarly to the proof of the previous two lemmas, one can reduce to the bound
\begin{equation}\label{eq:bi3_2}
\frac{\frac{4}{h^2}\sin^2(\frac{h\xi}{2})}{\langle\xi\rangle^{s'-s}\la \tau + s_h(\xi) \ra^{2(1-b-\delta)}}\int_{-\frac{\pi}{h}}^{\frac{\pi}{h}} \frac{d\xi_1}{\la \tau + \frac{\xi}{h^2} -\frac{2}{h^3}(\sin(\frac{h\xi_1}{2})-\sin(\frac{h(\xi-\xi_1)}{2})) \ra^{2b}}.
\end{equation}
We may assume that $\xi>0$. Let $I_{\tau,\xi}$ denote the integral in \eqref{eq:bi3_2}. Changing the variable $\xi-2\xi_1\mapsto \xi_1$ and using the sum-to-product formula yields
$$\begin{aligned}
I_{\tau,\xi}&=\int_{-\frac{\pi}{h}}^{\frac{\pi}{h}} \frac{d\xi_1}{\la \tau + \frac{\xi}{h^2} +\frac{4}{h^3}\cos(\frac{h\xi}{4})\sin(\frac{h(\xi-2\xi_1)}{4}) \ra^{2b}}\\
&=\frac{2}{h}\int_{-\frac{\pi}{2}+\frac{h\xi}{4}}^{\frac{\pi}{2}+\frac{h\xi}{4}} \frac{d\xi_1}{\la \tau + \frac{\xi}{h^2} +\frac{4}{h^3}\cos(\frac{h\xi}{4})\sin\xi_1\ra^{2b}}.
\end{aligned}$$
By the trivial identity $\sin(\theta)=\sin(\pi-\theta)$,
$$\int_{\frac{\pi}{2}}^{\frac{\pi}{2}+\frac{h\xi}{4}}\frac{d\xi_1}{\langle\tau+\frac{\xi}{h^2}+\frac{4}{h^3}\cos(\frac{h\xi}{4})\sin\xi_1\rangle^{2b}}=\int_{\frac{\pi}{2}-\frac{h\xi}{4}}^{\frac{\pi}{2}}\frac{d\xi_1}{\langle\tau+\frac{\xi}{h^2}+\frac{4}{h^3}\cos(\frac{h\xi}{4})\sin\xi_1\rangle^{2b}}.$$
Thus, 
$$I_{\tau,\xi}\leq \frac{4}{h}\int_{-\frac{\pi}{2} + \frac{h\xi}{4}}^{\frac{\pi}{2}}\frac{d\xi_1}{\langle\tau+\frac{\xi}{h^2}+\frac{4}{h^3}\cos(\frac{h\xi}{4})\sin\xi_1\rangle^{2b}}.$$
By performing change of variables $\mu =  \frac{4}{h^3} \cos(\frac{h\xi}{4})\sin\xi_1$, we write
\[\begin{aligned}
I_{\tau,\xi} &\leq\frac{4}{h}\int_{-\frac{4}{h^3}\cos^2(\frac{h\xi}{4})}^{\frac{4}{h^3}\cos(\frac{h\xi}{4})} \frac{d \mu}{\langle\tau+\frac{\xi}{h^2}+\mu\rangle^{2b}\sqrt{\frac{4}{h^3}\cos(\frac{h\xi}{4}) + \mu}\sqrt{\frac{4}{h^3}\cos(\frac{h\xi}{4}) - \mu}}\\
&\leq\frac{1}{h\sqrt{\frac{4}{h^3}\cos(\frac{h\xi}{4})(1-\cos(\frac{h\xi}{4}))}}\int_{-\infty}^{\frac{4}{h^3}\cos(\frac{h\xi}{4})}  \frac{d \mu}{\langle\tau+\frac{\xi}{h^2}+\mu\rangle^{2b}\sqrt{|\frac{4}{h^3}\cos(\frac{h\xi}{4}) - \mu|}}\\
&\lesssim\frac{1}{\sqrt{\frac{4}{h}\cos(\frac{h\xi}{4})(1-\cos(\frac{h\xi}{4}))}}\int_{-\infty}^\infty \frac{d \mu}{\langle\mu\rangle^{2b}\sqrt{|\mu-(\tau+\frac{\xi}{h^2}+\frac{4}{h^3}\cos(\frac{h\xi}{4}))|}},
\end{aligned}\]
which, in addition to the half-angle formula and \eqref{eq:integral2}, implies that
\[I_{\tau,\xi} \lesssim \frac{\sqrt{h}}{\sin(\frac{h\xi}{8})\langle \tau + \frac{\xi}{h^2} + \frac{4}{h^3}\cos(\frac{h\xi}{4})\rangle^{\frac12}}.\]
Finally, by applying \eqref{eq:abz} and the half-angle formula, we prove that 
$$\begin{aligned}
&\frac{\frac{4}{h^2}\sin^2(\frac{h\xi}{2})}{\langle\xi\rangle^{s'-s}\la \tau + s_h(\xi) \ra^{2(1-b-\delta)}}I_{\tau,\xi}\\
&\lesssim\frac{\sqrt{h}\cdot\frac{4}{h^2}\sin^2(\frac{h\xi}{2})}{\langle\xi\rangle^{s'-s}\sin(\frac{h\xi}{8})\la \tau + \frac{\xi}{h^2}-\frac{2}{h^3}\sin(\frac{h\xi}{2}) \ra^{2(1-b-\delta)}\langle \tau + \frac{\xi}{h^2} + \frac{4}{h^3}\cos(\frac{h\xi}{4})\rangle^{\frac12}}\\
&\lesssim \frac{\sqrt{h}\cdot\frac{4}{h^2}\sin^2(\frac{h\xi}{2})}{\langle\xi\rangle^{s'-s}\sin(\frac{h\xi}{8})\sqrt{\frac{4}{h^3}\cos(\frac{h\xi}{4})(1+\sin(\frac{h\xi_1}{4}))}}\sim\frac{\sin^2(\frac{h\xi}{2})}{\langle\xi\rangle^{s'-s}\sin(\frac{h\xi}{8})}\\
&=\frac{16\sin^2(\frac{h\xi}{8})\cos^2(\frac{h\xi}{8})\cos^2(\frac{h\xi}{4})}{\langle\xi\rangle^{s'-s}\sin(\frac{h\xi}{8})}\sim \frac{\sin(\frac{h\xi}{8})}{\langle\xi\rangle^{s'-s}}\lesssim h^{s'-s}.
\end{aligned}$$
Therefore, we complete the proof.
\end{proof}

Finally, we show the regularity gain for the bilinear estimates in  higher regularity norms.

\begin{lemma}\label{lem:prod}
Let $0 < T \le1$ and $s > \frac34$. Then, we have
\begin{equation}\label{productrule}
\|u_h v_h\|_{L_t^2([-T,T];H_x^s(h\Z))} \ls \| \nabla_h^{-1} u_h \|_{X_{h,\pm}^{s,b}}\| v_h\|_{X_{h,\pm}^{s,b}},
\end{equation}
for $\nabla_h^{-1} u_h$, $v_h \in X_{h,\pm}^{s,b}$.
\end{lemma}
\begin{proof}
The Littlewood-Paley theory yields
\begin{equation}\label{prod1}
\|u_h v_h\|_{L_t^2H_x^s} \ls \| P_{\le N_0}(u_hv_h) \|_{L_t^2H_x^s} +\Bigg(\sum_{N_0\le N\le 1} \left(\frac{N}{h}\right)^{2s} \| P_N(u_hv_h) \|_{L_t^2L_x^2}^2\Bigg)^\frac12,
\end{equation}
where $N_0 < 1$ is the maximum dyadic number satisfying $N \le h$. Here, the time interval $[-T,T]$ is omitted in the norms. 

The first term on the right-hand side of \eqref{prod1} is easily treated compared to the second one. Indeed, from the fact that $\la \xi \ra \sim 1$ on the support of $P_{N_0}$ and Corollary \ref{cor:Strichartz},
we show that
\begin{equation}\label{low}
\| P_{\le N_0}(u_hv_h) \|_{L_t^2H_x^s}\ls \|u_hv_h\|_{L_t^2L_x^2} \ls \| u_h \|_{L_x^\infty L_t^2} \|v_h\|_{L_x^2 L_t^\infty} \ls \| \nabla_h^{-1} u_h \|_{X_{h,\pm}^{0,b}}\|v_h\|_{X_{h,\pm}^{s,b}}.
\end{equation}
For the second term, we further decompose
\[\begin{aligned}
\| P_N(u_hv_h) \|_{L_t^2L_x^2} &\le \| P_N((P_{\le \frac{N}{4}}u_h)v_h) \|_{L_t^2L_x^2} + \| P_N((P_{>\frac{N}{4}}u_h)v_h) \|_{L_t^2L_x^2}\\
&=: I+II. 
\end{aligned}\]
For $I$, we observe that $P_N((P_{\le \frac{N}{4}}u_h)v_h)=P_N((P_{\le \frac{N}{4}}u_h)(P_{\frac{N}{16} < \cdot <16N} v_h))$ owing to the support property\footnote{Roughly speaking, in a ($k$-)multilinear form, one has a frequency relation $\xi_1 + \cdots + \xi_k = \xi$; thus, the multilinear form vanishes unless the maximum two frequencies are comparable.}. Thus, by the H\"older and Bernstein inequalities \eqref{ineq:Bernstein} and Corollary \ref{cor:Strichartz}, we obtain
\begin{align*}
\| P_N((P_{\le \frac{N}{4}}u_h)v_h) \|_{L_t^2L_x^2} &\ls \| P_N((P_{\le \frac{N}{4}}u_h)(P_{\frac{N}{16} < \cdot <16N} v_h)) \|_{L_t^2 L_x^2}  \\
&\ls \| P_{\le \frac{N}{4}}u_h \|_{L_t^2L_x^\infty}  \|P_{\frac{N}{16} < \cdot <16N} v_h\|_{L_t^\infty L_x^2}  \\
&\ls T^{\frac12-\frac1q} \|u_h \|_{L_t^qL_x^\infty}\|P_{\frac{N}{16} < \cdot <16N} v_h\|_{L_t^\infty L_x^2}\\
&\lesssim T^{\frac{1}{2}-\frac1q}\| |\pa_h|^{-\frac1q} u_h \|_{X_{h,\pm}^{0,b}} \|P_{\frac{N}{16} < \cdot <16N} v_h\|_{X_{h,\pm}^{0,b}}.
\end{align*}
Thus, it follows that 
\begin{equation}\label{high1}
\begin{aligned}
&\sum_{N_0\le N\le 1} \left(\frac{N}{h}\right)^{2s} \| P_N((P_{\le \frac{N}{4}}u_h)v_h) \|_{L_t^2L_x^2}^2 \\
&\lesssim T^{1-\frac2q}\| |\pa_h|^{-\frac1q} u_h \|_{X_{h,\pm}^{0,b}}^2 \sum_{N_0\le N\le 1} \left(\frac{N}{h}\right)^{2s}\|P_{\frac{N}{16} < \cdot <16N} v_h\|_{X_{h,\pm}^{0,b}}^2\\
&\lesssim T^{1-\frac2q}\| \nabla_h^{-1} u_h \|_{X_{h,\pm}^{1-\frac1q,b}}^2
\| v_h \|_{X_{h,\pm}^{s,b}}^2.
\end{aligned}
\end{equation}
For $II$, by repeating the estimates in \eqref{low}, we obtain
$$\| P_N((P_{\ge \frac{N}{4}}u_h)v_h) \|_{L_t^2L_x^2}\leq\|P_{\ge \frac{N}{4}}u_h\|_{L_x^\infty L_t^2}\|v_h \|_{L_x^2 L_t^\infty}\lesssim \|P_{\ge \frac{N}{4}}\nabla_h^{-1}u_h\|_{_{X_{h,\pm}^{0,b}}} \|v_h \|_{X_{h,\pm}^{s,b}}.$$
Inserting this result and by Fubini's theorem for the sum, we obtain
\begin{equation}\label{high3}
\begin{aligned}
&\sum_{N_0\le N\le 1} \left(\frac{N}{h}\right)^{2s} \| P_N((P_{\ge \frac{N}{4}}u_h)v_h) \|_{L_t^2L_x^2}^2 \\
&\lesssim\|v_h \|_{X_{h,\pm}^{s,b}}^2\sum_{N_0\le N\le 1} \left(\frac{N}{h}\right)^{2s} \sum_{M\geq \frac{N}{4}}\|P_{M}\nabla_h^{-1}u_h\|_{_{X_{h,\pm}^{0,b}}}^2 \\
&\sim\|v_h \|_{X_{h,\pm}^{s,b}}^2\sum_{\frac{N_0}{4}\le M\le 1}\left(\frac{M}{h}\right)^{2s} \|P_{M}\nabla_h^{-1}u_h\|_{_{X_{h,\pm}^{0,b}}}^2\\
&\ls \|\nabla_h^{-1} u_h\|_{X_{h,\pm}^{s,b}} \| v_h \|_{X_{h,\pm}^{s,b}}^2.
\end{aligned}
\end{equation}
By combining \eqref{low}, \eqref{high1} and \eqref{high3} for the right-hand side of \eqref{prod1}, we complete the proof.
\end{proof}

\section{Uniform bounds for nonlinear solutions}\label{proof of uniform bound}

This section is devoted to bounds for solutions to the three equations in consideration. In Section \ref{sec: bound for KdV}, we briefly review the well-posedness of the KdV equation and state several mixed norm bounds for nonlinear solutions. In Section \ref{sec: uniform bound for coupled FPU}, we obtain analogous uniform bounds for the coupled and decoupled FPUs (Proposition \ref{Prop:LWP} and \ref{Prop:LWP}). 
The main results in this section play a crucial role in our analysis.

\subsection{Bounds for solutions to KdVs}\label{sec: bound for KdV}
We consider the KdV equations
\begin{align}\label{KdV2}
\pa_t w_\pm \pm \frac{1}{24} \pa_x^3 w_\pm \mp \frac14\pa_x(w_\pm^2)=0,
\end{align}
where $w_\pm=w_\pm(t,x):\mathbb{R}\times\mathbb{R}\to\mathbb{R}$, i.e., the differential form of \eqref{KdV}. This equation is nothing but the standard formulation of the KdV equation $\pa_t w+ \pa_x^3 w- \pa_x(w^2)=0$, because by simple changes of variables, $w_+(t,x)=w(\frac{\sqrt{6}}{4}t,\sqrt{6}x)$ and $w_-(t,x)=w_+(t,-x)$.

KdV has been a central research topic in various fileds of mathematics, especially becaues of its complete integrability. From an analysis perspective, its well-posedness has been investigated by many authors. We refer to, for instance, \cite{BS-1975, ST-1976, KPV-1989, KPV-1991, KPV-1993CPAM, KPV-1993DUKE, KPV-1996, CKSTT-2001, CKSTT-2003, G-2009, K-2009, KV-2019}, and the references therein. It should be noted that among the various important results, in \cite{KPV-1996}, Kenig, Ponce, and Vega established local well-posedness in the negative Sobolev space $H^{s}$ for $s>-\frac{3}{4}$; later, in \cite{CKSTT-2003}, Colliander, Keel, Staffilani, Takaoka, and Tao extended the previous result to global well-posedness at the same low regularity level.  It was further improved by Guo \cite{G-2009} and Kishimoto \cite{K-2009} independently at the end-point $s=-\frac34$. These low regularity well-posedness results are known to be the best possible ones via the contraction mapping argument, because uniform continuity of the data-to-solution map fails when $s<-\frac{3}{4}$ (see \cite{KPV-2001, CCT-2003}). Remarkably, very recently in \cite{KV-2019}, Killip and Visan established global well-posedness in $H^{-1}$ by exploiting integrability of the equation.

Coming back to our discussion, we restrict ourselves to non-negative Sobolev spaces, and state the following well-posedness theorem of Bourgain \cite[Appendix 2]{B-1993KdV}.

\begin{theorem}[Well-posedness of KdVs]\label{WP for KdVs} Let $s\ge0$.\\
$(i)$ (Local well-posedness) There exist $b\in(\frac{1}{2},1)$ such that for the given initial data $w_0^\pm\in H^s(\R)$, there exist $T>0$, depending on $\|w_{0}^\pm\|_{H^s}$, and a unique solution $w^{\pm}(t)$ to KdV \eqref{KdV2} in the interval $[-T,T]$ satisfying $w^\pm\in C_t([-T,T]; H^s(\R))$ and $w^\pm \in X_\pm^{s,b}$, 
where $X_{\pm}^{s,b}$ is the Bourgain space equipped with the norm as in \eqref{eq:XsbKdV}.
Moreover, the solution $w^\pm(t)$ conserves the momentum
$$P_\pm(t)=\int_{-\infty}^\infty w_\pm(t,x)^2 dx = P_\pm(0).$$
$(ii)$ (Global well-posedness) The solution $w^\pm(t)$ exists globally in time. 
\end{theorem}

Combining the previous theorem and Proposition \ref{Thm:linear estimates for KdV flows} with the transference principle, we deduce the following mixed norm bounds.

\begin{corollary}[Mixed norm bounds for the KdVs]\label{Cor:boundforKdV}
For $s>\frac{3}{4}$ and $b\in(\frac{1}{2},1)$, let $w^{\pm}(t)\in C_t([-T,T]; H_x^s(\R))$ be the solution to KdV \eqref{KdV2} with initial data $w_0^\pm\in H^s(\mathbb{R})$, constructed in Theorem \ref{WP for KdVs}. Then, we have
\begin{align*}
\big\||\pa_x|^{\frac{1}{4}}\langle\pa_x\rangle^s w^\pm\big\|_{L_t^{4}([-T,T]; L_x^{\infty}(\mathbb{R}))} &\ls \|w_{0}^\pm\|_{H^s(\R)}, \\
\| \pa_x\langle\pa_x\rangle^s w^\pm\|_{L_x^\infty(\R;L_t^2([-T,T]))} &\ls \|w_{0}^\pm\|_{H^s(\R)},\\
\| w^\pm\|_{L_x^2(\mathbb{R};L_{t}^\infty([-T,T])) }&\ls \|w_{0}^\pm\|_{H^s(\R)}.
\end{align*}
\end{corollary}

\subsection{Uniform bounds for coupled and  decoupled FPUs}\label{sec: uniform bound for coupled FPU}

Next, we state the main results of this section. They assert uniform bounds for the coupled and  decoupled FPUs, analogous to those for the KdV equations.

\begin{proposition}[Uniform bound for coupled and  decoupled FPU]\label{Prop:LWP}
Let $s \ge 0$. Suppose that
$$\sup_{h\in(0,1]}\sum_\pm\|u_{h,0}^\pm\|_{H^s(h\mathbb{Z})}\leq R.$$
Then, there exists $T>0$, depending on $R>0$ but not on $h\in(0,1]$, such that the solution $(u_h^+(t),u_h^-(t))$ (resp., $(v_h^+(t),v_h^-(t))$) to the coupled FPU \eqref{coupled FPU'} (resp., decoupled FPU \eqref{decoupled FPU}) with initial data $(u_{h,0}^+,u_{h,0}^-)$ such that
\begin{equation}\label{uniform}
\left\|\theta(\tfrac{t}{T})u_h^\pm(t)\right\|_{X_{h,\pm}^{s,b}}\lesssim \|u_{h,0}^\pm\|_{H^s(h\mathbb{Z})}\quad\Big(\textup{resp., }\left\|\theta(\tfrac{t}{T})v_h^\pm(t)\right\|_{X_{h,\pm}^{s,b}}\lesssim \|u_{h,0}^\pm\|_{H^s(h\mathbb{Z})}\Big),
\end{equation}
where $\theta\in C_c^\infty$ is a non-negative cut-off, and the Bourgain space $X_{h,\pm}^{s,b}$ is given in Definition \ref{def Xsb}.
\end{proposition}

\begin{remark}
For the same reason mentioned in Remark \ref{rem:simple}, the proof below does not include the estimate of the higher-order remainder term in \eqref{coupled FPU'}. See Lemma \ref{lem:AppB} for the proof of the estimate of the higher-order remainder term.
\end{remark}

\begin{remark}
As a consequence of the bound \eqref{uniform}), in addition to \eqref{eq:embedding}, we have (also for $v_h^{\pm}$)
\[\|(u_h^+(t), u_h^-(t))\|_{C_t([-T,T];H_x^s(h\Z))}\lesssim \|(u_{h,0}^+,u_{h,0}^-)\|_{H^s(h\Z)}.\]
\end{remark}

Analogous to Corollary \ref{Cor:boundforKdV}, we have the following result from Propositions \ref{prop:Strichartz} and \ref{Prop:LWP}.

\begin{corollary}[Mixed norm bounds for the decoupled FPU]\label{Cor:estimates}
Let $s\ge 0$,$\frac{3}{4} < s' < \frac{3}{2}$ and $b> \frac12$. 
Let $v_h^\pm(t)\in C_t([-T,T]; L_x^2(h\mathbb{Z}))$ be the solution to  the decoupled FPU \eqref{decoupled FPU} with initial data $u_{h,0}^{\pm}$ given in Proposition \ref{Prop:LWP}. 
Let $m(D_h)$ be the Fourier multiplier, i.e., $\mathcal F_h[m(D_h)f_h](\xi)=m(\xi)\widehat{f_h}(\xi)$, with a uniform bound such that
\begin{align*}
|m(\xi)|\le C \la \xi\ra^s , \text{ for all } \xi\in \T_h,
\end{align*}
where $C$ is independent of $h$.
Then, we have for $q$ as in Corollary~\ref{Corendpoint},
\begin{align*}
\big\|  m(D_h) \pa_h v_h^\pm\big\|_{L_t^{q}([-T,T];L_x^\infty(h\Z))} &\ls \|u_{h,0}^\pm\|_{H{s+s'}(h\Z)}, \\
\| m(D_h) \pa_h v_h^\pm\|_{L_x^\infty (h\Z;L_t^2([-T,T])} &\ls \|u_{h,0}^\pm\|_{H^s(h\Z)},\\
\|  m(D_h) v_h^\pm\|_{L_x^2 (h\Z;L_t^\infty([-T,T]))}&\ls \|u_{h,0}^\pm\|_{H^{s+s'}(h\Z)}.
\end{align*}
The discrete differential operator $\pa_h$ can be replaced by $\nabla_h$ or $\pa_h^+$.
\end{corollary}

Our proof of Proposition \ref{Prop:LWP} uses the standard iteration scheme (also known as the ``Picard iteration method") via the Fourier restriction norm method, and the main ingredients are the bilinear estimates established in Section \ref{sec: bilinear estimates} (in particular, Lemma \ref{lem:bi1}, \ref{lem:bi2} and \ref{lem:bi3}). Here, we present the  proof of only the coupled FPU, because the proof of the decoupled one closely follows. Indeed, the latter is simpler owing to the absence of the coupled terms.

\begin{proof}[Proof of Proposition \ref{Prop:LWP}]
Let $\theta \in C_c^\infty(\R)$ be a time cut-off function satisfying $\theta(t) \equiv 1$ on $[-1,1]$ and $\theta(t) = 0$ for $|t| > 2$. For sufficiently small $T\in(0,1]$ to be chosen later, we define
$$\Phi(u_h^+, u_h^-)=\big(\Phi^+(u_h^+), \Phi^-(u_h^-)\big)$$
by
$$\Phi^\pm(u_h^\pm):= \theta(\tfrac{t}{T})S_h^\pm(t)u_{h,0}^\pm\mp\frac{\theta(\tfrac{t}{T})}{4}  \int_0^t S_h^\pm(t-t_1) \theta(\tfrac{t_1}{2T})\nabla_h \Big\{u_h^\pm(t_1)+e^{\pm\frac{2t_1}{h^2}\partial_h}u_h^\mp(t_1)\Big\}^2dt_1.$$
Let $X_h^{s,b}$ denote the solution space for $(u_h^+,u_h^-)$ equpped with the norm
$$\big\| (u_h^+,u_h^-)\big\|_{X_h^{s,b}} := \sum_\pm\|u_h^\pm\|_{X_{h,\pm}^{s,b}},$$
where the exponent $b$ will be chosen later.

Lemma \ref{lem:properties} (4) and (5) yield
\begin{equation}\label{eq:contraction1}
\|\Phi^\pm(u_h^\pm)\|_{X_{h,\pm}^{s,b}} \leq CT^{\frac12-b}R+ CT^{\delta}\Big\|\nabla_h \Big\{u_h^\pm+e^{\pm\frac{2t}{h^2}\partial_h}u_h^\mp\Big\}^2\Big\|_{X_{h,\pm}^{s,b-1+\delta}}
\end{equation}
for some $C\geq1$ and $\delta >0$ (to be chosen later). Then, by applying Lemma \ref{lem:bi1}, \ref{lem:bi2} and \ref{lem:bi3} (with $s'=s$) to the second term on the right-hand side of \eqref{eq:contraction1}, we obtain
$$\begin{aligned}
\big\|\Phi^\pm(u_h^\pm)\big\|_{X_{h,\pm}^{s,b}} &\le CT^{\frac12-b} R+ \tilde{C}T^{\delta}\big\|\big(u_h^+,u_h^-\big)\big\|_{X_h^{s,b}}^2.
\end{aligned}$$
Thus, by summing in $\pm$, we have
\begin{equation}\label{eq:contraction2}
\big\|\Phi(u_h^+, u_h^-)\big\|_{X_h^{s,b}}\leq 2CT^{\frac12-b} R + 2\tilde{C}T^{\delta}\big\|\big(u_h^+,u_h^-\big)\big\|_{X_h^{s,b}}^2.
\end{equation}
Now, we choose $b$ and $\delta$ satisfying
\[\frac12 < b < \frac34 - \delta \quad \mbox{and} \quad b-\frac12 < \delta < \frac14,\]
and we take $T>0$ such that
\begin{equation}\label{eq:T choice}
32C\tilde{C}T^{\frac12 + \delta -b}R \leq 1.
\end{equation}
Then, $\Phi$ maps from the set
\[\mathfrak{X}:=\Big\{ (u_h^+, u_h^-) \in X_h^{s,b} : \big\|\big(u_h^+, u_h^-\big) \big\|_{X_h^{s,b}} \le 4CT^{\frac12-b}R \Big\}\]
to itself. Indeed, it follows from \eqref{eq:contraction2} that 
$$\begin{aligned}
\left\|\Phi(u_h^+, u_h^-)\right\|_{X_h^{s,b}}&\leq 2CT^{\frac12-b}R + 2\tilde{C}T^{\delta}\cdot(4CT^{\frac12-b}R)^2\\
&= 2CT^{\frac12-b}R + 32C\tilde{C}T^{\frac{1}{2}+\delta-b}R\cdot CT^{\frac12-b}R\\
&\leq 4 CT^{\frac12-b}R.
\end{aligned}$$

We repeat the procedure for the difference. By Lemma \ref{lem:properties} (4) and (5), it follows that
$$\begin{aligned}
&\big\|\Phi^\pm(u_h^\pm)-\Phi^\pm(\tilde{u}_h^\pm)\big\|_{X_{h,\pm}^{s,b}} \\
&\leq CT^{\delta}\Big\|\nabla_h \Big(u_h^\pm+e^{\pm\frac{2t}{h^2}\partial_h}u_h^\mp\Big)^2-\nabla_h \Big(\tilde{u}_h^\pm+e^{\pm\frac{2t}{h^2}\partial_h}\tilde{u}_h^\mp\Big)^2\Big\|_{X_{h,\pm}^{s,b-1+\delta}}\\
&= CT^{\delta}\Big\|\nabla_h\Big((u_h^\pm+\tilde{u}_h^\pm)+e^{\pm\frac{2t}{h^2}\partial_h}(u_h^\mp+\tilde{u}_h^\mp)\Big) \Big((u_h^\pm-\tilde{u}_h^\pm)+e^{\pm\frac{2t}{h^2}\partial_h}(u_h^\mp-\tilde{u}_h^\mp)\Big)\Big\|_{X_{h,\pm}^{s,b-1+\delta}}.
\end{aligned}$$
Next, by applying Lemmas \ref{lem:bi1}, \ref{lem:bi2}, and \ref{lem:bi3} (with $s'=s$), we obtain
$$\big\|\Phi^\pm(u_h^\pm)-\Phi^\pm(\tilde{u}_h^\pm)\big\|_{X_{h,\pm}^{s,b}}\leq \tilde{C}T^{\delta}\left(\sum_\pm\|u_h^\pm\|_{X_{h,\pm}^{s,b}}+\|\tilde{u}_h^\pm\|_{X_{h,\pm}^{s,b}}\right) \sum_\pm\|u_h^\pm-\tilde{u}_h^\pm\|_{X_{h,\pm}^{s,b}}.$$
Hence, if $(u_h^+, u_h^-), (\tilde{u}_h^+, \tilde{u}_h^-)\in \mathfrak{X}$, then
\begin{equation}\label{eq:diff}
\begin{aligned}
\big\|\Phi\big(u_h^+,u_h^-\big)-\Phi\big(\tilde{u}_h^+, \tilde{u}_h^-\big)\big\|_{X_h^{s,b}}&=\sum_\pm\big\|\Phi^\pm(u_h^\pm)-\Phi^\pm(\tilde{u}_h^\pm)\big\|_{X_{h,\pm}^{s,b}} \\
&\leq 2\cdot\tilde{C}T^{\delta}\cdot 8CT^{\frac12-b}R\cdot\left\|\left(u_h^+-\tilde{u}_h^+,u_h^--\tilde{u}_h^-\right)\right\|_{X_h^{s,b}}\\
&\leq \frac{1}{2}\left\|\left(u_h^+-\tilde{u}_h^+,u_h^--\tilde{u}_h^-\right)\right\|_{X_h^{s,b}}.
\end{aligned}
\end{equation}
Therefore, we conclude that $\Phi$ is contractive on $\mathfrak{X}$. Consequently, $(u_h^+, u_h^-)$ is a solution to the coupled FPU \eqref{coupled FPU'}, which by uniqueness coincides with the solution in $C_t([-T,T]; L_x^2(h\mathbb{Z}))$, constructed in Proposition \ref{LWP of CS}. Moreover, it satisfies the desired bound $\|u_h^\pm\|_{X_{h,\pm}^{s,b}}\leq 4CT^{\frac12-b}R$.
\end{proof}

\section{Convergence from FPU to counter-propagating KdVs}\label{sec:main proof}

\subsection{From coupled to decoupled FPUs: Proof of Proposition \ref{convergence to decoupled FPU}}\label{sec: coupled to decoupled}
This section is devoted to showing that solutions to the coupled system \eqref{coupled FPU'} approximate to those to the decoupled system \eqref{decoupled FPU} in $L^2(h\Z)$ as $h \to \infty$. 

The remainder of this section will be (roughly) presented as follows (see also Fig. \ref{Fig:1}). For given (slightly) regular initial data $u_{h,0}^{\pm}$ (in $H^s(h\Z)$, $0 < s \le 1$), we have $H^s(h\Z)$ local solutions $u_h^{\pm}$ and $v_h^{\pm}$ (see Proposition \ref{LWP of CS}). Then, a suitable choice of $0 < T \ll 1$ and uniform bounds of $u_h^\pm$ and $v_h^\pm$ (see Propositions \ref{Prop:LWP} and \ref{Prop:LWP})) ensure that
\[\norm{u_h^{\pm}(t) - v_h^{\pm}(t)}_{L_x^2(h\Z)} \lesssim O(h^s), \quad \mbox{for all} \quad |t| \le T.\]
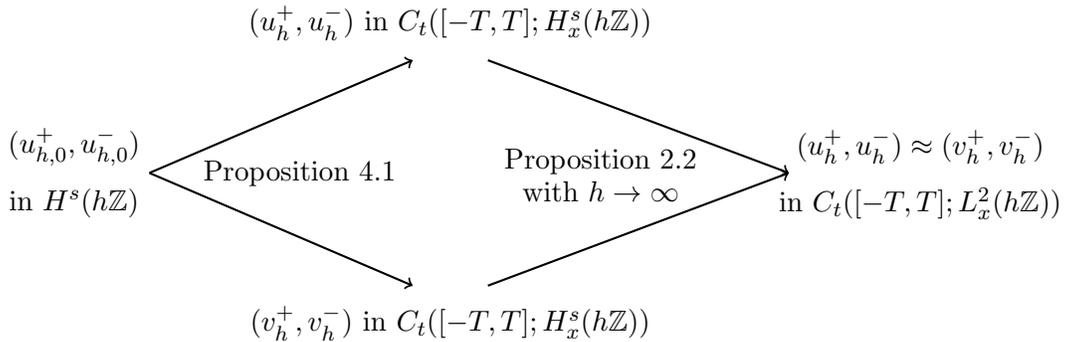
\begin{figure}[h!]
\begin{center}
\begin{tikzpicture}[scale=0.5]
\node at (-10,0.7){$(u_{h,0}^+,u_{h,0}^-)$};
\node at (-10,-0.8){in $H^{s}(h\Z)$};
\draw[thick,->] (-8,0) -- (-1,3);
\node at (-4,0){Proposition \ref{LWP of CS}};
\node at (0,4){$(u_h^+,u_h^-)$ in $C_t([-T,T];H_x^{s}(h\Z))$};
\draw[thick,->] (-8,0) -- (-1,-3);
\node at (0,-4){$(v_h^+,v_h^-)$ in $C_t([-T,T];H_x^{s}(h\Z))$};
\node at (4,0.3){Proposition \ref{convergence to decoupled FPU}};
\draw[thick,->] (1,3) -- (9,0);
\node at (4,-0.5){with $h \to \infty$};
\draw[thick,->] (1,-3) -- (9,0);
\node at (12.5,0.7){$(u_h^+,u_h^-) \approx (v_h^+,v_h^-)$};
\node at (12.5,-0.8){in $C_t([-T,T];L_x^2(h\Z))$};
\end{tikzpicture}
\end{center}
\caption{Schematic representation for decoupling \eqref{decoupled FPU} from \eqref{coupled FPU'}.}\label{Fig:1}
\end{figure}

\begin{remark}
As $h^2$ is involved in the higher-order remainder term in \eqref{coupled FPU'} (see Lemma \ref{lem:AppB}), the estimate of the higher-order term is not essential in the proof below; thus, we omit it.
\end{remark}

\begin{proof}[Proof of Proposition \ref{convergence to decoupled FPU}]
For given initial data $u_{h,0}^{\pm}$, let $u_h^\pm(t)$ (resp., $v_h^\pm(t)$) be the solution to the coupled FPU (resp., the decoupled FPU) constructed in Proposition \ref{LWP of CS}, and let $T=T(R)>0$ be the minimum of the existence times for two solutions. Moreover, Proposition \ref{Prop:LWP} implies that the $X_{h, \pm}^{s,b}$ norms of the solutions are uniformly bounded by the size of the initial data, i.e.,
\begin{equation}\label{eq:decoupling1}
\big\|u_h^\pm\big\|_{X_{h,\pm}^{s,b}}, \big\|v_h^\pm\big\|_{X_{h,\pm}^{s,b}} \le M=M(R).
\end{equation}
Note that $M(R) = 4CT^{\frac12 -b}R$.

First, by subtracting \eqref{decoupled FPU} from \eqref{coupled FPU'}, we write
$$\begin{aligned}
&u_h^\pm(t) - v_h^{\pm}(t) \\
=&~{}\mp\frac{\theta(\frac{t}{T})}4 \int_0^t S_h^\pm(t-t_1) \theta(\tfrac{t_1}{2T})\nabla_h \Big\{\big(u_h^\pm(t_1) + v_h^\pm(t_1)\big)\big(u_h^\pm(t_1) - v_h^\pm(t_1)\big)\Big\}dt_1 \\
&~{}\mp\frac{\theta(\frac{t}{T})}{4} \int_0^t S_h^\pm(t-t_1) \theta(\tfrac{t_1}{2T})\nabla_h \Big\{2u_h^\pm(t_1)\cdot e^{\mp\frac{2t_1}{h^2}\partial_h}u_h^\mp(t_1) + \big(e^{\mp\frac{2t_1}{h^2}\partial_h}u_h^\mp(t_1)\big)^2\Big\}dt_1.
\end{aligned}$$ 
%
%
We take the $X_{h,\pm}^{0,b}$ norm\footnote{One can fix $b > \frac12$ here,and $\delta > 0$ below such that the argument of the local well-posedness is valid.} on both sides. Similarly to the proof of Proposition \ref{Prop:LWP}, we apply Lemma \ref{lem:bi1} to the first integral and Lemma \ref{lem:bi2} and \ref{lem:bi3} to the second integral to obtain
$$\begin{aligned}
\big\|u_h^{\pm}(t)- v_h^{\pm}(t)\big\|_{X_{h,\pm}^{0,b}}\leq&~{} \tilde CT^{\delta}\Big(\|u_h^{\pm}\|_{X_{h,\pm}^{0,b}}+\|v_h^{\pm}\|_{X_{h,\pm}^{0,b}}\Big)\big\|u_h^{\pm} - v_h^{\pm}\big\|_{X_{h,\pm}^{0,b}}\\
&+ \tilde Ch^s T^{\delta}\Big(\|u_h^{\pm}\|_{X_{h,\pm}^{s,b}}\|u_h^{\mp}\|_{X_{h,\pm}^{s,b}} + \|u_h^{\mp}\|_{X_{h,\pm}^{s,b}}^2\Big),
\end{aligned}$$
for some uniform constant $\tilde C>0$ as in \eqref{eq:diff}. 
With the bounds \eqref{eq:decoupling1}, we extend the time interval $[-T,T]$ to be as long as  $2CT^{\delta}M < \frac12$, 
The local existence time $T>0$ satisfying \eqref{eq:T choice} enables us to conclude that 
$$4CM^2h^s \geq\big\|u_h^{\pm}(t)-v_h^{\pm}(t)\big\|_{X_{h,\pm}^{0,b}}\gtrsim\big\|u_h^{\pm}(t)-v_h^{\pm}(t)\big\|_{C_tL_x^2},$$
where the embedding $X_{h,\pm}^{0,b} \hookrightarrow C_t([-T,T];L_x^2(h\Z))$ is used in the last step (see Lemma \ref{lem:properties} (3)).
\end{proof}

\subsection{From decoupled FPU to KdV: Proof of Proposition \ref{convergence to KdV}}\label{sec: decoupled to KdV}

We now prove that solutions to the decoupled FPU \eqref{decoupled FPU} can be approximated by those to KdV \eqref{KdV} as $h\to0$ (Proposition \ref{convergence to KdV}). To compare the two solutions, we employ the linear interpolation $l_h$ defined as in \eqref{def: linear interpolation} and the spacetime norm $S([-T,T])$ given by \eqref{mixed spacetime norm}.

\begin{proof}[Proof of Proposition \ref{convergence to KdV}]
Let $T>0$ be a sufficiently small number to be chosen later independently of $h\in(0,1]$. 
Using \eqref{decoupled FPU} and \eqref{KdV}, we write\footnote{In what follows, as mentioned in Remark \ref{rem:w_h}, we denote the solutions to KdVs \eqref{KdV} by $w_h^{\pm}$, even if they are posed on $\R$.}
\begin{align*}
l_h v_h^\pm(t)-w_h^\pm(t)
&=\Big\{ l_h S_h^\pm(t) u_{h,0}^\pm-S^\pm(t) l_h u_{h,0}^\pm \Big\} \\
&\quad\mp\frac14 \int_0^t \big( l_h S_h^\pm(t-t_1) - S^\pm(t-t_1) l_h \big)\nabla_h\big(v_h^\pm(t_1)^2\big)  dt_1 \\
&\quad\mp\frac14\int_0^t S^\pm(t-t_1)\Big\{ l_h \nabla_h  \big(v_h^\pm(t_1)^2\big) -  \partial_x\big((l_hv_h^\pm)(t_1)^2\big) \Big\} dt_1  \\
&\quad\mp\frac14 \int_0^t S^\pm(t-t_1)\partial_x\Big\{ \big( l_h v_h^\pm)(t_1)^2 -  w_h^\pm(t_1)^2  \Big\} dt_1.
\end{align*}
Then, Propositions \ref{linear approx} and \ref{Thm:linear estimates for KdV flows} enable us to obtain
\begin{equation}\label{eq:kdvlimit}
\begin{aligned}
\|l_h v_h^\pm-w_h^\pm\|_{S([-T,T])}
&\lesssim h^{\frac{2s}{5}} \|u_{h,0}\|_{H_h^s}+h^{\frac{2s}{5}} \big\|\nabla_h\big(v_h^\pm\big)^2\big\|_{L_t^1H_x^s}\\
&\quad+\Big\|l_h \nabla_h (v_h^\pm)^2-\partial_x(l_hv_h^\pm)^2 \Big\|_{L_t^1L_x^2}\\
&\quad+\Big\|\partial_x\Big\{ \big( l_h v_h^\pm)^2 -  (w_h^\pm)^2  \Big\}\Big\|_{L_t^1L_x^2}\\
&=:I_1+I_2+I_3+I_4.
\end{aligned}
\end{equation}

For $I_2$, we apply the Leibniz rule \eqref{Leibnitz rule2} and the H\"older inequality, 
\[\begin{aligned}
I_2 \le&~{} 2h^{\frac{2s}{5}} \big\| \nabla_h v_h^{\pm}(t) \cos\big(\tfrac{-i\pa_h}{2h}\big)v_h^{\pm}(t)\big\|_{L_t^1H_x^s} \\
\lesssim&~{} h^{\frac{2s}{5}}T^\frac12 \Big\| \la \pa_h\ra^s\Big\{\nabla_h v_h^{\pm}(t) \cos(\tfrac{-i\pa_h}{2h})v_h^{\pm}(t)\Big\} \Big\|_{L_t^2L_x^2}.
\end{aligned}\]
Since the operator $\cos(\frac{-i\pa_h}{2h})$ is bounded in $L_h^2$ (independent of $h$), from Lemma \ref{lem:prod} and Proposition~\ref{Prop:LWP}, we obtain
\begin{equation}\label{eq:I_2}
I_2\ls h^{\frac{2s}{5}}T^\frac12 \|v_h^{\pm}(t)\|_{X_{h,\pm}^{s,b}}\left\| v_h^{\pm}(t)\right\|_{X_{h,\pm}^{s,b}}
\ls T^\frac12 h^{\frac{2s}{5}} \|u_{h,0}^{\pm}\|_{H^s}^2.
\end{equation}
For $I_3$, Propositions \ref{Pro:reverse order} and \ref{Pro:almost distribution} immediately yield
\begin{align}\begin{aligned}\label{I3}
I_3&\leq \Big\|l_h \nabla_h (v_h^\pm)^2-\partial_xl_h(v_h^\pm)^2 \Big\|_{L_t^1L_x^2}+\Big\|\partial_x\Big\{l_h(v_h^\pm)^2-(l_hv_h^\pm)^2\Big\} \Big\|_{L_t^1L_x^2}\\
&\ls h\Big\{\big\|(v_h^\pm)^2\big\|_{L_t^1\dot H_x^2}+\big\| \big(\pa_h^+ v_h^\pm\big)^2 \big\|_{L_t^1L_x^2}\Big\}.
\end{aligned}\end{align}
For the first term in \eqref{I3}, we apply \eqref{Leibnitz rule2} twice to get
\begin{align*}
\nabla_h^2 \Big\{(v_h^\pm)^2\Big\}
&=2\nabla_h^2 v_h^\pm\cdot \cos^2(\tfrac{-ih\pa_h}{2})v_h^\pm
+2\Big\{\cos(\tfrac{-ih\pa_h}{2}) \nabla_h v_h^\pm\Big\}^2.
\end{align*}
Then, by \eqref{norm equivalence}, the H\"older inequality and Corollary~\ref{Cor:estimates}, we estimate 
\begin{align*}
\big\|  (v_h^\pm)^2\big\|_{L_t^1\dot H_x^2}&\sim \Big\| \nabla_h^2 \Big\{(v_h^\pm)^2\Big\}\Big\|_{L_t^1L_x^2}  \\
&\ls
T^{1-\frac{1}{q}}\| \nabla_h^2 v_h^\pm\|_{L_x^2L_t^\infty}  \big\| \cos^2(\tfrac{-ih\pa_h}{2})v_h^\pm  \big\|_{L_x^\infty L_t^2}\\
&\quad+T^{1- \frac1q}\big\| \nabla_h\cos(\tfrac{-ih\pa_h}{2})v_h^\pm  \big\|_{L_t^qL_x^\infty}
 \big\| \nabla_h\cos(\tfrac{-ih\pa_h}{2})v_h^\pm \big\|_{L_t^\infty L_x^2}\\
&\ls T^{1-\frac{1}{q}}\| u_{h,0}^\pm \|_{\dot{H}^1}  \| u_{h,0}^\pm \|_{H^{s}}.
\end{align*}
Similarly, we estimate the second term on the right-hand side of \eqref{I3} as 
\begin{align*}
\big\| \big(\pa_h^+ v_h^\pm\big)^2 \big\|_{L_t^1L_x^2}
&\ls T^{1-\frac1q} \| \pa_h^+  v_h^\pm \|_{L_t^qL_x^\infty}
\|  \pa_h^+  v_h^\pm \|_{L_t^\infty L_x^2} \\
&\ls  T^{1-\frac1q}  \|u_{h,0}\|_{H^s}\|u_{h,0}\|_{\dot{H}^1}.
\end{align*}
Thus, by Lemma~\ref{Lem:Sobolev bound depending on h}, we conclude that 
\begin{equation}\label{eq:I_3}
I_3 \ls h^{s} T^{\frac12}\|u_{h,0}^{\pm}\|_{H^s}^2.
\end{equation}
Before dealing with $I_4$, we first observe that a direct computation gives
\[
\sup_{x \in R} |\partial_x l_h f_h (x)| = \sup_{x_m \in h\Z} |(\partial_h^+ f_h)(x_m)|
\]
and
\[\begin{aligned}
\int_{\R} (\sup_{t} |l_h f_h (t,x)|)^2 \; dx =&~{} \sum_{x_m \in h\Z} \int_0^h (\sup_{t} |f_h(x_m) + (\partial_h^+ f_h)(x_m) \cdot x|)^2 \; dx\\
\lesssim&~{} h\sum_{x_m \in h\Z} (\sup_t |f_h(t,x_m)|)^2 + h^3 \sum_{x_m \in h\Z} (\sup_t |\partial_h^+ f_h)(t,x_m)),
\end{aligned}\]
which implies that
\[\norm{\partial_x l_h f_h}_{L_t^4L_x^{\infty}} = \norm{\partial_h^+f_h}_{L_t^4L_x^{\infty}}\]
and
\[\norm{l_h f_h}_{L_x^2L_t^{\infty}} \lesssim \norm{f_h}_{L_x^2L_t^{\infty}} + h \norm{\partial_h^+ f_h}_{L_x^2L_t^{\infty}},\]
respectively. With these observations, by the H\"older inequality and Corollary \ref{cor:Strichartz} and \ref{cor:linear estimates for KdV flows}, we obtain
$$\begin{aligned}
I_4 \leq&~{} T^{\frac{3}{4}}\Big(\|\partial_xl_h v_h^\pm\|_{L_t^4L_x^\infty}+\|\partial_xw_h^\pm\|_{L_t^4L_x^\infty}\Big)\|l_h v_h^\pm -  w_h^\pm\|_{L_t^{\infty}L_x^2}\\
&+T^{\frac{1}{2}}\Big(\|l_h v_h^\pm\|_{L_x^2L_t^\infty}+\|w_h^\pm\|_{L_x^2L_t^\infty}\Big) \big\|\partial_x( l_h v_h^\pm - w_h^\pm)\big\|_{L_x^\infty L_t^2}\\
\lesssim&~{} T^{\frac{1}{2}}\Big(\|v_h^\pm\|_{X_\pm^{s,b}} + h\|\partial_h^+v_h^\pm\|_{X_\pm^{s,b}} +\|w_h^\pm\|_{X_\pm^{s,b}}\Big)\|l_h v_h - w_h^\pm \|_{S([-T,T])}.
\end{aligned}$$
Owing to Theorem \ref{WP for KdVs} and Proposition \ref{Prop:LWP} in addition to Lemma \ref{Lem:discretization linearization inequality},  by choosing sufficiently small $0< T \ll 1$, we have
\begin{equation}\label{eq:I_4}
I_4 \le \frac12\|l_h v_h - w^\pm \|_{S([-T,T])}.
\end{equation}
Finally, going back to \eqref{eq:kdvlimit}, we employ \eqref{eq:I_2}, \eqref{eq:I_3} and \eqref{eq:I_4}, as well as Lemma \ref{Lem:discretization linearization inequality} to complete the proof.
\end{proof}

\subsection{Proof of continuum limit}
Propositions \ref{convergence to decoupled FPU} and \ref{convergence to KdV} do not immediately guarantee Theorem \ref{main theorem} owing to a  lack of commutativity between the linear interpolation and translation operators in the following sense
\[\ell_he^{\pm \frac{t}{h^2}\partial_h} u_h^{\pm} \neq e^{\pm \frac{t}{h^2}\partial_x}\ell_h u_h^{\pm}.\]
However, for every $t = h^3k \in [-T,T]$, $k \in \Z$, 
\[\ell_h(e^{\mp hk\partial_h} u_h^{\pm}) = e^{\mp hk\partial_x}\ell_h(u_h^{\pm})\]
in $L^2$ holds, owing to Lemma \ref{p_h symbol}; thus, Propositions \ref{convergence to decoupled FPU} and \ref{convergence to KdV} ensure Theorem \ref{main theorem}. To complete the proof of Theorem \ref{main theorem}, it is necessary to extend our result at $t = h^3k$ for all $t \in [-T,T]$. For given $t \in [-T,T]$, there exists $k \in \Z$ such that $t \in [h^3k, h^3(k+1))$. Since
\[\begin{aligned}
&~{}\big\|(\ell_h\tilde{r}_h)(t,x)- w_h^+(t, x-\tfrac{t}{h^2})- w_h^-(t, x+\tfrac{t}{h^2})\big\|_{L_x^2(\mathbb{R})}\\
\le&~{} \big\|(\ell_h\tilde{r}_h^+)(t,x)- w_h^+(t, x-\tfrac{t}{h^2})\big\|_{L_x^2(\mathbb{R})} + \big\|(\ell_h\tilde{r}_h^-)(t,x)-w_h^-(t, x+\tfrac{t}{h^2})\big\|_{L_x^2(\mathbb{R})},
\end{aligned}\]
we only deal with the ``$+$" term, since the other part follows analogously. A straightforward computation yields
\[\begin{aligned}
\big\|(\ell_h\tilde{r}_h^+)(t,x)- w_h^+(t, x-\tfrac{t}{h^2})\big\|_{L_x^2(\mathbb{R})} =&~{} \big\|(\ell_he^{-\frac{t}{h^2}\partial_h}u_h^+)(t,x)- e^{-\frac{t}{h^2}\partial_x}w_h^+(t,x)\big\|_{L_x^2(\mathbb{R})} \\
\le&~{}\big\|(\ell_he^{-\frac{t}{h^2}\partial_h}u_h^+)(t,x)-(\ell_he^{-hk\partial_h}u_h^+)(h^3k,x)\big\|_{L_x^2(\mathbb{R})}\\
&+ \big\|(\ell_he^{-hk\partial_h}u_h^+)(h^3k,x)- e^{-hk\partial_x}w_h^+(h^3k,x)\big\|_{L_x^2(\mathbb{R})}\\
&+ \big\|e^{-hk\partial_x}w_h^+(h^3k,x) - e^{-\frac{t}{h^2}\partial_x}w_h^+(t,x)\big\|_{L_x^2(\mathbb{R})}\\
=:&~{} I + II + III.
\end{aligned}\]
Propositions \ref{convergence to decoupled FPU} and \ref{convergence to KdV} show that $II \lesssim h^{\frac{2s}{5}}$. 

\medskip

For $I$, we further split it by
\[\big\|e^{-\frac{t}{h^2}\partial_h}u_h^+(t)-e^{-hk\partial_h}u_h^+(t)\big\|_{L_x^2(h\Z)} + \big\|e^{-hk\partial_h}u_h^+(t)-e^{-hk\partial_h}u_h^+(h^3k)\big\|_{L_x^2(h\Z)} =: I_1+I_2.\] 
Here, we use the boundedness of the linear interpolation operator. Note that 
\[\left|e^{-i\frac{t}{h^2}\xi} - e^{-ihk \xi}\right| \lesssim \frac{1}{h^2}|t-h^3k||\xi| \lesssim h|\xi|, \quad t \in [h^3k, h^3(k+1)).\]
Applying Plancherel theorem, the continuity of $e^{i\theta}$ and Lemma \ref{Lem:Sobolev bound depending on h} to $I_1$, we obtain
\[I_1 \lesssim h\|u_h^+(t)\|_{\dot{H}_x^1(h\Z)} \lesssim h^{s}\|u_h^+(t)\|_{H_x^s(h\Z)},\]
for $\frac34 < s \le 1$. For $I_2$, since
\[\begin{aligned}
u_h^+(t)=&~{} S_h^+(t-h^3k)u_h^+(h^3k) \\
&+ \frac14 \int_{h^3k}^t S_h^+(t-t_1) \nabla_h \bigg[\Big\{u_h^+(t_1)+e^{\frac{2t_1}{h^2}\partial_h}u_h^-(t_1)\Big\}^2 + h^2e^{\frac{t_1}{h^2}\partial_h}\mathcal{R}(t_1)\bigg]dt_1,
\end{aligned}\]
and it suffices to deal with 
\begin{equation}\label{eq:linear diff}
\|S_h^+(t-h^3k)u_h^+(h^3k) - u_h^+(h^3k)\|_{L_x^2(h\Z)}
\end{equation}
and the nonlinear terms. Mixed and $u_h^-$ terms and the higher-order term in the nonlinear part can be controlled by at least $h^s$, owing to Lemmas \ref{lem:bi2}, \ref{lem:bi3} and \ref{lem:AppB}. Meanwile, the $u_h^+$ quadratic term can be roughly estimated by
\[|t-h^3k|h^{-1} \|u_h^+(t)\|_{L_t^{\infty}L_x^4}^2 \lesssim h^{\frac32}\|u_h^+(t)\|_{L_t^{\infty}L_x^2}^2,\]
which itself is sufficient. Moreover, \eqref{eq:linear diff} is bounded by $h^s\|u_h^+(h^3k)\|_{H_x^s(h\Z)}$ analogously to $I_1$, owing to
\[|S_h^+(t-h^3k) -1| = |e^{-i\frac{t}{h^2}(\xi - \frac{2}{h}\sin(\frac{h\xi}{2}))}| \lesssim h\left(|\xi| + \left|\frac{2}{h}\sin\left(\frac{h\xi}{2}\right)\right|\right), \quad t \in [h^3k, h^3(k+1))\]
and Lemma \ref{Prop:norm equivalence}. An analogous argument is available for the estimate of $III$.

\appendix

\section{Failure of the linear estimate}\label{app:FLE}
%
In subsection \ref{sec:AFA}, we measured the size of linear interpolation of the FPU flows in $C_t([-T,T]:H_x^s(h\Z))$ in order to approximate the FPU flows by the Airy flows.
More precisely, the crucial estimates were that  for $0\le s\le1$,
\begin{align*}
\| l_h S_h^{\pm}(t) f_h \|_{C_t([0,T]:H_x^s(h\Z))} 
\ls \| l_h f_h \|_{H^s(h\Z)} \ls \| f_h\|_{H^s(h\Z)}.
\end{align*}
However, such uniform estimates fail if we consider instead the $X^{s,b}$ spaces associated to KdVs \eqref{KdV2} as approximation spaces, which means that even though FPU and KdVs are shown to be well-posed in $L^2$ via $X^{s,b}$, justification of approximation from FPU to KdVs via $X^{s,b}$ is nontrivial.



\begin{proposition}\label{Xsb bad}
Let $0\le s\le 1$ and $b>0$. Then, 
\begin{equation}\label{Yestimates} 
\sup_{h> 0, \ f_h\in H^s(h\Z) } \frac{\| \theta(t) l_h S_h^{\pm}(t) f_h \|_{X_{\pm}^{s,b}}}{\|f_h\|_{H^s(h\Z)}} = \infty,
\end{equation}
where $X_{\pm}^{s,b}$ is defined as in \eqref{eq:XsbKdV}.

\end{proposition}   
\begin{proof}
We claim that there exist a constant $C_b > 0$ independent of $h>0$ such that
\begin{align}\label{lowbound}
\| \theta(t) l_h S_h^{\pm}(t) f_h \|_{X_{\pm}^{s,b}}^2 
  \gs&~{}C_b\frac{1}{h^{6b}}\|f_h\|_{H^s}^2. 
\end{align}
for $f_h \in H_h^s$ satisfying $\mbox{supp} \mathcal F_h(f_h) \subset \{ \xi \in \T_h : |\xi|\ge\frac{\pi}{2h} \}$. Then \eqref{Yestimates} immediately follows.
We prove only \eqref{lowbound} for the $+$ case, since the other case can be treated similarly. 
Using Lemma~\ref{p_h symbol}, we compute
\begin{align*}
&\| \theta(t) l_h S_h^{+}(t) f_h \|_{X_{+}^{s,b}}^2 \\
=&~{}\sum_{m\in\Z}\int_{\gamma_{m,h}+[-\frac{\pi}{h},\frac{\pi}{h})}\left\| \la \xi \ra^s \left\la \tau - \frac{\xi^3}{24} \right\ra^b \mathcal{F}_{t,x} \big(\theta(t) l_h S_h^{+}(t) f_h\big)(\tau,\xi) \right\|_{L_{\tau}^2}^2 d\xi \\
=&~{}\sum_{m\in\Z} \int_{-\frac{\pi}{h}}^{\frac{\pi}{h}}
\left\| \widehat \theta( \tau)\left\la \tau - \frac{1}{24}(\xi +\gamma_{m,h} )^3 + s_h^{+}(\xi) \right\ra^b  \right\|_{L_{\tau}^2}^2 
 \la \xi+\gamma_{m,h} \ra^{2s}  \mathcal{L}_h(\xi+\gamma_{m,h})^2 |\mathcal F_h(f_h)(\xi)|^2 d\xi
\end{align*}
for $\gamma_{m,h} := \frac{2m\pi}{h}$.
First, let us compute the $L_\tau^2$ norm. A direct computation gives 
\begin{align*}
\frac{1}{24}(\xi+\gamma_{m,h})^3 - s_h^{+}(\xi)
&=\frac{\pi^3}{3}\left(\frac{m}{h}\right)^3+\frac{\pi^2 \xi}{6}\left(\frac{m}{h}\right)^2+\frac{\pi\xi^2}{4}\left(\frac{m}{h}\right)+\frac{\xi^3}{24}+\frac{1}{h^2}\left(\xi-\frac2h\sin\left(\frac{h\xi}{2}\right)\right)
\end{align*}
and it is easy to verify that
\[\left|\frac{\pi^2 \xi}{6}\left(\frac{m}{h}\right)^2+\frac{\pi\xi^2}{4}\left(\frac{m}{h}\right)+\frac{\xi^3}{24}+\frac{1}{h}\left(\xi-\frac2h\sin\left(\frac{h\xi}{2}\right)\right)\right|\lesssim \frac{m^2}{h^3}, \text{ for all } \xi \in \T_h,\]
which indicates that  $\frac{\pi^3}{3}(\frac{m}{h})^3$ is the dominant part in $\frac{1}{24}(\xi+\gamma_{m,h})^3 - s_h^{+}(\xi)$. In particular, there exists $m_0 \gg 1$, independent of $h$, such that for $m \le - m_0$, 
\begin{align}\label{m}
- \frac{1}{24}(\xi+\gamma_{m,h})^3 + s_h^{+}(\xi) \gtrsim \left(\frac{|m_0|}{h}\right)^3 \gg1,
\text{ for all } \xi \in \T_h.
\end{align}
%
%
%
Using the above-mentioned observation, we have for $m\le m_0$
\begin{align*}
\left\| \widehat \theta( \tau)\left\la \tau- \frac{1}{24}(\xi + \gamma_{m,h} )^3 + s_h^+(\xi) \right\ra^b  \right \|_{L_{\tau}^2}^2 &\ge
\int_0^\infty |\widehat \theta( \tau)|^2
\left\la \tau- \frac{1}{24}(\xi + \gamma_{m,h} )^3 + s_h^+(\xi) \right\ra^{2b}  d\tau \\
&\gs \left(\frac{|m_0|}{h}\right)^{6b}\int_0^\infty |\widehat \theta( \tau)|^2
  d\tau \\
&\gs \left(\frac{|m_0|}{h}\right)^{6b},
\end{align*}
which implies that
\begin{align*}
\| \theta(t) l_h S_h^{+}(t) f_h \|_{X_{+}^{s,b}}^2 
\gtrsim&~{} \left(\frac{|m_0|}{h}\right)^{6b}\sum_{m\le -m_0}\int_{-\frac{\pi}{h}}^{\frac{\pi}{h}} 
 \la \xi+\gamma_{m,h} \ra^{2s}  \mathcal{L}_h(\xi+\gamma_{m,h})^2 |\mathcal F_h(f_h)(\xi)|^2 d\xi\\
\gs&~{}  \left(\frac{|m_0|}{h}\right)^{6b}
\int_{-\frac{\pi}{h}}^{\frac{\pi}{h}}
 \la \xi+\gamma_{m_0,h} \ra^{2s}  \left( \frac{4\sin^2\left(\frac{h\xi}{2}\right)}{h^2(\xi+\gamma_{m_0,h})^2} \right)^2  |\mathcal F_h(f_h)(\xi)|^2 d\xi. 
\end{align*}
Since 
\[ |\xi| \lesssim |\xi+\gamma_{m_0,h}| \lesssim |\gamma_{m_0,h}|  \quad \mbox{and} \quad \sin^2\left(\frac{h\xi}{2}\right) \ge \frac12\] 
for all $\xi \in \mbox{supp} \mathcal F_h(f_h)$, we conclude

\begin{align*}
\| \theta(t) l_h S_h^{+}(t) f_h \|_{X_{+}^{s,b}}^2 
 \gs&~{}\left(\frac{|m_0|}{h}\right)^{6b}
\int_{-\frac{\pi}{h}}^{\frac{\pi}{h}}
 \la \xi \ra^{2s} (h\gamma_{m_0,h})^{-4}|\mathcal F_h(f_h)(\xi)|^2 d\xi \\
  \gs&~{}\frac{|m_0|^{6b-4}}{h^{6b}}\|f_h\|_{H_h^s}^2. 
\end{align*}
\end{proof}

\section{Analysis for general nonlinearities}\label{app:GP}
This appendix is devoted to some estimates for the higher-order remainder term introduced in Section \ref{sec: outline} to complete our analysis established in Sections \ref{sec: Well-posedness} and \ref{sec: uniform bound for coupled FPU}. 
For any real number $\rho \in \R$, we write $\rho^+$ if there exists a small $0 < \epsilon \ll 1$ such that $\rho^+ = \rho + \epsilon$. Analogously, we use $\rho^-$. The main estimate dealt with in this section is as follows:
\begin{lemma}\label{lem:AppB}
Let $0 \le s \le 1$ and $0<h \le 1$ be given. Assume that
$$\|u_h^\pm\|_{X_{h,\pm}^{s,\frac{1}{2}^+}}\leq M,$$
for some constant $M > 0$. Then, for $\mathcal R$ as in \eqref{eq:R}, we have
\begin{equation}\label{eq:conclu}
\bigg\|\int_0^t S_h^\pm(t-t_1) \nabla_h e^{\pm\frac{t_1}{h^2}\partial_h}h^2\mathcal{R}(t_1) dt_1\bigg\|_{X_{h,\pm}^{s,\frac{1}{2}^+}}\lesssim h^{\min\{\frac{5}{4}-s, \frac{3}{4}+s\}^{-}}T^{\frac{3}{4}}M^3 \sup_{|r|\le Ch^\frac32M}|V^{(4)}(r)|,
\end{equation}
where the constant $C$ in supremum depends only on $\frac12^+$.
\end{lemma}

\begin{remark}
As seen in the proofs of Propositions \ref{LWP of CS}, $M$ depends on the initial condition. Meanwhile, in the proofs of Propositions \ref{Prop:LWP} and \ref{convergence to decoupled FPU}, $M$ depends not only on the initial condition but also on the local existence time, especially, $T^{0^-}$. However, owing to $T^{\frac34}$, the right-hand side of \eqref{eq:conclu} can be sufficiently small by choosing a suitable time $T$ independent of $h$. 
\end{remark}

\begin{remark}\label{rem:X}
Lemma \ref{lem:AppB} indeed completes the proof of Proposition \ref{Prop:LWP}.
\end{remark}

\begin{remark}\label{rem:L2}
Together with the embedding property (Lemma \ref{lem:properties} (3)), Lemma \ref{lem:AppB} completes the proofs of Propositions \ref{LWP of CS} and \ref{convergence to decoupled FPU}.
\end{remark}

\begin{remark}\label{rem:error}
Lemma \ref{lem:AppB} ensures that the higher-order term in \eqref{coupled FPU'} is indeed the error term as $h \to 0$ in the proof of Proposition \ref{convergence to decoupled FPU}. More precisely, in a strong contrast to the quadratic error terms
\[\int_0^t S_h^\pm(t-t_1) \nabla_h \left(2u_h^\pm(t_1)(e^{\pm\frac{2t_1}{h^2}\partial_h}u_h^\mp(t_1)) + (e^{\pm\frac{2t_1}{h^2}\partial_h}u_h^\mp(t_1))^2 \right) \; dt_1\]
in the proof of Proposition \ref{convergence to decoupled FPU} (see also Lemmas \ref{lem:bi2} and \ref{lem:bi3}), Lemma \ref{lem:AppB} ensures that the higher-order term itself in \eqref{coupled FPU'} can be understood as a strong error term as $h \to 0$ in the sense that the smoothness condition on the data is not necessary. 
\end{remark}

\begin{proof}[Proof of Lemma \ref{lem:AppB}]
By assumption, we consequently have 
$$\|u_h^\pm\|_{C_tH_x^s}\lesssim M \quad \mbox{and} \quad \|\tilde{r}_h\|_{C_tH_x^s}\lesssim M.$$
By \eqref{inhomogeneous estimate}, we estimate the higher-order remainder 
$$\bigg\|\int_0^t S_h^\pm(t-t_1) \nabla_h e^{\pm\frac{t_1}{h^2}\partial_h}h^2\mathcal{R}(t_1) dt_1\bigg\|_{X_{h,\pm}^{s,\frac{1}{2}^+}}\lesssim \big\| \nabla_h e^{\pm\frac{t}{h^2}\partial_h}h^2\mathcal{R}(t)\big\|_{X_{h,\pm}^{s,-(\frac{1}{2}^-)}}.$$
Interpolating the dualization of the Strichartz estimates (Corollary \ref{cor:Strichartz}), i.e.,
$$\|u_h\|_{X_{h,\pm}^{0,-(\frac{1}{2}^+)}}\lesssim \||\nabla_h|^{-(\frac{1}{4}^-)}u_h\|_{L_t^{\frac{4}{3}^-}L_x^{1^+}},$$
with the trivial identity $\|u_h\|_{X_{h,\pm}^{0,0}}=\|u_h\|_{L_t^2L_x^2}$, we have
$$\|u_h\|_{X_{h,\pm}^{0,-(\frac{1}{2}^-)}}\lesssim \||\nabla_h|^{-(\frac{1}{4}^-)}u_h\|_{L_t^{\frac{4}{3}}L_x^{1^+}}.$$
Using this bound and the H\"older inequality, we obtain
$$\begin{aligned}
&~{}\big\| \nabla_h e^{\pm\frac{t}{h^2}\partial_h}h^2\mathcal{R}(t)\big\|_{X_{h,\pm}^{s,-(\frac{1}{2}^-)}} \\
\lesssim&~{} h^{(\frac{5}{4}-s)^-} \| e^{\pm\frac{t}{h^2}\partial_h}\mathcal{R}(t)\|_{L_t^{\frac{4}{3}}L_x^{1+}}\\
\lesssim
&~{} h^{(\frac{5}{4}-s)^{-}}\Big\|\big(e^{\pm\frac{t}{h^2}\partial_h}V^{(4)}(h^2\tilde{r}_h^*)\big)\cdot\big(e^{\pm\frac{t}{h^2}\partial_h}\tilde{r}_h\big)^3\Big\|_{L_t^{\frac{4}{3}}L_x^{1}}\\
\leq&~{} h^{(\frac{5}{4}-s)^{-}}T^{\frac{3}{4}}\big\|\big(e^{\pm\frac{t}{h^2}\partial_h}V^{(4)}(h^2\tilde{r}_h^*)\big)\cdot\big(e^{\pm\frac{t}{h^2}\partial_h}\tilde{r}_h\big)\big\|_{L_t^\infty L_x^2}\|e^{\pm\frac{t}{h^2}\partial_h}\tilde{r}_h\|_{L_t^\infty L_x^4}^2.
\end{aligned}$$
By unitarity (with the algebra in footnote 1), we remove the translation operator as follows:
\[\big\|\big(e^{\pm\frac{t}{h^2}\partial_h}V^{(4)}(h^2\tilde{r}_h^*)\big)\cdot\big(e^{\pm\frac{t}{h^2}\partial_h}\tilde{r}_h\big)\big\|_{L_t^\infty L_x^2} \lesssim \|V^{(4)}(h^2\tilde{r}_h^*)\|_{L_t^\infty L_x^\infty}\|\tilde{r}_h\|_{L_t^\infty L_x^2}.\]
By assumption, we have
$$\|h^2\tilde{r}_h^*\|_{C_tL_x^\infty}\leq \|h^2\tilde{r}_h\|_{C_tL_x^\infty}\leq h^{\frac{3}{2}}\|\tilde{r}_h\|_{C_tL_x^2}\leq h^{\frac{3}{2}}\Big\{\|u_h^+\|_{C_tL_x^2}+\|u_h^-\|_{C_tL_x^2}\Big\}\le Ch^{3/2} M.$$
Hence, it follows that 
$$ \|V^{(4)}(h^2\tilde{r}_h^*)\|_{L_t^\infty L_x^\infty}\leq \sup_{|r|\leq Ch^{3/2} M}|V^{(4)}(r)|<\infty.$$
For $\|e^{\pm\frac{t}{h^2}\partial_h}\tilde{r}_h\|_{L_t^\infty L_x^4}$, if $0\leq s\leq\frac{1}{4}$, then by the Sobolev inequality, unitarity and Lemma \ref{Lem:Sobolev bound depending on h}, 
$$\|e^{\pm\frac{t}{h^2}\partial_h}\tilde{r}_h\|_{L_x^4}\lesssim \|e^{\pm\frac{t}{h^2}\partial_h}\tilde{r}_h\|_{\dot{H}^{\frac14}} \lesssim h^{-\frac{1-4s}{4}}\|\tilde{r}_h\|_{\dot{H}_x^s}\lesssim h^{-\frac{1-4s}{4}}M.$$
Meanwhile, if $\frac{1}{4}<s\leq 1$, then 
$$\|e^{\pm\frac{t}{h^2}\partial_h}\tilde{r}_h\|_{L_x^4}\lesssim \|e^{\pm\frac{t}{h^2}\partial_h}\tilde{r}_h\|_{H_x^s}=\|\tilde{r}_h\|_{H_x^s}\lesssim M.$$
Therefore, by combining all these results, we complete the proof of \eqref{eq:conclu}.
\end{proof}

\bibliographystyle{amsplain} \bibliography{FPU}

\end{document}